\documentclass{article}

\usepackage[utf8]{inputenc} 
\usepackage[T1]{fontenc}    
\usepackage{hyperref}       
\usepackage{url}            
\usepackage{booktabs}       
\usepackage{amsfonts}       
\usepackage{nicefrac}       
\usepackage{microtype}      
\usepackage{lipsum}
\usepackage{amsmath}
\usepackage{graphicx}
\usepackage{amssymb}
\usepackage{doi}
\usepackage{amsthm}
\usepackage{xcolor}
\usepackage{bbm}
\usepackage{geometry}
\title{Strong ill-posedness and non-existence in Sobolev spaces for generalized-SQG}
\author{Diego Córdoba\footnote{Corresponding author. Instituto de Ciencias Matem\'aticas CSIC-UAM-UCM-UC3M, Spain. e-mail: dcg@icmat.es} , \hspace{7mm} José Lucas-Manchón\footnote{Instituto de Ciencias Matem\'aticas CSIC-UAM-UCM-UC3M, Spain. e-mail: jose.lucas@icmat.es} , \hspace{7mm} Luis Martínez-Zoroa\footnote{University of Basel, Switzerland. e-mail: luis.martinezzoroa@unibas.ch}.}

\begin{document}
\newtheorem{thm}{Theorem}
\newtheorem*{thm*}{Theorem}
\newtheorem{lma}{Lemma}
\newtheorem{coro}{Corollary}
\newtheorem{prp}{Proposition}
\newtheorem{dfn}{Definition}
\newtheorem{rmk}{Remark}
\maketitle

\begin{abstract}
The general surface quasi-geostrophic equation is the  scalar transport equation defined by
 \begin{equation*}
 \left\{\begin{array}{l}
 \frac{\partial \theta}{\partial t}+v^\gamma_1 \frac{\partial \theta}{\partial x_1}+v^\gamma_2 \frac{\partial \theta}{\partial x_2} =0 ,\\ \\
 v^\gamma=\nabla^{\perp} \psi_\gamma=\left(\partial_{2} \psi_\gamma,-\partial_{1} \psi_\gamma \right), \quad \psi_\gamma=-\Lambda^{-1+\gamma} \theta , \\ \\
 \theta(\cdot,0)=\theta_0(\cdot),
 \end{array}\right.
 \end{equation*}
 for $\gamma \in (-1,1)$, where the non-local operator $\Lambda^{\alpha}=(-\Delta)^{\frac{\alpha}{2}}$ is defined on the Fourier side by $\widehat{\Lambda^{\alpha} f}(\xi)=|\xi|^{\alpha} \widehat{f}(\xi)$. 
 The PDE is well-posed in the Sobolev spaces $H^s$ with $s>2+\gamma$.\\ \\
 In this paper we prove strong ill-posedness in the super-critical regime $H^\beta$ with $\beta\in [1,2+\gamma)\cap(\frac{3}{2}+\gamma,2+\gamma)$. To do this, we will derive an approximated PDE solvable by some family of functions that we will call pseudosolutions and that will allow us to control the norms of the real solutions.
 \\ \\
 Using this result and a gluing argument we also prove non-existence of solutions in the same Sobolev spaces. Since the pseudosolution will control the real one, we can build a solution that will be initially in $H^{\beta}$ and will leave it instantaneously. Nevertheless, this solution exists for a long time and remains the only classical solution in a high regularity class.
\end{abstract}

\newpage
\tableofcontents
\newpage
 \section{Introduction}
In the study of fluid dynamics, the Euler equations play a fundamental role in describing the behavior of ideal, incompressible fluids. These equations capture the conservation of mass and momentum, providing information about the dynamics of fluid flow. However, the Euler equations are highly complex and challenging to solve analytically or numerically in many cases. As a result, it is interesting to think about simplified models that capture essential features of fluid motion while being more simple.
\\ \\
One such two-dimensional model that has attracted significant attention in recent years is the surface quasi-geostrophic (SQG) equation. The SQG equation provides a simplified representation of the dynamics of large-scale geophysical flows, such as oceanic and atmospheric currents. While they are derived as a simplified version of the Euler equations (see \cite{sqgEulerRelacion} and \cite{majda}), they retain essential characteristics of the flow, making them a valuable tool for studying some geophysical phenomena.\\ \\
SQG involves a velocity described by a singular integral operator, namely the Riesz transform, which is a Calderón-Zygmund type operator. For the generalized SQG (gSQG) a more general velocity operator is considered, with regularity parameterized by $\gamma$.
\\ \\
Let $\gamma \in [-1,1]$. We define the $\gamma$-gSQG equations as
  \begin{equation}
  \label{gSQG}
 \left\{\begin{array}{l}
 \frac{\partial \theta}{\partial t}+v^\gamma_1 \frac{\partial \theta}{\partial x_1}+v^\gamma_2 \frac{\partial \theta}{\partial x_2} =0 ,\\ \\
 v^\gamma=\nabla^{\perp} \psi_\gamma=\left(\partial_{2} \psi_\gamma,-\partial_{1} \psi_\gamma \right), \quad \psi_\gamma=-\Lambda^{-1+\gamma} \theta , \\ \\
 \theta(\cdot,0)=\theta_0(\cdot),
 \end{array}\right.
 \end{equation}
 where $\theta=\theta(x,t)$ is the scalar unknown, $v^\gamma=v^\gamma(x,t)$ is the velocity of propagation of the magnitude $\theta$ and $\psi_\gamma=\psi_\gamma(x,t)$ is the stream function. Notice that $v^\gamma$ is divergence free, since it is the orthogonal gradient of some function.
 \\ \\
The non-local operator $\Lambda^{\alpha}=(-\Delta)^{\frac{\alpha}{2}}$ is defined on the Fourier side by $\widehat{\Lambda^{\alpha} f}(\xi)=|\xi|^{\alpha} \widehat{f}(\xi)$.
\\ \\
More precisely, for $\gamma \in [-1,1)$, the velocity $v^\gamma$ can be also expressed as
\begin{equation}
    \label{velocidad}
    v^\gamma(x,t):= C(\gamma) P.V.\int_{\mathbb{R}^2} \frac{(x-y)^\perp}{|x-y|^{3+\gamma}} \theta (y,t) dy,
\end{equation}
where $(a,b)^\perp:=(-b,a)$ and 
\begin{align*}
    C(\gamma)=\frac{\Gamma \left(\frac{1+\gamma}{2}\right)}{2^{1-\gamma}\pi\Gamma \left(\frac{1-\gamma}{2}\right)}.
\end{align*}
We define
\begin{align}
\label{kernel}
    K^\gamma(x)=\frac{x^{\perp}}{|x|^{3+\gamma}},
\end{align}
which is the kernel used in (\ref{velocidad}).\\\\
As we may see, in terms of Sobolev norms the velocity will be more singular than the scalar itself for positive gammas ($\gamma>0$) and more regular than the scalar quantity for negative ones ($\gamma<0$). This can be seen on the Fourier side. \\ \\
Some particular values of $\gamma$ are specially interesting, namely:
\begin{itemize}
    \item For $\gamma=-1$, we have 2D Euler equation in the vorticity-stream formulation, where $\theta$ plays the role of the scalar vorticity (see \cite{majda}, Section 2.1).
    \item For $\gamma = 0$, we recover SQG.
    \item For $\gamma=1$, gSQG is the steady-state trivial equation.
\end{itemize}
\vspace{2mm}
Let us summarize the state of the art related to the well/ill-posedness and the local or global existence of generalized SQG and some related family of equations in Sobolev and Hölder spaces.
\begin{itemize}
    \item \textbf{Local well-posedness in Sobolev spaces:} local existence of solutions for $\gamma$-gSQG with $\gamma\in [-1,1)$ holds in $H^s$ for $s>2+\gamma$ (see \cite{existence}). Formation of singularities at positive finite times for initial data in $H^s$ with $s>2+\gamma$ remains an open problem for the range $\gamma \in (-1,1)$. Nevertheless, there are a few rigorous constructions of non-trivial global solutions in $H^s$ (for some $s$ satisfying $s > 2 + \gamma$) in \cite{global1}, \cite{global2}, \cite{global3}, \cite{global4}. For a logarithmic regularization of this equation, local existence holds in borderline spaces $H^{2+\gamma}$ (see \cite{Logarithmic} and \cite{Logarithmic2}). Moreover, regularizing logarithmically up to a high enough power, global existence of modified 2D Euler holds in the critical Sobolev space $H^1$ (see \cite{Logarithmic3}). 
\item \textbf{In terms of $H^s$-norm growth}, in \cite{crecimientoHs} the authors prove that there exist solutions for SQG with initial conditions arbitrarily small that become large after a long period of time. 
\item Attending now to bad behaviours \textbf{in Sobolev spaces, strong ill-posedness} for 2D Euler ($\gamma=-1$) in the critical space $H^1$ is proved in \cite{illposednessEuler}. Moreover, for some low logarithmic regularizations of 2D Euler, strong ill-posedness in $H^1$ is proved in \cite{Logarithmic4}. In \cite{gap} the authors prove an instantaneous gap loss of Sobolev regularity of the  velocity for the 2D incompressible Euler equations in the super-critical regime. For SQG ($\gamma=0$) we can see in \cite{paper} that the equation is strongly ill-posed in the super-critical and critical Sobolev spaces $H^s$ for $s\in (\frac{3}{2},2]$ and the authors prove also non-existence of solutions for $s\in (\frac{3}{2},2)$. The technique used in \cite{paper} is similar to the one we use here and proves the existence of a solution that will leave immediately the space $H^s$ for positive times but will still be the unique solution in $H^{\frac{3}{2}+\gamma}$ for times as long as we want. In \cite{illposednessContinuo} the authors improve the non-existence result constructing it in a compact support set and proving a continuous loss of regularity. Strong ill-posedness of SQG was also proved in the torus in \cite{sobolevcriticosqg}. Finally in \cite{tabla} some strong ill-posedness results are proved for several transport equations. In particular, strong ill-posedness is proved for gSQG with $\gamma>1$.
\item \textbf{Global existence of weak solutions in $L^2$} for SQG has been obtained in \cite{soldebiles} and extended to $\gamma\in(0,1)$ in \cite{existence}. The uniqueness of weak solutions is not known and, in fact, in \cite{nounicidad} a nonuniquness result was proved for solutions of SQG with regularity $\Lambda^{-1} \theta \in C^\sigma_t C^\beta_x$ (where $\frac{1}{2}<\beta <\frac{4}{5}$ and $\sigma <\frac{\beta }{2-\beta }$). Finally, in \cite{nounicidadL2}, \cite{nounicidad1}, \cite{nounicidad2}, \cite{nounicidad3}, \cite{nounicidad4}, \cite{Onsager}, \cite{nounicidad5}, \cite{nounicidad6}, \cite{nounicidad7} the reader can check numerous results concerning the non-uniqueness of weak solutions and the conservation of Hamiltonian systems for weak solutions for the SQG equation.
\item \textbf{Local well-posedness in Hölder spaces:} local existence of solutions in $C^{k,\beta}\cap L^q$ ($k\geq 1$, $\beta\in(0,1)$, $q>1$) was established for SQG in \cite{existencialocalholder} and improved in \cite{existencialocalholder2} dropping the hypothesis $\theta\in L^q$. For $\gamma\in [-1,0)$ the result does hold for $\beta \in [0,1]$ (see the estimates in \cite{noexistenciaholdergsqg}) and it has been improved in \cite{gsqgwellillholder}, establishing well-posedness in $C^{0,\beta}$ for $\beta>1+\gamma$. 
\item \textbf{Strong ill-posedness in Hölder spaces:} despite of the local existence result in $C^{k,\beta}$ ($k\geq 1$, $\beta\in(0,1)$, $q>1$) for SQG  (\cite{existencialocalholder} and \cite{existencialocalholder2}), it is shown in \cite{paper} that the result is not valid for $\beta=0$ since the authors prove strong ill-posedness and non-existence of solutions. In \cite{gsqgwellillholder} the authors prove strong ill-posedness for gSQG with $\gamma\in [-1,0)$ in the critical Hölder space $C^{0,1+\gamma}$. When the velocity is more singular than the active scalar function (i.e. $\gamma \in (0,1)$), two of the authors of this paper establish strong ill-posedness in $C^{k,\beta}$ ($k\geq 1, \beta\in(0,1]$ and $k+\beta>1+\gamma$) in \cite{diegoluis}. For the particular case of 2D Euler ($\gamma=-1$) see \cite{noexistenciaholdereuler} and \cite{bourgainli} where it is proved an ill-posedness and nonexistence result for velocities $v\in C^k$. Finally, in \cite{crecimientoSQG} it is proved an exponential lower bound on time of the $C^2$ norm of the solution, which proves that either one of the second derivatives of the solution grows infinite in time or that blows up in finite time.
\end{itemize}
Finally we prove strong ill-posedness and non-existence of solutions for gSQG when $\gamma \in (-1,1)$ in the Sobolev spaces $H^s$ with $s \in [1,2+\gamma)\cap (\frac{3}{2}+\gamma,2+\gamma)$. In fact the result for $\gamma=0$ is already done in \cite{paper} and we extend it to the rest of the  interval.\\\\
Notice that ill-posedness can be presented in several forms: non uniqueness of solutions in the  considered space, lack of continuity with respect of the initial conditions or norm inflation (strong ill-posedness). In the spaces that we consider in this paper the solutions are unique. We prove strong ill-posedness and, using this, we also demonstrate non-existence of solutions. Furthermore, although this is not the focus of the paper, one can find solutions that are initially arbitrarily close in $H^{\beta}$ to our solution (that leaves $H^{\beta}$ instantly) that remain bounded in $H^{\beta}$ for some long time, obtaining, in particular, lack of continuity with respect to the initial conditions around our solutions.

 \subsection{Notation}
 \begin{itemize}
    \item We denote the ball of radio $r>0$ and center $(x_0,y_0)\in \mathbb{R}^2$ as $$B_r(x_0,y_0):=\{(x,y)\in\mathbb{R}^2: (x-x_0)^2+(y-y_0)^2<r^2\}.$$  
    \item Whenever we do not specify the space where a function is defined, we mean $\mathbb{R}^2$. For example, $H^s:=H^s(\mathbb{R}^2)$
    \item If $s\in \mathbb{N}$ and $\Omega \subset \mathbb{R}^2$, we say that $f\in H^s(\Omega)$ with norm $ \|  f  \| _{H^s(\Omega)}$ if
    \[
     \|  f  \| _{H^s(\Omega)}:=\sum_{i=0}^s\sum_{j=0}^i  \|  \frac{\partial^i f}{\partial^{j}x_1 \partial^{i-j}x_2}  \| _{L^2(\Omega)} <\infty.
    \]
    The homogeneous norm comes defined by 
    \begin{align*}
        \|  f  \| _{\Dot{H}^s(\Omega)}:=\sum_{j=0}^s  \|  \frac{\partial^s f}{\partial^{j}x_1 \partial^{s-j}x_2}  \| _{L^2(\Omega)}.
    \end{align*}
     \item For $s>0$ a non-integer, we do the analogous definition with the norm
    \[
     \|  f  \| _{H^s(\Omega)}:= \|f\|_{L^2}+\|f\|_{\Dot{H}^s(\Omega)}, \hspace{4mm} \|f\|_{\Dot{H}^s(\Omega)}:= \|\mathcal{F}^{-1}(|\xi|^s\mathcal{F}f)\|_{L^2(\Omega)},
    \]
    where $\mathcal{F}$ is the Fourier transform. Notice that this is not the standard definition of the fractional Sobolev norm in domains. Nevertheless, we will use this one for technical reasons. In the case $\Omega=\mathbb{R}^2$, we recover a norm equivalent to the standard one.
        \item For $\bold{j}=(j_1,j_2) \in \mathbb{N}_0^2$, let us denote $\partial^\bold{j} f=\frac{\partial^{|\bold{j}|}f}{\partial^{j_1}x_1\partial^{j_2}x_2}$ with $|\bold{j}|=j_1+j_2$.
      \item We will use the symbol $\Lambda^\alpha$ to denote the fractional laplacian, more specifically $\Lambda^\alpha=(-\Delta)^\frac{\alpha}{2}$. Moreover, we will also use the symbols $D^{s,\bold{j}}$ with $s\geq0$, $\bold{j}\in \mathbb{N}_0^2$ to denote differential operators defined on the Fourier side as
    \begin{equation*}
                \widehat{D^{s,\bold{j}}f} ( \xi )= i^{-|\bold{j}|} \partial_{\xi}^{\bold{j}} (|\xi|^s)\hat{f}(\xi).
    \end{equation*}
    Notice that $D^s:=D^{s,0}$ coincides with $\Lambda^s$.
    \item Eventually, we will use polar coordinates, denoting with $\alpha$ the angular coordinate and with $r$ the radial one. Sometimes we will abuse of the notation writing $f(r,\alpha)$ for $f$ a function defined in cartesian coordinates. What we really mean by $f(r,\alpha)$ is $f(x_1(r,\alpha),x_2(r,\alpha))$.
    \item We will denote the characteristic function in $A\subset\mathbb{R}^2$ with the symbol 
    \begin{align*}
        \mathbbm{1}_A(x)=
         \left\{ \begin{array}{lcc} 
    1 & if & x\in A ,
    \\ \\ 0 & if & x\notin A .
    \end{array} \right.
    \end{align*}
    \item When we use the notation $P.V.$ we mean that we take limits when $\varepsilon \rightarrow 0^+$ of an integral over the hole space without the ball of radio $\varepsilon>0$ centered on a singularity, independently of whether we are in cartesian or polar coordinates. This is, if $K$ is a kernel with a singularity at the origin,
    \begin{align*}
        P.V. \int_{\mathbb{R}^2}K(x-y)f(y)dy:=\lim_{\varepsilon\rightarrow 0^+} \int_{|x-y|>\varepsilon}K(x-y)f(y)dy,
    \end{align*}
    and hence in polar coordinates, 
    \begin{align*}
        &P.V. \int_{0}^\infty\int_{-\pi}^\pi K(r\cos(\alpha)-r'\cos(\alpha'),r\sin(\alpha)-r'\sin(\alpha'))f(r'\cos(\alpha'),r'\sin(\alpha'))d\alpha' r' dr'
        \\
        &:=\lim_{\varepsilon\rightarrow 0^+}\int \int_{A_{\varepsilon,r,\alpha}} 
        K(r\cos(\alpha)-r'\cos(\alpha'),r\sin(\alpha)-r'\sin(\alpha'))f(r'\cos(\alpha'),r'\sin(\alpha'))d\alpha' r'dr',
    \end{align*}
    where $x=(r\cos(\alpha),r\sin(\alpha))$, $y=(r'\cos(\alpha'),r'\sin(\alpha'))$ and 
    \begin{align*}
        A_{\varepsilon,r,\alpha}=\{(r',\alpha')\in \mathbb{R}_+\times[-\pi,\pi]:|(r-r')^2+2rr'(1-\cos(\alpha-\alpha'))|>\varepsilon^2\}.
    \end{align*}
    \item We will use $\lfloor \cdot  \rfloor$ to denote the floor function. 
    \item For $A,B\subset \mathbb{R}^2$, we will define $A+B:=\{x+y\in \mathbb{R}^2:x\in A, y \in B\}$.
\end{itemize}
\subsection{Main results}
\begin{thm}[Strong ill-posedness of gSQG in supercritical $H^\beta$ spaces]
\label{teorema}
For any $T>0$, $0<c_0<1$, $M>1$, $\gamma \in (-1,1)$, $\beta \in [1,2+\gamma)\cap (\frac{3}{2}+\gamma,2+\gamma)$ and $t_\ast>0$ as small as we want, we can find a $H^{\beta+\frac{1}{2}}$ function $\theta_0$ with $  \|  \theta_0  \|  _{H^{\beta}}\leq c_0$ such that the unique solution $\theta (x,t)$ in $H^{\beta+\frac{1}{2}}$ for $t\in[0,T]$ to the $\gamma$-gSQG equation (\ref{gSQG}) with initial conditions $\theta_0$ is such that $  \|  \theta(x,t_\ast)  \|  _{H^\beta}\geq Mc_0$. Moreover, along its time of existence the solution is $C^\infty_c$ and $supp(\theta)\subset supp(\theta_0)+B_{2c_0T}(0)$.
\end{thm}
\vspace{3mm}
\begin{thm}[Non-existence of gSQG in supercritical $H^\beta$ spaces.]
 \label{nonexistence}
 For any $t_0>0$, $0<c_0<1$, $\gamma \in (-1,1)$ and $\beta \in [1,2+\gamma)\cap (\frac{3}{2}+\gamma,2+\gamma)$, we can find initial conditions $\theta_0$ with $\|\theta_{0}\|_{H^\beta}\leq c_0$ such that there exist a solution $\theta$ to $\gamma$-gSQG (\ref{gSQG}) with $\theta(x,0)=\theta_0(x)$ satisfying $\|\theta(x,t)\|_{H^\beta}=\infty$ for all $t\in (0,t_0]$. Furthermore, it is the only solution with initial conditions $\theta_0$ that satisfies $\theta \in L^\infty([0,t_0],H^{\frac{3}{2}+\gamma})\cap C([0,t_0],C^2_x(K))$ for any compact set $K\subset\mathbb{R}^2$.
 \end{thm}
\vspace{3mm}
\begin{rmk}
    In our construction, $\theta\in L^\infty([0,T],H^{\frac{3}{2}+\gamma})$. This is, we prove that $\theta$ will instantaneously leave $H^\beta$ and also assure that the solution will still be in $H^{\frac{3}{2}+\gamma}$ for every time $[0,T]$. The fact that the solution is constructed ``by hand'' is the key to allow us to control the $H^s$ norm of it for $0\leq s\leq \beta+\frac{1}{2}$.
\end{rmk}
\subsection{Sketch of the proof}
The main idea of the proof is to construct good approximations of a certain family of solutions that are perturbations of a radial one. Let us consider a radial function $f_1(r)$. Notice that it will be a solution of the equation (\ref{gSQG}) for all $\gamma\in [-1,1]$ (every radial function is a solution). Let us perturb it, this is, let us consider $\theta(r,\alpha,t)=f_1(r)+\theta_p(r,\alpha,t)$. Then the perturbation $\theta_p$ will satisfy the PDE
\begin{align*}
    \partial_t \theta_p +\frac{v_\alpha^\gamma(f_1)(r)}{r}\partial_\alpha\theta_p+v_r^\gamma(\theta_p)\partial_rf_1+v^\gamma(\theta_p)\cdot\nabla\theta_p=0.
\end{align*}
If the quadratic term in $\theta_p$ and the term with $v_r^\gamma(\theta_p)$ were small in comparison to the term including the angular derivative (this is better explained in the Remark \ref{remark2}, once we present the pseudosolutions), it would make sense to consider the approximated PDE
\begin{align*}
     \partial_t \overline{\theta}_p +\frac{v_\alpha^\gamma(f_1)(r)}{r}\partial_\alpha\overline{\theta}_p = 0,
\end{align*}
and we expect the pseudosolution $\overline{\theta}_p$ to be similar to the real solution $\theta_p$ for a long time. This PDE can be solved explicitly with the solutions fulfilling
\begin{align*}
    \overline{\theta}_p(r, \alpha, t)= \overline{\theta}_p \left(r,  \alpha-\frac{v_\alpha^\gamma(f_1)(r)}{r}t,0\right),
\end{align*}
with the initial condition $\overline{\theta}_p(r,\alpha,0)$ being a $2\pi$-periodic function in the angular coordinate.\\ \\
We want the perturbation $\overline{\theta}_p(r, \alpha, t)$ to be small in $C^{1+\gamma}$ (so that the approximation we are making is correct) but we want it to start small in $H^{\beta}$ and then become very big (to achieve ill-posedness). A way to get this is considering a perturbation of the form
\begin{align*}
    \overline{\theta}_{p,N}(r, \alpha, t)= N^{-\beta} g \left(r,  N\alpha-N\frac{v_\alpha^\gamma(f_1)(r)}{r}t\right),
\end{align*}
with $N>0$ big. Then, notice that the leading term (in terms of $N$) when we compute the $k$-th derivative with respect to the radial coordinate is something of the form 
\begin{align*}
    N^{k-\beta}(-1)^k\left(\frac{\partial}{\partial r}\frac{v^\gamma_\alpha(f_1)(r)}{r} \right)^k \partial_2^k g\left(r,  N\alpha-N\frac{v_\alpha^\gamma(f_1)(r)}{r}t\right),
\end{align*}
so by interpolation we may see that the term that will make the  $H^\beta$ norm grow is 
\begin{align*}
    \left|\frac{\partial}{\partial r}\frac{v^\gamma_\alpha(f_1)(r)}{r}\right|^\beta.
\end{align*}
To obtain this growth we just have to find a compact supported $f_1$ such that this quantity is big somewhere in the plane far from the support of $f_1$. Then, we place there a compact supported $f_2$ multiplied by $N^{-\beta}$ and a $2\pi$ periodic function depending on the variable $N \left(\alpha-\frac{v_\alpha^\gamma(f_1)(r)}{r}t \right)$ and this is how we can obtain the inflation of the $H^\beta$ norm.\\\\
This motivates us to study the family of pseudosolutions 
\begin{align*}
    \overline{\theta}_N(r,\alpha,t)=f_1(r)+\overline{\theta}_{p,N}(r,\alpha,t)=f_1(r)+f_2(r)N^{-\beta}\sin
    \left(N\alpha-N\frac{v_\alpha^\gamma(f_1)(r)}{r}t
    \right),
\end{align*}
for $N\in \mathbb{N}$. Given $\beta \in [1,2+\gamma)\cap (\frac{3}{2}+\gamma,2+\gamma)$ and $c,\varepsilon>0$ as small as we want, we prove that there exist $f_1$, $f_2$ compactly supported and some annulus $\mathcal{C}\subset \mathbb{R}^2$ such that,
\begin{itemize} 
    \item $\|f_1\|_{H^\beta}\leq c$, $\|f_2\|_{L^2}= c$, 
    \item $\left|\frac{\partial}{\partial r}\left(\frac{v_\alpha^\gamma(f_1)(r)}{r}\right)\right|\sim \frac{K}{r}$ as big as we want in $\mathcal{C}$,
    \item $ supp(f_2)\subset \mathcal{C}$, $supp(f_1)\cap supp(f_2)=\emptyset$ and $supp(f_1)\cup supp(f_2)\subset B_\varepsilon (0)$.
\end{itemize}
Notice that as a consequence of the two last points, the derivative of the velocity generated by the radial function $f_1$ will act on $\overline{\theta}_{p,N}$, making its $H^\beta$ norm grow. Summarizing, by the previous conditions,
\begin{itemize}
    \item $\|\overline{\theta}_N(\cdot,0)\|_{H^\beta}\lesssim c$ uniformly on $N$ big enough,
    \item $\|\overline{\theta}_N(\cdot,t)\|_{H^\beta}\gtrsim c K^\beta t^\beta$ uniformly on $N$ big enough.
\end{itemize} 
Furthermore, if we are able to prove
\begin{align*}
    \|\theta_N-\overline{\theta}_N\|_{H^\beta}\rightarrow 0,
\end{align*}
then the real solution $\theta_N$ will be very similar to the pseudosolution for big $N$ and we will obtain 
\begin{itemize}
    \item $\|\theta_N(\cdot,0)\|_{H^\beta}\lesssim c$ for $N$ big,
    \item $\|\theta_N(\cdot,t)\|_{H^\beta}\gtrsim c K^\beta t^\beta$ for $N$ big,
\end{itemize}  
so given $c,t_\ast>0$, we can choose $K>0$ big enough to obtain the strong ill-posedness result.\\ \\
We prove the non-existence result via a gluing argument. To do this we consider initial conditions of the form 
\begin{align*}
        \theta(x,0)=\sum_{j=1}^{\infty} \overline{\theta}_{N_j, c_j,K_j} (x-R_j,0),
\end{align*}
where $R_j\rightarrow \infty$ grows fast to minimize the interaction between the different parts of the solution. Using the strong ill-posedness result,
\begin{itemize}
 \item $c_j>0$ small enough such that
    \begin{align*}
        \| \overline{\theta}_{N_j,c_j,K_j} (x,0)\|_{H^\beta} \leq c_0 2^{-j},
    \end{align*}
    \item and $N_j,K_j$ large enough such that
    \begin{align*}
        \|\overline{\theta}_{N_j,c_j,K_j} (x,t)\|_{H^\beta} \geq t c_0 2^{j}, \hspace{3mm} \forall t \in [0,T],
    \end{align*}
    with $T>0$ independent of $j\in \mathbb{N}$.
\end{itemize}
We can then prove that $\theta$ will exist via a limit argument and that it will be similar to $\sum_{j}\overline{\theta}_{N_j,c_j,K_j} $ in the Sobolev norm $H^\beta$. Thus, we have the desired instantaneous blow-up $\|\theta(\cdot,t)\|_{H^\beta}=\infty$ for $t>0$.
\begin{rmk}
    Although we only need $\|\theta_N-\overline{\theta}_N\|_{H^\beta}\rightarrow 0$ for proving norm inflation, we are just able to propagate the error in a space where we have local existence, so all the bounds of the error $\theta_N-\overline{\theta}_N$ will be done in $H^{\beta+\frac{1}{2}}$. This is due to the fact that a bootstrap argument will be used in the Lemma \ref{cotaTheta} and this only works in a space where local existence holds.
\end{rmk}
\begin{rmk}
\label{remark2}
    Notice that, although we are assuming in the approximation of the PDE that the quadratic terms in $\theta_{p,N}$ are small, the growth in the Sobolev norm is given precisely by the perturbation. Nevertheless, it heuristically makes sense since the behaviour of the $H^\beta$ norm of the different parts of the PDE in terms of $N$ comes given by
    \begin{itemize}
        \item $\|\partial_t \theta_{p,N}\|_{H^\beta}\sim N$,
        \item $\|\frac{v_\alpha^\gamma(f_1)(r)}{r}\partial_\alpha\theta_{p,N}\|_{H^\beta}\sim N$,
        \item $\|v_r^\gamma(\theta_{p,N})\partial_rf_1\|_{H^\beta}\sim N^\gamma$ with $\gamma<1$,
        \item $\|v^\gamma(\theta_{p,N})\cdot\nabla\theta_{p,N}\|_{H^\beta}\sim N^{\gamma+1+\beta-2\beta}=N^{\gamma+1-\beta}$ with $\gamma+1-\beta<0$,
    \end{itemize}
    so we are including in our approximated evolution the biggest terms and neglecting the lower order ones.
    \end{rmk}
    \subsubsection{Differences between the proofs of non-existence of SGQ and generalized SQG.}
    Notice that, although the strategy of the proof for SQG and its generalization are similar, a lot of technical issues come up when the operator velocity is not a Riesz transform.\\ \\ We find the main difficulty when one tries to bound the Sobolev norms of the difference between the pseudosolution and the exact solution. More specifically, in Lemma \ref{FnormL2}, Corollary \ref{FnormH} and Lemmas \ref{cotadiferencia} and \ref{cotaTheta}, the reader may check that bounds of the Hölder and Sobolev norms of $v^\gamma(\overline{\theta}_{N})$ are necessary to complete the results. In order to achieve this, it has been necessary to prove Lemma \ref{vradial} and the bounds (\ref{vangular}), (\ref{vradialmedio}) and (\ref{vangularmedio}), that imply the bounds (\ref{normahsoperados}) and (\ref{normackvelocidad}) of the operator velocity applied to the pseudosolution. \\ \\
    Another aspect that needs a treatment apart is the fact that we cannot use  anymore the Leibniz type bound of Lemmas \ref{leibniz} and \ref{leibnizv} when $\gamma<0$. Then, in several bounds of Lemma \ref{cotaTheta} (where we bound the $H^{\beta+\frac{1}{2}}$ norm of the difference between the pseudosolution and the solution), as in (\ref{gammanegativo1}), (\ref{gammanegativo2}) or (\ref{gammanegativo3}), we must overcome this using different tools.
 \subsection{Outline of the paper}
 In subsection \ref{lemastecnicos}, we first prove some technical lemmas in which we show that it is possible to find the radial functions $f_1,f_2$ with the desired conditions, as well as some bounds for the velocity. In subsection \ref{seccionpseudosol} we prove that the pseudo-solutions will be similar in $H^s$ to the real solutions. Using all of the above, in subsection \ref{prueba} we prove strong ill-posedness. Finally in Section \ref{seccion3} we use the strong ill-posedness result to prove non-existence of solutions via a gluing argument. 
 \section{Strong ill-posedness}
 \label{seccion2}
During this whole chapter we will work to prove strong ill-posedness of the generalized SQG (gSQG) in $H^\beta$ for $\beta \in [1,2+\gamma)\cap (\frac{3}{2}+\gamma,2+\gamma)$.
\subsection{Technical lemmas}
\label{lemastecnicos}
 We remember the expression of the velocity (already defined in (\ref{velocidad}))
\begin{align*}
     v^\gamma(\theta)(x,t):= C(\gamma) P.V.\int_{\mathbb{R}^2} \frac{(x-y)^\perp}{|x-y|^{3+\gamma}} \theta (y,t) dy,
\end{align*}
where $\gamma\in [-1,1)$. We consider the angular component of it, 
\begin{equation*}
v_\alpha^\gamma (\theta)(x,t)=C(\gamma) P.V.\int_{\mathbb{R}^2} \hat{x}^\perp \cdot\frac{(x-y)^\perp}{|x-y|^{3+\gamma}} \theta (y,t) dy,
\end{equation*}
 that, in the case of a function of the form $g_N(r,\alpha)=\tilde{g}_N(r)\sin{(N\alpha-N\alpha_0(r))}$ with $\tilde{g}_N$ and $\alpha_0$ regular enough, can be expressed in polar coordinates as follows
 \begin{align*}
&v_\alpha^\gamma (g_N)(r,\alpha)=\\
&= C(\gamma) P.V. \int_{\mathbb{R}_+}\int_{-\pi}^\pi r' \frac{\sin(\alpha)(r\sin(\alpha)-r'\sin(\alpha'))+\cos(\alpha)(r\cos(\alpha)-r'\cos(\alpha'))}{|(r\cos(\alpha)-r'\cos(\alpha'))^2+(r\sin(\alpha)-r'\sin(\alpha'))^2|^{\frac{3+\gamma}{2}}}g_N(r',\alpha')dr'd\alpha'\\
&= C(\gamma) P.V. \int_{\mathbb{R}_+ \times [-\pi,\pi]} r' \frac{r-r'(\sin(\alpha)\sin(\alpha')+\cos(\alpha)\cos(\alpha'))}{|(r-r')^2+2rr'(1-\cos(\alpha-\alpha'))|^{\frac{3+\gamma}{2}}}g_N(r',\alpha')dr'd\alpha'\\
&= C(\gamma) P.V. \int_{-r}^\infty \int_{-\pi}^\pi (r+h) \frac{r-(r+h)\cos(\alpha-\alpha')}{|h^2+2r(r+h)(1-\cos(\alpha-\alpha'))|^{\frac{3+\gamma}{2}}}g_N(r+h,\alpha')dh d\alpha'\\
&= C(\gamma) P.V. \int_{-r}^\infty \int_{-\pi}^\pi  
\left(
\frac{[(r+h)(1-\cos(\alpha-\alpha'))-h]}{|h^2+2r(r+h)(1-\cos(\alpha-\alpha'))|^{\frac{3+\gamma}{2}}}
\tilde{g}_N(r+h)\mathcal{S}_N(r+h,\alpha')
\right) dh d\alpha'
\\
&= C(\gamma) P.V. \int_{-r}^\infty \tilde{g}_N(r+h) \mathcal{S}_N(r+h,\alpha)
\left(
\int_{-\pi}^\pi \frac{[(r+h)(1-\cos(\alpha'))-h]\cos(N\alpha') }{|h^2+2r(r+h)(1-\cos(\alpha'))|^{\frac{3+\gamma}{2}}}d\alpha'
\right)dh, 
\end{align*}
where $\mathcal{S}_N(\rho,\phi)=\rho\sin(N\phi-N\alpha_0(\rho))$.
\\ \\In the last equality we summed and subtracted $N\alpha$ in the angular component of $\mathcal{S}_N$, we used the trigonometric formula for the sinus of the sum and finally we did the change of variables $\alpha'-\alpha\rightarrow \alpha'$, noticing that the term involving $\sin (N\alpha')$ vanishes when we integrate due to imparity. \\\\
Notice that, if we had a radial function $\theta=\tilde{g}_N(r)$ it would be even simpler. Indeed,
\begin{align*}
    v_\alpha^\gamma (\theta)(x,t)=\hat{x}^\perp \cdot v^\gamma(\theta)(x,t)=C(\gamma) P.V.\int_{\mathbb{R}^2} \hat{x}^\perp\cdot \frac{(x-y)^\perp}{|x-y|^{3+\gamma}} (\theta (y,t)-\theta(x,t)) dy\\
    =C(\gamma) P.V. \int_{0}^\infty \int_{-\pi}^\pi r' \frac{r-r'\cos(\alpha')}{|(r-r')^2+2rr'(1-\cos(\alpha'))|^{\frac{3+\gamma}{2}}}(\tilde{g}_N(r')-\tilde{g}_N(r))dr' d\alpha'.
\end{align*}
On the other hand, the radial component of the velocity is
\begin{equation*}
    v_r^\gamma (\theta)(x,t)=C(\gamma)  P.V.\int_{\mathbb{R}^2} \hat{x} \cdot\frac{(x-y)^\perp}{|x-y|^{3+\gamma}} \theta (y,t) dy,
\end{equation*}
that, in the case of a function of the form $g_N(r,\alpha)=\tilde{g}_N(r)\sin{(N\alpha-N\alpha_0(r))}$, can be expressed in polar coordinates as follows:
\begin{align*}
         &v_r^\gamma (g_N)(r,\alpha)=\\&
         = C(\gamma)  P.V.\int_{\mathbb{R}^+\times [-\pi,\pi]}(r')^2\frac{(\cos(\alpha)\sin(\alpha')-\sin(\alpha)\cos(\alpha'))\sin(N\alpha'-N\alpha_0(r'))\tilde{g}_N(r')}{|(r\cos(\alpha)-r'\cos(\alpha'))^2+(r\sin(\alpha)-r'\sin(\alpha'))^2|^{\frac{3+\gamma}{2}}}dr'd\alpha'\\&
         =C(\gamma)  P.V.\int_{\mathbb{R}^+\times [-\pi,\pi]}(r')^2\frac{\sin(\alpha'-\alpha)\sin(N\alpha'-N\alpha_0(r'))\tilde{g}_N(r')}{|(r-r')^2+2rr'(1-\cos(\alpha-\alpha'))|^{\frac{3+\gamma}{2}}}dr'd\alpha'\\&
         =C(\gamma)  P.V.\int_{\mathbb{R}^+\times [-\pi,\pi]}(r')^2\frac{\sin(\alpha'-\alpha)\sin(N\alpha'-N\alpha)\cos(N\alpha-N\alpha_0(r'))\tilde{g}_N(r')}{|(r-r')^2+2rr'(1-\cos(\alpha-\alpha'))|^{\frac{3+\gamma}{2}}}dr'd\alpha'\\&
        =C(\gamma) P.V. \int_{-r}^\infty\tilde{g}_N(r+h)\mathcal{C}_N(r+h,\alpha)\int_{ -\pi}^\pi\frac{\sin(\alpha')\sin(N\alpha')}{|h^2+2r(r+h)(1-\cos(\alpha'))|^{\frac{3+\gamma}{2}}}d\alpha'dh.
\end{align*}
where $\mathcal{C}_N(\rho,\phi)=\rho^2\cos(N\phi-N\alpha_0(\rho))$.\\ \\
Hence, summarizing, for $\gamma\in [-1,1)$ and $g_N(r,\alpha)=\tilde{g}_N(r)\sin{(N\alpha-N\alpha_0(r))}$ with $\tilde{g}_N$ regular enough,
\begin{align}
 \label{velocidadangular2}
    &v_\alpha^\gamma (g_N)(r,\alpha)=C(\gamma)P.V. \int_{-r}^\infty \tilde{g}_N(r+h) \mathcal{S}_N(r+h,\alpha)
\int_{-\pi}^\pi \frac{[(r+h)(1-\cos(\alpha'))-h]\cos(N\alpha') }{|h^2+2r(r+h)(1-\cos(\alpha'))|^{\frac{3+\gamma}{2}}}d\alpha'dh, \\
 \label{velocidadangular}
    &v_\alpha^\gamma (\tilde{g}_N)(r)=C(\gamma)P.V. \int_{0}^\infty \int_{-\pi}^\pi r' \frac{r-r'\cos(\alpha')}{|(r-r')^2+2rr'(1-\cos(\alpha'))|^{\frac{3+\gamma}{2}}}(\tilde{g}_N(r')-\tilde{g}_N(r))d\alpha' dr',\\
    \label{velocidadradial}
   & v_r^\gamma (g_N)(r,\alpha)
=C(\gamma)P.V.  \int_{-r}^\infty\tilde{g}_N(r+h)\mathcal{C}_N(r+h,\alpha)\int_{-\pi}^\pi\frac{\sin(\alpha')\sin(N\alpha')}{|h^2+2r(r+h)(1-\cos(\alpha'))|^{\frac{3+\gamma}{2}}}d\alpha'dh,
\end{align}
where $\mathcal{S}_N(\rho,\phi)=\rho\sin(N\phi-N\alpha_0(\rho))$ and  $\mathcal{C}_N(\rho,\phi)=\rho^2\cos(N\phi-N\alpha_0(\rho))$.\\\\
Once stated those important expressions, let us start with the following technical lemmas.
\vspace{3mm}
\begin{lma}
\label{dilations}
Let $\gamma \in [-1,1)$, let $f:\mathbb{R}^2 \rightarrow \mathbb{R}$ be a radial $C^\infty$ function (without relabelling its radial expression $f(r)=f(|x|)=f(x_1,x_2)$), and let $\lambda>0$. Then
\[
v^\gamma_\alpha (f(\lambda \cdot))\left(\frac{r}{\lambda}\right)=\lambda^{\gamma} 
 v^\gamma_\alpha(f(\cdot))(r),
\]
and
\[
\frac{\partial v^\gamma_\alpha (f(\lambda \cdot)) }{\partial r} \left(\frac{r}{\lambda}\right)= \lambda^{1+\gamma} \frac{\partial v_\alpha^\gamma (f(\cdot))}{\partial r}(r),
\]
hold for every $r \in \mathbb{R}_+$.
\begin{proof}
Firstly, the expression of $v_\alpha^\gamma$ in polar coordinates can be written using (\ref{velocidadangular}) as 
\[
v_\alpha^\gamma (f)(r)=C(\gamma)P.V. \int_{\mathbb{R}_+ \times [-\pi,\pi]} r'\frac{r-r' \cos(\alpha')}{|r^2+r'^2-2rr' \cos(\alpha')|^\frac{3+\gamma}{2}}(f(r')-f(r)) d\alpha' d r'.
\]
So  
\[
v_\alpha^\gamma (f(\lambda \cdot))\left(\frac{r}{\lambda}\right)=C(\gamma)P.V. \int_{\mathbb{R}_+ \times [-\pi,\pi]} r'\frac{r/\lambda-r' \cos(\alpha')}{|(r/\lambda)^2+r'^2-2r'r/\lambda \cos(\alpha')|^\frac{3+\gamma}{2}}(f( \lambda r')-f(r)) d\alpha' d r',
\]
and we consider the change of variables $s=\lambda r'$, obtaining
\[
=  C(\gamma) P.V. \int_{\mathbb{R}_+ \times [-\pi,\pi]} \frac{s}{\lambda^2}\frac{r/\lambda-s/\lambda \cos(\alpha')}{|(r/\lambda)^2+(s/\lambda)^2-2(s/\lambda)(r/\lambda) \cos(\alpha')|^\frac{3+\gamma}{2}}(f(s)-f(r)) d\alpha' d s=
\]
\[
=  C(\gamma) P.V. \int_{\mathbb{R}_+ \times [-\pi,\pi]}\lambda^{\gamma} s \frac{r-s \cos(\alpha')}{|r^2+s^2-2sr \cos(\alpha')|^\frac{3+\gamma}{2}}(f(s)-f(r)) d\alpha' d s= \lambda^{\gamma}  v^{\gamma}_\alpha(f(\cdot))(r).
\]
Now, the second result follows from the first as
\[
\frac{\partial v_\alpha^\gamma(f(\lambda \cdot))}{\partial r} \left(\frac{r}{\lambda} \right)=
\left. \frac{\partial v_\alpha^\gamma(f(\lambda \cdot))(s)}{\partial s} \right\vert_{s=\frac{r}{\lambda}}=
\left. \frac{\partial \lambda^{\gamma} v_\alpha^\gamma(f(\cdot))( \lambda s)}{\partial s} \right\vert_{s=\frac{r}{\lambda}}= 
\lambda^{1+\gamma} \frac{\partial v_\alpha^\gamma (f(\cdot))}{\partial r}(r).
\]
\end{proof}
\end{lma}
\begin{lma}
\label{lemag}
    Let $\gamma \in [-1,1)$ and $\varepsilon>0$. There exists a radial $C^\infty$ function $g:\mathbb{R}^2 \rightarrow \mathbb{R}$ compactly supported in $r\in[R_0,R_1]$ for $0<R_0<R_1<\varepsilon$ and  such that 
    \[
    \left|
    \frac{\partial \frac{v_\alpha^\gamma(g)(r)}{r}}{\partial r}(r=1)
    \right|
    >0.
    \]
    \begin{proof}
    First, let us consider a function $h(r)$ smooth and with compact support in some $[a_0,a_1]$, $0<a_0<a_1$ and such that 
    \[
    \int_{0}^\infty sh(s)ds=1.
    \]
    Let $g_\lambda(r)=\lambda^2h(\lambda r)$ and let us compute the following derivative
    \[
    \frac{\partial \frac{v_\alpha^\gamma(g_\lambda)(r)}{r}}{\partial r}(r)= 
  \frac{1}{r} \frac{\partial v_\alpha^\gamma(g_\lambda)(r)}{\partial r}-\frac{1}{r^2} v_\alpha^\gamma(g_\lambda)(r).
    \]
   Using differentiation under the integral sign, we obtain 
    \[
     \frac{1}{r}\frac{\partial v_\alpha^\gamma(g_\lambda)(r)}{\partial r}(r)=C(\gamma)
    \int_{\mathbb{R}_+ \times [-\pi,\pi]} \frac{r'}{r}\partial_r \left( \frac{r-r' \cos(\alpha')}{|r^2+r'^2-2rr' \cos(\alpha')|^\frac{3+\gamma}{2}}\right) g_\lambda(r') d\alpha' d r'
    \]
    \[
    =C(\gamma) \int_{\mathbb{R}_+ \times [-\pi,\pi]} \frac{r'}{r} \left( 
    \frac{1}{|r^2+r'^2-2rr' \cos(\alpha')|^\frac{3+\gamma}{2}}-(3+\gamma)\frac{(r-r' \cos(\alpha'))^2}{|r^2+r'^2-2rr' \cos(\alpha')|^\frac{5+\gamma}{2}} 
    \right) g_\lambda(r') d\alpha' d r'.
    \]
\\ \\
    We have to  subtract $\frac{1}{r^2}v_\alpha^\gamma(g_\lambda)(r)$ to this and check whether we can find $g_\lambda$ in the conditions such that we get a non-zero real number. 
    \begin{align*}
\frac{\partial \frac{v_\alpha^\gamma(g_\lambda)(r)}{r}}{\partial r}(r) 
   = C(\gamma)\int_{\mathbb{R}_+ \times [-\pi,\pi]} \frac{r'}{r} \left( 
    \frac{1}{|r^2+r'^2-2rr' \cos(\alpha')|^\frac{3+\gamma}{2}}-(3+\gamma)\frac{(r-r' \cos(\alpha'))^2}{|r^2+r'^2-2rr' \cos(\alpha')|^\frac{5+\gamma}{2}} 
    \right.\\
    \left.
    -\frac{1}{r}\frac{r-r' \cos(\alpha')}{|r^2+r'^2-2rr' \cos(\alpha')|^\frac{3+\gamma}{2}}
    \right)
    \lambda^2 h(\lambda r') d\alpha' d r',
    \end{align*}
    and now let us compute the following limit:
    \begin{align*}
    \lim_{\lambda \rightarrow \infty} \frac{\partial \frac{v_\alpha^\gamma(g_\lambda)(r)}{r}}{\partial r}(r) 
    =2 \pi C(\gamma)\int_0^{\infty} \frac{r'}{r} 
        \left(
\frac{1}{r^{3+\gamma}}-(3+\gamma)\frac{r^2}{r^{5+\gamma}}-\frac{1}{r^{3+\gamma}}
    \right) h( r')dr'   
    =- \frac{2 \pi (3+\gamma)}{r^{4+\gamma}} .
    \end{align*}  
    So for every compact $K\subset\mathbb{R}_+$ and positive number $\delta>0$, there is a $\lambda > 0$ big enough such that 
    \[
    -\frac{2 \pi (3+\gamma)C(\gamma)}{r^{4+\gamma}}-\delta\leq
    \frac{\partial \frac{v_\alpha^\gamma(g_\lambda)(r)}{r}}{\partial r}(r)\leq-\frac{2 \pi (3+\gamma)C(\gamma)}{r^{4+\gamma}}+\delta, \hspace{3mm} \forall r\in K,
    \]
    so taking $K=[1/2,3/2]$,  $\delta<\pi (3+\gamma)$, and $g=g_\lambda$ with the $\lambda$ associated to $\delta$, then
    \[
     \left|
     \frac{\partial \frac{v_\alpha^\gamma(g_\lambda)(r)}{r}}{\partial r}(1)
     \right| >\pi (3+\gamma)C(\gamma)>0.
    \]
    Finally notice that $g=g_\lambda$ is compactly supported in $[R_0,R_1]$ with $R_i=\frac{a_i}{\lambda}$ for $i=0,1$ and notice also that, since we can do $\lambda$ as big as we choose, the support is as concentrated at the origin as we want.
   \end{proof}
\end{lma}
\begin{rmk}
    In the proof of the following lemma we find the restriction $\beta<2+\gamma$.
\end{rmk}
\begin{lma}
\label{lemaf1}
    For any $\beta \in [1,2+\gamma)$, $\gamma \in [-1,1)$ and $c,K,\varepsilon>0$, there exist a $C^\infty$ radial function $f_1: \mathbb{R}^2 \rightarrow \mathbb{R}$ with support in some small crown around the origin (this is $supp(f_1(r)) \subset (a_0,a_1)$ with $0<a_0<a_1<\varepsilon$ depending on $K,c, \beta$) such that 
    \[
       \|  f_1   \|  _{H^\beta}\leq c,
    \]
    and 
    \[
 \left| \frac{\partial \left( \frac{v^\gamma_\alpha (f_1( \cdot))}{r} \right)}{\partial r} \left(r=2a_1\right) \right| \geq \frac{ K}{2a_1}.
    \]
    \begin{proof}
        Let $g\in C^\infty_c$ be given by Lemma \ref{lemag} supported in $r\in[R_0,R_1]\subset (0,\frac{1}{2})$, with Sobolev norm $\|g\|_{H^\beta}<C$ and such that $|\frac{\partial \frac{v_\alpha^\gamma(g)(r)}{r}}{\partial r}(r=1)|=\delta>0$. We define 
        \[
        g_{\lambda_1,\lambda_2}(r)= \frac{g(\lambda_1 r)}{\lambda_2 \lambda_1^{\beta-1}},
        \]
        with $\lambda_1,\lambda_2>1$. Notice that, since $g$ is supported in $[R_0,R_1]$, then $g_{\lambda_1,\lambda_2}$ is supported in $[\frac{R_0}{\lambda_1}, \frac{R_1}{\lambda_1}]$. \\ \\
        The Sobolev norm $H^\beta$ of $g_{\lambda_1,\lambda_2}$ is bounded by
        \[
        \|   g_{\lambda_1,\lambda_2}    \|  _{H^\beta} \leq \frac{   \|  g   \|  _{H^\beta}}{\lambda_2}\leq \frac{C}{\lambda_2}.
        \]
        Using now Lemma \ref{dilations}, 
        \[
        \frac{\partial \frac{v_\alpha^\gamma(g_{\lambda_1,\lambda_2})(r)}{r}}{\partial r}\left(r=\frac{1}{\lambda_1}\right)=
        \lambda_1 \frac{\partial v_\alpha^\gamma(g_{\lambda_1,\lambda_2})(r)}{\partial r}\left(r=\frac{1}{\lambda_1}\right)- \lambda_1^2 v_\alpha^\gamma(g_{\lambda_1,\lambda_2})\left(r=\frac{1}{\lambda_1}\right)
         \]
         \[
         =\frac{\lambda_1 \lambda_1^{1+\gamma}}{\lambda_2 \lambda_1^{\beta-1}}\frac{\partial v_\alpha^\gamma(g)(r)}{\partial r}(r=1)-\frac{\lambda_1^2 \lambda_1^{\gamma}}{\lambda_2\lambda_1^{\beta-1}} v_\alpha^\gamma (g)(r=1)=
         \frac{\lambda_1^{2+\gamma}}{\lambda_2 \lambda_1^{\beta-1}}\frac{\partial \frac{v_\alpha^\gamma(g)(r)}{r}}{\partial r}(r=1)
         \]
         \[
         = \frac{\lambda_1^{2+\gamma}}{\lambda_2 \lambda_1^{\beta-1}} \delta =\frac{\lambda_1^{3-\beta+\gamma}}{\lambda_2}\delta=\frac{\lambda_1^{2-\beta+\gamma}}{\frac{1}{\lambda_1}\lambda_2}\delta.
         \]
         Notice that $1/\lambda_1$ is now the point where we are evaluating the radial derivative of $v_\alpha(g_{\lambda_1,\lambda_2})/r$ and where we want it to be big. It is still far from the support of $g_{\lambda_1,\lambda_2}$ (indeed, $0<\frac{R_0}{\lambda_1}<\frac{R_1}{\lambda_1}<\frac{1}{2\lambda_1}<\frac{1}{\lambda_1}$).
    So now let us make the right choice of the parameters $\lambda_1,\lambda_2>1$. \\ \\
    First of all, we choose $\lambda_2$ big enough such that $C/\lambda_2<c$. Now, we choose $\lambda_1$ big enough such that
    \[
    \frac{\lambda_1^{2-\beta+\gamma}}{\lambda_2}\delta>K,
    \]
    holds. Then defining $a_0=\frac{R_0}{\lambda_1}$ and $a_1=\frac{1}{2\lambda_1}$ we conclude
     \[
 \left| \frac{\partial \left( \frac{v^\gamma_\alpha (g_{\lambda_1,\lambda_2}( \cdot))}{r} \right)}{\partial r} \left(r=2a_1\right) \right| \geq \frac{K}{2a_1}. 
    \]
     \end{proof}
\end{lma}
\begin{lma}
\label{vradial}
Let $\tilde{g}_N:\mathbb{R}_+\rightarrow \mathbb{R}$ be a $C^\infty(\mathbb{R}_+)$ function with support in the interval $(R_0,R_1)$ and let $\alpha_0(r,t) \in C^\infty(\mathbb{R}_+\times[0,\infty))$ be a radial function. Let us define $g_N:\mathbb{R}^2\rightarrow \mathbb{R}$ in polar coordinates as 
\[
g_N(r,\alpha):= \tilde{g}_N(r)\sin{(N\alpha+N \alpha_0(r,t))},
\]
with $N$ a natural number and consider $\gamma\in [-1,1)$. Then there exist a constant $C>0$ depending on $R_0,R_1$ and $\gamma$ such that
\[
|v^{\gamma}_r(g_N)(r,\alpha)|\leq \frac{C  \|   \tilde{g}_N  \|   _\infty}{N^2|r-R_1|^{2+\gamma}} \ln
\left(
e+\frac{r}{(r-R_1)^2}
\right), \hspace{5mm} \forall r>R_1,
\]
and 
\[
|v^{\gamma}_r(g_N)(r,\alpha)|\leq \frac{C  \|   \tilde{g}_N  \|   _\infty}{N(R_0-r)^{2+\gamma}}, \hspace{5mm} \forall 0\leq r<R_0.
\]
Furthermore, we have that if $  \|   \tilde{g}_N  \|   _{C^k}\leq MN^k$, for $k=0,1,...,m$, then 
\[
\left|\frac{\partial^m v^\gamma_{r} (g_N)}{\partial x_1^{m-k} \partial x_2^{k}}(r,\alpha)\right| \leq \frac{CMN^{m-1}}{(R_0-r)^{2+\gamma}}, \hspace{3mm} \forall 0\leq r<R_0,
\]
with $C$ depending on $R_0,R_1,\gamma$ and $m$.
\begin{proof}
For the first inequality, according to the expression (\ref{velocidadradial}) and noticing that principal values are not necessary anymore because we are far from the support, it is enough to find a bound for \[
 \left|
 \int_{-r}^\infty\tilde{g}_N(r+h)\cos(N\alpha-N\alpha_0(r+h))(r+h)^2\int_{ [-\pi,\pi]}\frac{\sin(\alpha')\sin(N\alpha')}{|h^2+2r(r+h)(1-\cos(\alpha'))|^{\frac{3+\gamma}{2}}}d\alpha'dh
 \right|.
\]
for $r\notin [R_0,R_1]$.\\\\
We first integrate in the angular coordinate, so let us fix $h\in \mathbb{R_+}$ such that $|g_N(r+h)|>0$ and bound
\[
 \left| \int_{ [-\pi,\pi]}\frac{\sin(\alpha') \sin(N\alpha')}{|h^2+2r(r+h)(1-\cos(\alpha'))|^{\frac{3+\gamma}{2}}} \, d\alpha' \right|.
\]
We define temporarily
\[
w(\alpha')= \frac{\sin(\alpha') }{|h^2+2r(r+h)(1-\cos(\alpha'))|^{\frac{3+\gamma}{2}}},
\]
and we start the bounding integrating by parts twice, taking into account that all the involved $\alpha'$-dependent function are $2\pi$ periodic and the fact that the denominator is never zero. Hence,
\[
 \left| \int_{ [-\pi,\pi]}\frac{\sin(\alpha') \sin(N\alpha')}{|h^2+2r(r+h)(1-\cos(\alpha'))|^{\frac{3+\gamma}{2}}} \, d\alpha' \right|
 \]
 \begin{align*}
 &= \left| \left[ -\frac{\cos(N \alpha')}{N} w(\alpha') \right]_{-\pi}^{\pi}+ \int_{ [-\pi,\pi]}\frac{\cos(N\alpha')}{N}w'(\alpha') \, d\alpha' \right|
 =\left|\int_{ [-\pi,\pi]}\frac{\cos(N\alpha')}{N}w'(\alpha') \, d\alpha' \right|
 \end{align*}
 \begin{equation}
 \label{partes}
 =
 \left|\int_{-\pi}^{\pi}\frac{\cos(N\alpha')}{N}
\left(
\frac{\cos(\alpha')}{|h^2+2r(r+h)(1-\cos(\alpha'))|^{\frac{3+\gamma}{2}}} -\frac{(3+\gamma)\sin^2(\alpha')r(r+h)}{|h^2+2r(r+h)(1-\cos(\alpha'))|^{\frac{5+\gamma}{2}}}
\right)
d \alpha'\right|
\end{equation}
\begin{align*}
 &=\left|\int_{ [-\pi,\pi]}\frac{\sin(N\alpha')}{N^2}w''(\alpha') \, d\alpha' \right|
 \leq \frac{1}{N^2}\int_{ [-\pi,\pi]}\left|w''(\alpha')\right| \, d\alpha' 
\\
&\leq \frac{1}{N^2} \int_{-\pi}^{\pi}\left|\frac{\partial}{\partial \alpha'}
\left(
\frac{\cos(\alpha')}{|h^2+2r(r+h)(1-\cos(\alpha'))|^{\frac{3+\gamma}{2}}} -\frac{(3+\gamma)\sin^2(\alpha')r(r+h)}{|h^2+2r(r+h)(1-\cos(\alpha'))|^{\frac{5+\gamma}{2}}}
\right)
\right| d \alpha'
\end{align*}
\[
\leq \frac{1}{N^2} \int_{-\pi}^\pi \left| 
-\frac{\sin(\alpha')}{|h^2+2r(r+h)(1-\cos(\alpha'))|^{\frac{3+\gamma}{2}}} -(3+\gamma)\frac{\cos(\alpha')\sin(\alpha')r(r+h)}{|h^2+2r(r+h)(1-\cos(\alpha'))|^{\frac{5+\gamma}{2}}}
\right.
\]
\[
\left.
-\frac{2(3+\gamma)\sin(\alpha')\cos(\alpha')r(r+h)}{|h^2+2r(r+h)(1-\cos(\alpha'))|^{\frac{5+\gamma}{2}}}+\frac{(3+\gamma)(5+\gamma)\sin^3(\alpha')r^2(r+h)^2}{|h^2+2r(r+h)(1-\cos(\alpha'))|^{\frac{7+\gamma}{2}}}
\right| d\alpha'
\]
\[
\leq \frac{1}{N^2} \int_{-\pi}^\pi \left| 
-\frac{\sin(\alpha')}{|h^2+2r(r+h)(1-\cos(\alpha'))|^{\frac{3+\gamma}{2}}} - \frac{3(3+\gamma)\sin(\alpha')\cos(\alpha') r(r+h)}{|h^2+2r(r+h)(1-\cos(\alpha'))|^{\frac{5+\gamma}{2}}}
\right.
\]
\[
\left.
+\frac{(3+\gamma)(5+\gamma)\sin^3(\alpha')r^2(r+h)^2}{|h^2+2r(r+h)(1-\cos(\alpha'))|^{\frac{7+\gamma}{2}}}
\right| d\alpha'.
\]
Now, we realize that there exist $c_1,c_2>0$ such that $|\sin(\alpha')| \leq c_1 |\alpha'|$ and $1-\cos(\alpha')\geq c_2 \alpha'^2 $ for all $\alpha' \in [-\pi,\pi]$, hence
\begin{equation}
\label{cociente}
\frac{|\sqrt{r(r+h)}\sin(\alpha')|}{|h^2+2r(r+h)(1-\cos(\alpha'))|^{\frac{1}{2}}} \leq \frac{|\sqrt{r(r+h)}c_1\alpha'|}{|2r(r+h)c_2\alpha'^2|^{\frac{1}{2}}} \leq C,
\end{equation}
and also
\begin{align*}
    \frac{(3+\gamma)(5+\gamma)\sin^3(\alpha')r^2(r+h)^2}{|h^2+2r(r+h)(1-\cos(\alpha'))|^{\frac{7+\gamma}{2}}}\leq C\frac{r(r+h)|\sin(\alpha')|}{|h^2+2r(r+h)(1-\cos(\alpha'))|^{\frac{5+\gamma}{2}}}.
\end{align*}
Now we continue with the bounding of the desired quantity as
\begin{align*}
 \left| \int_{ [-\pi,\pi]}\frac{\sin(\alpha') \sin(N\alpha')}{|h^2+2r(r+h)(1-\cos(\alpha'))|^{\frac{3+\gamma}{2}}} \, d\alpha' \right|\\
  \leq \frac{C}{N^2} \int_{-\pi}^{\pi} \frac{1}{|h^2+2r(r+h)(1-\cos(\alpha'))|^{\frac{3+\gamma}{2}}}
\left(
1+
\frac{|r(r+h)\sin(\alpha')|}{|h^2+2r(r+h)(1-\cos(\alpha'))|}
\right)
  d \alpha' 
\\
    \leq \frac{2C}{N^2|h|^{3+\gamma}} \int_{0}^{\pi} 
\left(
1+
\frac{c_1r(r+h)\alpha'}{h^2+2r(r+h)c_2\alpha'^2}
\right)
  d \alpha' 
  \leq 
\frac{2C}{N^2|h|^{3+\gamma}} 
\left(
\pi+
\ln
\left(1+\frac{2r(r+h)c_2\pi^2}{h^2}
\right)
\right).
\end{align*}
 Finally, integrating in $h$ and taking into account that the support of $g_N$ is in $(R_0,R_1)$ and as a consequence $r+h\in(R_0,R_1)$ and $ h\in (R_0-r,R_1-r)$, we obtain
\begin{align*}
 &|v_r^\gamma (g_N)(r,\alpha)|
 \\&\leq \int_{-r}^\infty  \frac{C}{N^2} |\mathcal{C}_N(r+h,\alpha)||\tilde{g}_N(r+h)| \frac{1}{|h|^{3+\gamma}}
\left|
\pi+
\ln
\left(1+\frac{2r(r+h)c_2\pi^2}{h^2}
\right)
\right|dh
\\& \leq \frac{C  \|   \tilde{g}_N  \|   _\infty}{N^2} \int_{(R_0-r,R_1-r)} \frac{1}{|h|^{3+\gamma}}
\left(
\pi+
\ln(1+\frac{2c_2 R_1r\pi^2}{h^2})
\right) dh\\&
\leq \frac{C  \|   \tilde{g}_N  \|   _\infty}{N^2|r-R_1|^{2+\gamma}} \ln
\left(
e+\frac{r}{(r-R_1)^2}
\right), \hspace{3mm} \forall r>R_1 .
\end{align*}
The bound for small radii $0<r<R_0$ is done in a simpler way. Indeed repeating the previous process until (\ref{partes}) and using (\ref{cociente}) we have 
\begin{align*}
    \left|
 \int_{-r}^\infty\tilde{g}_N(r+h)\cos(N\alpha-N\alpha_0(r+h))(r+h)^2\int_{ [-\pi,\pi]}\frac{\sin(\alpha')\sin(N\alpha')}{|h^2+2r(r+h)(1-\cos(\alpha'))|^{\frac{3+\gamma}{2}}}d\alpha'dh
 \right|
    \\ \leq \frac{C}{N} 
    \int_{-r}^\infty |\Tilde{g}_N(r+h)|
    \int_{-\pi}^\pi 
    \left|
    \frac{1}{|h^2+2r(r+h)(1-\cos(\alpha'))|^{\frac{3+\gamma}{2}}}
    \right|d\alpha'dh\\
    \leq \frac{C\|\Tilde{g}_N\|_{L^\infty}}{N} \int_{R_0-r}^{R_1-r}\frac{1}{|h|^{3+\gamma}}dh\leq 
    \frac{C\|\Tilde{g}_N\|_{L^\infty}}{N(R_0-r)^{2+\gamma}} .
\end{align*}
For higher derivatives we use the property
\[
v_r^\gamma(f)=\cos(\alpha(x))v_1^\gamma(f)+\sin(\alpha(x))v_2^\gamma(f).
\]
and the previous bound 
\begin{align*}
    |v_r^\gamma(g_N)(r,\alpha)|
\leq \frac{C  \|   \tilde{g}_N  \|   _\infty}{N|R_0-r|^{2+\gamma}}, \hspace{5mm} \forall 0<r<R_0.
\end{align*}
Differentiating $m$ times, applying the chain rule as many times as necessary and taking into account that, although $r^{-k}$ for $k\in \{0,...,n-1\}$ may appear when differentiating in polar coordinates, it will not be a problem because we consider our solution in a fixed annulus, then
\begin{align*}
&\left|\frac{\partial^m v^\gamma_{r} (g_N)}{\partial x_1^{m-k} \partial x_2^{k}}(r,\alpha)\right| \\&
\leq
\left|v^\gamma_{r} \left(\frac{\partial^m g_N}{\partial x_1^{m-k} \partial x_2^{k}}\right)(r,\alpha)\right|
+C\sum_{i=0}^{m-1}\sum_{j=0}^i \left|\frac{\partial^i v^\gamma_{1} (g_N)}{\partial x_1^{i-j} \partial x_2^{j}}(r,\alpha)\right|
+C\sum_{i=0}^{m-1}\sum_{j=0}^i 
\left|\frac{\partial^i v^\gamma_{2} (g_N)}{\partial x_1^{i-j} \partial x_2^{j}}(r,\alpha)\right|
\\&
=
\left|v^\gamma_{r} \left(\frac{\partial^m g_N}{\partial x_1^{m-k} \partial x_2^{k}}\right)(r,\alpha)\right|
+C\sum_{i=0}^{m-1}\sum_{j=0}^i 
\left|v^\gamma_{1}\left(\frac{\partial^i  g_N}{\partial x_1^{i-j} \partial x_2^{j}}\right)(r,\alpha)\right|
+C\sum_{i=0}^{m-1}\sum_{j=0}^i 
\left|v^\gamma_{2}\left(\frac{\partial^i  g_N}{\partial x_1^{i-j} \partial x_2^{j}}\right)(r,\alpha)\right|,
\end{align*}
with $C$ depending on $R_0,R_1$ and $m$.\\ \\
Now, using the previous bound we get 
\[
\left|v^\gamma_{r} \left(\frac{\partial^m g_N}{\partial x_1^{m-k} \partial x_2^{k}}\right)(r,\alpha)\right|\leq
\frac{C  \|   \tilde{g}_N  \|   _{C^m}}{N|r-R_1|^{2+\gamma}}\leq \frac{CMN^{m-1}}{|R_0-r|^{2+\gamma}}.
\]
Also, using
\[
\left|v^\gamma_{1}(f)(x)\right| \leq C\frac{  \|   f  \|   _{L^1}}{d(x,supp(f))^{2+\gamma}},
\]
and $  \|   g_N  \|   _{C^k}\leq MN^k$ for $k=0,1,...,m-1$, then
\[
\left|v^\gamma_{1}\left(\frac{\partial^i  g_N}{\partial x_1^{i-j} \partial x_2^{j}}\right)(r,\alpha)\right| \leq C\frac{ \|   \frac{\partial^i  g_N}{\partial x_1^{i-j} \partial x_2^{j}}  \|   _{L^1}}{|R_0-r|^{2+\gamma}} \leq C\frac{  \|   g_N  \|   _{C^i}}{|R_0-r|^{2+\gamma}}\leq \frac{CMN^i}{|R_0-r|^{2+\gamma}}, \hspace{3mm} \forall 0<r<R_0,
\] 
for all $i=0,1,...,m-1$. 
Doing the same reasoning for the velocity in the second component and putting it all together, we obtain the result.
\end{proof}
\end{lma}
\begin{lma}
Let $\tilde{g}_N:[0,\infty)\rightarrow \mathbb{R}$ be a $C^\infty$ function with support in the interval $(R_0,R_1)$ and let $\alpha_0(r,t) \in C^\infty(\mathbb{R}_+\times[0,\infty))$ be a radial function. Let us define $g_N:\mathbb{R}^2\rightarrow \mathbb{R}$ in polar coordinates as 
\begin{align*}
g_N(r,\alpha):= \tilde{g}_N(r)\sin{(N\alpha+N \alpha_0(r,t))},
\end{align*}
with $N$ a natural number and consider $\gamma\in [-1,1)$. Then there exist a constant $C>0$ depending on $R_0,R_1$ and $\gamma$ such that if $r>R_1$, then
\begin{align}
    \label{vangular}
|v^{\gamma}_\alpha(g_N)(r,\alpha)|\leq \frac{C  \|   \tilde{g}_N  \|   _{L^\infty}}{N^2(r-R_1)^{2+\gamma}}.
\end{align}
\begin{proof}
    According to \ref{velocidadangular2}, it is enough to find a bound for 
    \begin{align*}
    \left|
        \int_{-r}^\infty \tilde{g}_N(r+h) \mathcal{S}_N(r+h,\alpha)
\int_{-\pi}^\pi \frac{[(r+h)(1-\cos(\alpha'))-h]\cos(N\alpha') }{|h^2+2r(r+h)(1-\cos(\alpha'))|^{\frac{3+\gamma}{2}}}d\alpha'dh 
\right|,
    \end{align*}
    where we removed the principal values due to the fact that we are doing the bound far from the support.\\\\
We pay attention first to the integral in $\alpha'$,
\begin{align*}
    \left|
\int_{-\pi}^\pi w(\alpha')\cos(N\alpha') d\alpha'
\right|,
\end{align*}
where 
\begin{align*}
    w(\alpha')=\frac{[(r+h)(1-\cos(\alpha'))-h]}{|h^2+2r(r+h)(1-\cos(\alpha'))|^{\frac{3+\gamma}{2}}},
\end{align*}
ignoring in the notation the dependence on the radial coordinate, that will be recovered later. Then, choosing $c_1,c_2>0$ such that $|\sin(\alpha')| \leq c_1 |\alpha'|$ and $c_2 \alpha'^2\leq 1-\cos(\alpha')\leq c_1 \alpha'^2 $ for all $\alpha' \in [-\pi,\pi]$ such that we can use the bound (\ref{cociente}), 
\begin{align}
    \label{cotaw''}
    \begin{split}
    |w''(\alpha')|= \left|
\frac{(r+h)\cos(\alpha')}{|h^2+2r(r+h)(1-\cos(\alpha'))|^{\frac{3+\gamma}{2}}}-\frac{3+\gamma}{2}
\frac{4r(r+h)^2(\sin(\alpha'))^2}{|h^2+2r(r+h)(1-\cos(\alpha'))|^{\frac{5+\gamma}{2}}}
\right.\\
\left.
-\frac{3+\gamma}{2}
\frac{2r(r+h)\cos(\alpha')[(r+h)(1-\cos(\alpha'))-h]}{|h^2+2r(r+h)(1-\cos(\alpha'))|^{\frac{5+\gamma}{2}}}
\right.\\
\left.
-\frac{3+\gamma}{2}\frac{5+\gamma}{2}
\frac{4r^2(r+h)^2(\sin(\alpha'))^2[(r+h)(1-\cos(\alpha'))-h]}{|h^2+2r(r+h)(1-\cos(\alpha'))|^{\frac{7+\gamma}{2}}}
\right|
\\ \leq
\frac{C(r+h)}{|h^2+2r(r+h)(1-\cos(\alpha'))|^{\frac{3+\gamma}{2}}}
\left(
1+r\frac{|(r+h)c_1\alpha'^2|+|h|}{|h^2+2r(r+h)c_2\alpha'^2)|}
\right).
\end{split}
\end{align}
Thus, taking into account that $R_0<r+h<R_1$, integrating by parts twice and using (\ref{cotaw''}),
\begin{align*}
     & \left|
\int_{-\pi}^\pi w(\alpha')\cos(N\alpha') d\alpha'
\right|=  
\left|
\left[
w(\alpha')\frac{\sin(N\alpha')}{N}
\right]_{-\pi}^\pi-
\int_{-\pi}^\pi w'(\alpha')\frac{\sin(N\alpha')}{N} d\alpha'
\right|\\&
=\left|\int_{-\pi}^\pi w'(\alpha')\frac{\sin(N\alpha')}{N} d\alpha' 
\right|=
\left|\int_{-\pi}^\pi w''(\alpha')\frac{\cos(N\alpha')}{N^2} d\alpha'
\right|\\
& \leq
C\int_{-\pi}^\pi
\frac{1}{N^2|h|^{3+\gamma}}
\left(
1+r\frac{|(r+h)c_1\alpha'^2|+|h|}{|h^2+2r(r+h)c_2\alpha'^2)|}
\right)
d\alpha'
\\& \leq
C\int_{-\pi}^\pi
\frac{1}{N^2|h|^{3+\gamma}}
\left(
1+\frac{|r(r+h)c_1\alpha'^2|}{|h^2+2r(r+h)c_2\alpha'^2)|}
+\frac{r}{|h|}\frac{1}{(1+\frac{2r(r+h)c_2}{h^2}\alpha'^2)}
\right)
d\alpha'
\\& \leq
C\int_{-\pi}^\pi
\frac{1}{N^2|h|^{3+\gamma}}
\left(
2+\frac{r}{|h|}\frac{1}{(1+\frac{2r(r+h)c_2}{h^2}\alpha'^2)}
\right)d\alpha'
\\&\leq
2C
\frac{1}{N^2|h|^{3+\gamma}}
\left(
\pi+\frac{\sqrt{r}\arctan(\frac{\sqrt{2r(r+h)c_2}}{|h|}\pi)}{\sqrt{2(r+h)c_2}}
\right)
\leq
 \frac{C}{N^2|h|^{3+\gamma}},
\end{align*}
 using in the last bound that $r\sim |h|$ when $r\mapsto \infty$ and that $\frac{\arctan(cy)}{y}\sim c$ when $y\mapsto 0$. Finally, integrating in $h\in (R_0-r,R_1-r)$ we obtain the desired result.
\end{proof}
\end{lma}
\begin{lma}
Let $\gamma\in [-1,1)$, $K_{\gamma}(y)$ be defined by (\ref{kernel}), $\tilde{g}_N:[0,\infty)\rightarrow \mathbb{R}$ be a $C^\infty$ function  with support in the interval $(R_0,R_1)$ and $\alpha_0(r,t) \in C^\infty(\mathbb{R}_+\times[0,\infty))$ be a radial function. If we define $g_N:\mathbb{R}^2\rightarrow \mathbb{R}$ in polar coordinates as 
\[
g_N(r,\alpha):= \tilde{g}_N(r)\sin{(N\alpha+N \alpha_0(r,t))},
\]
for $N$ the natural numbers, then there exist a $C>0$ depending on $R_0,R_1$ and $\gamma$ such that
\begin{align}
\label{vradialmedio}
     \left|
     C(\gamma) \int_{\{1/N<|x-y|<3R_1\}} \hat{x}\cdot K_{\gamma}(x-y) g_N (y,t) dy
    \right|\leq C\|\tilde{g}_N\|_{L^\infty}N^{\gamma}\ln(e+N),\hspace{4mm} \forall x\in \mathbb{R}^2.
\end{align}
\begin{proof}
It is enough to find a bound for 
\begin{align*}
 \left|
\int_{(-1/N,1/N)}\tilde{g}_N(r+h)\mathcal{C}_N(\alpha,r+h)\int_{ [-\pi,\pi]\setminus (-\varepsilon_N,\varepsilon_N)}\frac{\sin(\alpha')\sin(N\alpha')}{|h^2+2r(r+h)(1-\cos(\alpha'))|^{\frac{3+\gamma}{2}}}d\alpha'dh
 \right.
 \\
 \left.
 +
 \int_{(-r,3R_1)\setminus (-1/N,1/N)}\tilde{g}_N(r+h)\mathcal{C}_N(\alpha,r+h)\int_{ [-\pi,\pi]}\frac{\sin(\alpha')\sin(N\alpha')}{|h^2+2r(r+h)(1-\cos(\alpha'))|^{\frac{3+\gamma}{2}}}d\alpha'dh
 \right|,
\end{align*}
where we define $\varepsilon_N$ as
 \begin{align}
 \label{epsilon}
        h^2+2r(r+h)(1-\cos(\varepsilon_N(h))=\frac{1}{N^2}.
    \end{align}
Let us compute the first part of the sum. We are first integrating in the angular coordinate, so let us fix $h\in (-\frac{1}{N},\frac{1}{N})$ such that $|g_N(r+h)|>0$ and bound
\[
 \left| \int_{ [-\pi,\pi]\setminus (-\varepsilon_N,\varepsilon_N)}\frac{\sin(\alpha') \sin(N\alpha')}{|h^2+2r(r+h)(1-\cos(\alpha'))|^{\frac{3+\gamma}{2}}} \, d\alpha' \right|.
\]
Proceeding as in the Lemma \ref{vradial} we define 
\begin{align*}  w(\alpha')=\frac{\sin(\alpha') }{|h^2+2r(r+h)(1-\cos(\alpha'))|^{\frac{3+\gamma}{2}}}.
\end{align*}
 Let $c_1,c_2>0$ be such that $|\sin(\alpha')| \leq c_1 |\alpha'|$, $c_2 \alpha'^2\leq 1-\cos(\alpha')\leq c_1 \alpha'^2 $. Then $w'$ is bounded as follows
\begin{align}
\label{primeraderivada}
\begin{split}
    |w'(\alpha')|=
    \left|
\frac{\cos(\alpha') }{|h^2+2r(r+h)(1-\cos(\alpha'))|^{\frac{3+\gamma}{2}}}
-
\frac{(3+\gamma)r(r+h)\sin(\alpha')^2 }{|h^2+2r(r+h)(1-\cos(\alpha'))|^{\frac{5+\gamma}{2}}}
    \right|
    \\
     \leq 
 \frac{C }{|h^2+2r(r+h)(1-\cos(\alpha'))|^{\frac{3+\gamma}{2}}},
 \end{split}
\end{align}
and for $w''$ it is proved in the Lemma \ref{vradial} that
\begin{align}
\label{segundaderivada}
    |w''(\alpha')|\leq
    \frac{1}{|h^2+2r(r+h)(1-\cos(\alpha'))|^{\frac{3+\gamma}{2}}}
\left(
1+
\frac{|r(r+h)\sin(\alpha')|}{|h^2+2r(r+h)(1-\cos(\alpha'))|}
\right).
\end{align}
In the following calculations we integrate by parts twice. We must take into account that, although the outer parts of the boundary terms vanish because all the functions involved are $2\pi$-periodic, the inner parts do not. To bound the border terms we will use (\ref{primeraderivada}) and (\ref{segundaderivada}) evaluated in $\varepsilon_N$. Using also the expression (\ref{epsilon}) we may say $\varepsilon_N\leq CN^{-1}$. Hence,
\begin{align*}
\left|
 \int_{ [-\pi,\pi]\setminus (-\varepsilon_N,\varepsilon_N)}w(\alpha')\sin(N\alpha') d\alpha' 
 \right|
 \\\leq
  \left|
  2w(\varepsilon_N) \frac{\cos(N\varepsilon_N)}{N}
  \right|+
  \left|\int_{ [-\pi,\pi]\setminus (-\varepsilon_N,\varepsilon_N)}w'(\alpha')\frac{\cos(N\alpha')}{N} d\alpha' 
  \right|
 \\\leq
  C\varepsilon_N(h)N^{2+\gamma}+
  \left|\int_{ [-\pi,\pi]\setminus (-\varepsilon_N,\varepsilon_N)}w'(\alpha')\frac{\cos(N\alpha')}{N} d\alpha' 
  \right|
  \\\leq
  C\varepsilon_N(h)N^{2+\gamma}+   \left|2w'(\varepsilon_N)\frac{\sin(N\varepsilon_N)}{N^2}
   \right|
   +
   \left|
   \int_{ [-\pi,\pi]\setminus (-\varepsilon_N,\varepsilon_N)}w''(\alpha')\frac{\sin(N\alpha')}{N^2} d\alpha' 
   \right|
  \\\leq
  C\varepsilon_N(h)N^{2+\gamma}+
   CN^{1+\gamma}
   +
   \left|
   \int_{ [-\pi,\pi]\setminus (-\varepsilon_N,\varepsilon_N)}w''(\alpha')\frac{\sin(N\alpha')}{N^2} d\alpha' 
   \right|\\
    \leq
   CN^{1+\gamma}
   +
   \frac{1}{N^2}
   \int_{ [-\pi,\pi]\setminus (-\varepsilon_N,\varepsilon_N)}
    \left|w''(\alpha') 
   \right|
   d\alpha' 
   \\
    \leq CN^{1+\gamma}
   +\frac{C}{N^2} \int_{(\varepsilon_N,\pi)} \frac{1}{|h^2+2r(r+h)(1-\cos(\alpha'))|^{\frac{3+\gamma}{2}}}
\left(
1+
\frac{|r(r+h)\sin(\alpha')|}{|h^2+2r(r+h)(1-\cos(\alpha'))|}
\right)
  d \alpha' 
  \\ \leq 
   CN^{1+\gamma}
   +CN^{1+\gamma} \int_{(\varepsilon_N,\pi)}
\frac{r(r+h)\sin(\alpha')}{h^2+2r(r+h)(1-\cos(\alpha'))}
  d \alpha'
  \\
  \leq 
  CN^{1+\gamma}+CN^{1+\gamma}\ln
  \left(
  \frac{h^2+4r(r+h)}{h^2+2r(r+h)(1-\cos(\varepsilon_N))}
  \right)
  \\
  \leq 
  CN^{1+\gamma}+CN^{1+\gamma}\ln
  \left(
  (h^2+4r(r+h))N^2
  \right)
  \\
  \leq
  CN^{1+\gamma} 
  \ln(e+N),
\end{align*}
using (\ref{epsilon}), $R_0<r+h<R_1$ and $|h|\leq 1/N$ in the final inequalities. Now, integrating in $h\in (-1/N,1/N)$ we obtain the bound 
\begin{align*}
   \int_{ (-1/N,1/N)}|\tilde{g}_N(r+h)\cos(N\alpha+N\alpha_0(r+h))(r+h)^2 CN^{1+\gamma} \ln(e+N)|dh\\
   \leq C\|\tilde{g}_N\|_{L^\infty}N^\gamma \ln(e+N).
\end{align*}
We do a similar argument for the second part of the sum. Again, we integrate by parts twice (now all the boundary terms vanish) and we bound $|w''|$. Indeed, 
\begin{align*}
 &\left| \int_{ [-\pi,\pi]}\frac{\sin(\alpha') \sin(N\alpha')}{|h^2+2r(r+h)(1-\cos(\alpha'))|^{\frac{3+\gamma}{2}}} \, d\alpha' \right|\\&
 \leq 
\left|
\int_{-\pi}^\pi 
w''(\alpha') \frac{\sin (N\alpha')}{N^2}
d\alpha'
\right|
 \\& \leq
 \frac{C}{N^2|h|^{3+\gamma}}
 \left(1+
 \int_{ [-\pi,\pi]}\frac{|r(r+h)\sin(\alpha') |}{|h^2+2r(r+h)(1-\cos(\alpha'))|}
  \, d\alpha' \right)
  \\& \leq
 \frac{C}{N^2|h|^{3+\gamma}}
 \left(1+
 \int_{ [0,\pi]}\frac{r(r+h)c_1\alpha' }{h^2+2r(r+h)c_1\alpha'^2}
  \, d\alpha' \right)
  \\&
  \leq 
  \frac{C}{N^2|h|^{3+\gamma}} 
  \left(
  1+\ln
  \left(
  \frac{h^2+2r(r+h)c_1\pi^2}{h^2}
  \right)
  \right)
    \\&
  \leq 
  \frac{C}{N^2|h|^{3+\gamma}} 
    \ln
  \left(e+
  \frac{1}{|h|}
  \right),
\end{align*}
and now integrating in $h\in (-r,3R_1)\setminus (-1/N,1/N)$ we have the bound
\begin{align*}
  & \int_{(-r,3R_1)\setminus (-1/N,1/N)}|\tilde{g}_N(r+h)\cos(N\alpha+N\alpha_0(r+h))(r+h)^2  \frac{C}{N^2|h|^{3+\gamma}} 
    \ln
  \left(e+
  \frac{1}{|h|}
  \right)dh
  \\& \leq 
  C\|\Tilde{g_N}\|_{L^\infty}\int _{1/N}^\infty
  \frac{1}{N^2h^{3+\gamma}}\ln(e+\frac{1}{h})dh
  \leq 
  C\|\Tilde{g_N}\|_{L^\infty}N^\gamma \ln (e+N).
\end{align*}
\end{proof}
\end{lma}
\begin{lma}
    Let $\gamma\in [-1,1)$, $K_{\gamma}(y)$ be defined by (\ref{kernel}) and let $\tilde{g}_N:[0,\infty)\rightarrow \mathbb{R}$ be a $C^\infty$ function  with support in the interval $(R_0,R_1)$ and $\alpha_0(r,t) \in C^\infty(\mathbb{R}_+\times[0,\infty))$. If we define $g_N:\mathbb{R}^2\rightarrow \mathbb{R}$ in polar coordinates as 
\[
g_N(r,\alpha):= \tilde{g}_N(r)\sin{(N\alpha+N \alpha_0(r,t))},
\]
with $N$ a natural number, then
\begin{align}
\label{vangularmedio}
     \left|
     C(\gamma) \int_{\{1/N<|x-y|<3R_1\}} \hat{x}^\perp\cdot K_{\gamma}(x-y) g_N (y,t) dy
    \right|\leq C\|\tilde{g}_N\|_{L^\infty}N^{\gamma}, \hspace{4mm} \forall x\in \mathbb{R}^2.
\end{align}
\begin{proof}
It is enough to find a bound for 
    \begin{align*}
    \left|
        \int_{-1/N}^{1/N} \tilde{g}_N(r+h) \mathcal{S}_N(r+h,\alpha)
\int_{[-\pi,\pi]\setminus (-\varepsilon_N,\varepsilon_N)} \frac{[(r+h)(1-\cos(\alpha'))-h]\cos(N\alpha') }{|h^2+2r(r+h)(1-\cos(\alpha'))|^{\frac{3+\gamma}{2}}}d\alpha'dh 
\right.\\
+
\left.
   \int_{(-r,3R_1)\setminus (-1/N,1/N)}\tilde{g}_N(r+h) \mathcal{S}_N(r+h,\alpha)
\int_{-\pi}^\pi \frac{[(r+h)(1-\cos(\alpha'))-h]\cos(N\alpha') }{|h^2+2r(r+h)(1-\cos(\alpha'))|^{\frac{3+\gamma}{2}}}d\alpha'dh 
\right|,
    \end{align*}
    where 
    \begin{align*}
        h^2+2r(r+h)(1-\cos(\varepsilon_N(h)))=\frac{1}{N^2}.
    \end{align*}
    Again we integrate in the angular coordinate, so we fix an $h\in \mathbb{R}$ such that $|\tilde{g}_N(r+h)|>0$ and we define 
    \begin{align*}
        w(\alpha')=\frac{[(r+h)(1-\cos(\alpha'))-h] }{|h^2+2r(r+h)(1-\cos(\alpha'))|^{\frac{3+\gamma}{2}}}.
    \end{align*}
    For the first integral we integrate by parts twice taking into account that the inner parts of the boundary terms do not vanish. The reader may check that $|w'(\varepsilon_N)|\leq CN^{3+\gamma}$. We use the bound for $|w''|$ done in (\ref{cotaw''}) with $c_1,c_2>0$ such that $|\sin(\alpha')| \leq c_1 |\alpha'|$, $c_2 \alpha'^2\leq 1-\cos(\alpha')\leq c_1 \alpha'^2 $, $R_0-\frac{1}{N}\leq r\leq R_1+\frac{1}{N}$, $R_0<r+h<R_1$. Hence,
\begin{align*}
&\left|
 \int_{ [-\pi,\pi]\setminus (-\varepsilon_N,\varepsilon_N)}w(\alpha')\cos(N\alpha') d\alpha' 
 \right|\\& \leq
 \left|
 \left[
 w(\alpha')\frac{\sin(N\alpha')}{N}
 \right]_{-\varepsilon_N}^{\varepsilon_N}
 -\int_{ [-\pi,\pi]\setminus (-\varepsilon_N,\varepsilon_N)}w'(\alpha')\frac{\sin(N\alpha')}{N} d\alpha' 
 \right|
 \\&\leq
  CN^{1+\gamma}+
  \left|\int_{ [-\pi,\pi]\setminus (-\varepsilon_N,\varepsilon_N)}w'(\alpha')\frac{\sin(N\alpha')}{N} d\alpha' 
  \right|
 \\&\leq
  CN^{1+\gamma}+
   \left|
   \int_{ [-\pi,\pi]\setminus (-\varepsilon_N,\varepsilon_N)}w''(\alpha')\frac{\cos(N\alpha')}{N^2} d\alpha' 
   \right|
    \\&\leq
  CN^{1+\gamma}+\frac{1}{N^2}
   \int_{ [-\pi,\pi]\setminus (-\varepsilon_N,\varepsilon_N)}|w''(\alpha')| d\alpha' 
   \\& \leq 
   CN^{1+\gamma}+
   CN^{1+\gamma}
    \int_{ [-\pi,\pi]\setminus (-\varepsilon_N,\varepsilon_N)}
   \left(
1+r\frac{|(r+h)c_1\alpha'^2|+|h|}{|h^2+2r(r+h)c_2\alpha'^2)|}
\right)
d\alpha'
\\& \leq 
   CN^{1+\gamma}+
   CN^{1+\gamma}
    \int_{ [-\pi,\pi]\setminus (-\varepsilon_N,\varepsilon_N)}
   \left(
1+\frac{|r(r+h)c_1\alpha'^2|}{|h^2+2r(r+h)c_2\alpha'^2)|}
+\frac{|rh|}{|h^2+2r(r+h)c_2\alpha'^2)|}
\right)
d\alpha'
\\&  \leq 
CN^{1+\gamma}+CN^{1+\gamma}\int_{[-\pi,\pi]\setminus (-\varepsilon_N,\varepsilon_N)}\frac{r}{|h|}\frac{1}{(1+\frac{2r(r+h)c_2}{h^2}\alpha'^2)}
 d\alpha'\\& \leq
 CN^{1+\gamma}+\frac{C\sqrt{r}N^{1+\gamma}}{\sqrt{2(r+h)c_2}} 
 \left[\arctan (\frac{\sqrt{2r(r+h)c_2}}{|h|}\alpha')
 \right]_{\varepsilon_N}^\pi\\&
 \leq 
 CN^{1+\gamma},
\end{align*}
using in the last bound that the arc tangent and $r$ are bounded and $r+h$ is far from zero. Integrating in $h\in (-1/N,1/N)$ we obtain the bound 
\begin{align*}
   C \int_{-1/N}^{1/N}
    |
\tilde{g}_N (r+h)\mathcal{S}(r+h,\alpha)N^{1+\gamma}dh
    |
    \leq C\|\tilde{g}_N\|_{L^\infty}N^{\gamma}.
\end{align*}
For the bound of the second part of the sum we integrate by parts twice and we recall the bound for $|w''|$ done in (\ref{cotaw''}), obtaining 
\begin{align*}
    &\left|
 \int_{ [-\pi,\pi]}w(\alpha')\cos(N\alpha') d\alpha' 
 \right|
 \leq
  \left|\int_{ [-\pi,\pi]}w'(\alpha')\frac{\sin(N\alpha')}{N} d\alpha' 
  \right|
 \leq
   \left|
   \int_{ [-\pi,\pi]}w''(\alpha')\frac{\cos(N\alpha')}{N^2} d\alpha' 
   \right|
    \\&\leq
\int_{-\pi}^\pi
\frac{C}{N^2|h|^{3+\gamma}}
\left(
1+r\frac{|(r+h)c_1\alpha'^2|+|h|}{|h^2+2r(r+h)c_2\alpha'^2)|}
\right)
d\alpha'
 \leq
\int_{-\pi}^\pi
\frac{C}{N^2|h|^{3+\gamma}}
\left(
1+\frac{r}{|h|}\frac{1}{(1+\frac{2r(r+h)c_2}{h^2}\alpha'^2)}
\right)d\alpha'
\\&\leq
2C
\frac{1}{N^2|h|^{3+\gamma}}
\left(
\pi+\frac{\sqrt{r}\arctan(\frac{\sqrt{2r(r+h)c_2}}{|h|}\pi)}{\sqrt{2(r+h)c_2}}
\right)
\leq
 \frac{C}{N^2|h|^{3+\gamma}},
\end{align*}
using in the last bound that $r\sim |h|$ when $r\mapsto \infty$ and that $\frac{\arctan(cy)}{y}\sim c$ when $y\mapsto 0$. Now integrating in $h\in (-r,3R_1)\setminus(-1/N,1/N)$ we obtain the desired bound.
\end{proof}
\end{lma}
\begin{lma}
 Given a $C^\infty$ function $\tilde{g}_N:[0,\infty)\rightarrow \mathbb{R}$ with support in the interval $(R_0,R_1)$, $\|\Tilde{g}_N\|_{C^{s_0+3}}\leq \overline{C} N^{-\beta}$ for $s_0\in \mathbb{N}$ and $\alpha_0(r,t) \in C^\infty(\mathbb{R}_+\times[0,\infty))$, let us define $g_N:\mathbb{R}^2\rightarrow \mathbb{R}$ in polar coordinates as 
\[
g_N(r,\alpha, t):= \tilde{g}_N(r)\sin{(N\alpha+N \alpha_0(r,t))},
\]
with $N$ a natural number. Let $\gamma\in (-1,1)$, then 
\begin{align}
    \label{normal2velocidad}
    \|v^\gamma(g_N)\|_{L^2}\leq CN^{-\beta+\gamma}\mathcal{X}(\gamma),
\end{align}
where 
\begin{align*}
    \mathcal{X}(\gamma)=
    \left\{ \begin{array}{lcc} 
    \ln(e+N) & if & \gamma <0 ,
    \\ \\ 1 & if & \gamma \geq0.
    \end{array} \right.
\end{align*}
Furthermore, this can be generalized to the $H^s$ norms with $0\leq s\leq s_0$ as
\begin{align*}
    \|v^\gamma(g_N)\|_{H^s}\leq C N^{-\beta+\gamma+s}\mathcal{X}(\gamma+s).
\end{align*}
Moreover, the result is also true for the homogeneous differential operator $\Lambda^{\alpha }v^\gamma$ with $\alpha\in [0,2]$, $\gamma\in (-1,1)$, this is
\begin{align}
    \label{normahsoperados}
   \|\Lambda^{\alpha }v^\gamma(g_N)\|_{H^s}\leq C N^{-\beta+\alpha+\gamma+s}\mathcal{X}(\gamma+s),
\end{align}
for $0\leq s\leq s_0$.\\\\
The constant $C>0$ just depends on $R_0,R_1,s_0>0$ and $\overline{C}>0$.
\begin{proof}
The Sobolev norms of $g_N$ can be computed as follows. Let $k\in (\mathbb{N}\cup \{0\})\cap [0,s_0+1]$. Then, since our function has compact support,
\begin{align}
\label{auxsobolev}
    \|g_N\|_{H^k}\leq C \|g_N\|_{C^k}\leq CN^{-\beta+k},
\end{align}
with $C$ depending on $R_0,R_1,k>0$ and $\overline{C}>0$. For $k\in [0,s_0+1]$, the same result follows by interpolation.\\\\
The $L^2$ norm of a function of the form $v^\gamma (g_N)$ is easy to calculate when $\gamma \in [0,1)$, since
\begin{align*}
    \|v^\gamma (g_N)\|_{L^2}=\|\mathcal{F}(v^\gamma (g_N))\|_{L^2}=\left\|\frac{|(-\xi_2,\xi_1)|}{|\xi|^{1-\gamma}}\mathcal{F}(g_N)\right\|_{L^2}= \| |\xi|^{\gamma}\mathcal{F}(g_N)\|_{L^2}=\|g_N\|_{\Dot{H}^\gamma}\leq CN^{-\beta+\gamma},
\end{align*}
where the last inequality follows from (\ref{auxsobolev}).\\\\
Analogously, for $0<s\leq s_0$,
\begin{align*}
    \|v^\gamma (g_N)\|_{H^{s}}\leq C(\|v^\gamma (g_N)\|_{L^2}+\|v^\gamma (g_N)\|_{\Dot{H}^{s}})\leq C(\|g_N\|_{\Dot{H}^\gamma}+\|g_N\|_{\Dot{H}^{s+\gamma}})\leq CN^{-\beta+s+\gamma},
\end{align*}
where the last inequality follows from (\ref{auxsobolev}).\\\\
Now, let us compute the $L^2$-norm of a function of the form $v^\gamma (g_N)$ when $\gamma=-\eta<0$. 
\begin{align*}
    \|v^{-\eta}(g_N)\|^2_{L^2(\mathbb{R}^2)}=\|v^{-\eta}(g_N)\|^2_{L^2(B_{2R_1})}+\|v^{-\eta}(g_N)\|^2_{L^2(B_{2R_1}^c)},
\end{align*}
and since $\tilde{g}_N$ has compact support in $B_{R_1}=B_{R_1}(0)\subset \mathbb{R}^2$, then
\begin{align*}
    \|v^{-\eta}(g_N)\|_{L^2(B_{2R_1})}\leq C \|v^{-\eta}(g_N)\|_{L^\infty(B_{2R_1})} ,
\end{align*}
so 
\begin{align}
      \|v^{-\eta}(g_N)\|_{L^2(\mathbb{R}^2)}\leq C\left(
      \|v^{-\eta}(g_N)\|_{L^\infty(B_{2R_1})}+\|v^{-\eta}(g_N)\|_{L^2(B_{2R_1}^c)}
      \right).
\end{align}
\begin{enumerate}
    \item Let $x\in B_{2R_1}$ and let $g_N$ be also the previous function even when expressed in Cartesian coordinates. We also remove the principal value because $K^{-\eta}$ given by (\ref{kernel}) is now integrable at zero. Then using (\ref{vradialmedio}) and (\ref{vangularmedio}), $\forall x\in B_{2R},$
\begin{align*}
    &|v^{-\eta}(g_N)(x, t)|=
    \left|
     C(\gamma) \int_{\mathbb{R}^2} K_{-\eta}(x-y) g_N (y,t) dy
    \right|\\&
    \leq \left|
     C(\gamma) \int_{\{|x-y|<1/N\}} K_{-\eta}(x-y) g_N (y,t) dy\right|  
     + \left|
     C(\gamma) \int_{\{1/N<|x-y|<3R_1\}} K_{-\eta}(x-y) g_N (y,t) dy
    \right| 
    \\&
    \leq 
     C(\gamma) \int_{\{|x-y|<1/N\}} |x-y|^{-2+\eta}|g_N (y,t)| dy  
     + \left|
     C(\gamma) \int_{\{1/N<|x-y|<3R_1\}} K_{-\eta}(x-y) g_N (y,t) dy
    \right|
    \\&
    \leq 
     C(\gamma)\|g_N\|_{L^\infty} N^{-\eta}   
     + \left|
     C(\gamma) \int_{\{1/N<|x-y|<3R_1\}} K_{-\eta}(x-y) g_N (y,t) dy
    \right|
    \\& \leq 
     C N^{-\beta-\eta}+ \left|
     C(\gamma) \int_{\{1/N<|x-y|<3R_1\}} \hat{x}\cdot K_{-\eta}(x-y) g_N (y,t) dy
     \right|\\&\hspace{18mm}+ \left|
     C(\gamma) \int_{\{1/N<|x-y|<3R_1\}} \hat{x}^\perp \cdot K_{-\eta}(x-y) g_N (y,t) dy
     \right|\\&
     \leq CN^{-\beta-\eta}\ln(e+N).
\end{align*}
\item Now the bound for the $L^2$ norm with $x\in B_{2R_1}^c$ is done using directly the Lemma \ref{vradial} and (\ref{vangular}), this is, for $r>2R_1$,
\begin{align*}
    |v^{-\eta}_r(g_N)(r,\alpha, t)|\leq \frac{C  \|   \tilde{g}_N  \|   _{L^\infty}}{N^2|r-R_1|^{2-\eta}}\ln\left(
    e+\frac{r}{(r-R_1)^2}\right), \hspace{6mm}
    |v^{-\eta}_\alpha(g_N)(r,\alpha, t)|\leq \frac{C  \|   \tilde{g}_N\|   _{L^\infty}}{N^2|r-R_1|^{2-\eta}}. 
\end{align*}
Then, as $\|   \tilde{g}_N  \|   _{L^\infty}\leq CN^{-\beta}$,
\begin{align*}
    \int_{2R_1}^\infty\int_{-\pi}^\pi |v^{-\eta}_r(g_N)(r,\alpha)|^2 rd\alpha dr \leq C N^{-2\beta-4}\leq CN^{-2\beta-2\eta},
\end{align*}
and proceeding analogously with $v_\alpha$ we obtain  $\|v^{-\eta}(g_N)\|_{L^2(B_{2R_1}^c)}\leq CN^{-\beta-\eta}$.
\end{enumerate}
Then, putting all this together we conclude $\|v^{\gamma}(g_N)\|_{L^2(\mathbb{R}^2)}=\|v^{-\eta}(g_N)\|_{L^2(\mathbb{R}^2)}\leq CN^{-\beta-\eta}\ln(e+N)=CN^{-\beta+\gamma}\ln(e+N)$ for $\gamma \in (-1,0)$.\\\\
It is easier to calculate the $H^s$-Sobolev norm of $v^{-\eta}(g_N)$ when $s\geq\eta$. Indeed,
\begin{align*}
    \|v^{-\eta}(g_N)\|_{\Dot{H}^{s}}\leq 
    \||\xi|^{s}\mathcal{F}(v^{-\eta} (g_N))\|_{L^2}=\left\||\xi|^{s}\frac{|(-\xi_2,\xi_1)|}{|\xi|^{1+\eta}}\mathcal{F}(g_N)\right\|_{L^2}
    = \|g_N\|_{\Dot{H}^{s-\eta}}\leq CN^{-\beta+s-\eta},
\end{align*}
where the last inequality can be done interpolating between the $L^2$ and the $H^1$ norms of $g_N$, that are easy to estimate. Summing this to the $L^2$ norm and taking into account that any power of $N$ grows faster than $\ln(e+N)$ we are finished. The case $s\in(0,\eta)$ follows by interpolation between the $L^2$ and the $H^\eta$ norms.\\\\
Finally, $\|\Lambda^\alpha v^\gamma(f)\|_{H^s}\leq C\|v^\gamma(f)\|_{H^{s+\alpha}}$ and (\ref{normahsoperados}) follows by applying the result of the velocity in $H^{s+\alpha}$. 
\end{proof}
\end{lma}
\begin{lma}
Given a $C^\infty$ function $\tilde{g}_N:[0,\infty)\rightarrow \mathbb{R}$ with support in the interval $(R_0,R_1)$, $\|\Tilde{g}_N\|_{C^{k_0+2}}\leq \overline{C}N^{-\beta}$ ($ k_0\in \mathbb{N}$) and $\alpha_0(r,t) \in C^\infty(\mathbb{R}_+\times[0,\infty))$ such that $\|\alpha_0\|_{C^{k_0}}\leq \overline{C}$. Let us define $g_N:\mathbb{R}^2\rightarrow \mathbb{R}$ in polar coordinates as 
\[
g_N(r,\alpha, t):= \tilde{g}_N(r)\sin{(N\alpha+N \alpha_0(r,t))},
\]
with $N$ a natural number. \\ \\
Let $0\leq k\leq k_0$, $\alpha \in [0,2]$. Then, for $N$ big enough,
\begin{align}
\label{normafraccionario}
    \|\Lambda^\alpha g_N\|_{C^k}\leq CN^{-\beta+\alpha+k}.
\end{align}
\\
Let $\gamma\in (-1,1)$, $0\leq k\leq k_0$, $\alpha\in[0,2]$. Then, for $N$ big enough, 
\begin{align}
    \label{normackvelocidad}
    \|\Lambda^\alpha v^\gamma(g_N)\|_{C^k}\leq C N^{-\beta+\gamma+\alpha+ k}\ln(e+N).
\end{align}
Here, $C>0$ depends just on $R_0,R_1,k_0$ and $\overline{C}$.
\begin{proof}
Let us proof (\ref{normafraccionario}) for $\alpha \in (0,2)$. Since $g_N\in C^\infty_c$, for $\bold{k}\in \mathbb{N}_ 0^2$,
\begin{align*}
    |\partial^\bold{k} \Lambda^\alpha g_N (x,t)| = |\Lambda^\alpha \partial^\bold{k}  g_N (x,t)|
    =\left|
    \mathcal{K}_\beta P.V. \int_{\mathbb{R}^2}\frac{\partial^\bold{k}g_N(y,t)-\partial^\bold{k}g_N(x,t)}{|y-x|^{2+\alpha}}dy
    \right|\\
    \leq \left|\mathcal{K}_\beta P.V.\int_{B_{1/N}(x)}\frac{\partial^\bold{k}g_N(y,t)-\partial^\bold{k}g_N(x,t)-\partial_{1}\partial^\bold{k}g_N(x,t)(y_1-x_1)-\partial_{2}\partial^\bold{k}g_N(x,t)(y_2-x_2)}{|y-x|^{2+\alpha}}dy \right|\\+
    \mathcal{K}_\beta \int_{\mathbb{R}^2\setminus B_{1/N}(x)}\frac{|\partial^\bold{k}g_N(y,t)-\partial^\bold{k}g_N(x,t)|}{|y-x|^{2+\alpha}}dy\\
    \leq C \|\partial^\bold{k} g_N\|_{C^2} \int_{B_{1/N}(x)}\frac{1}{|x-y|^{\alpha}}dy+ C\|\partial^{\bold{k}}g_N\|_{L^\infty}\int_{\mathbb{R}^2\setminus B_{1/N}(x)} \frac{1}{|x-y|^{2+\alpha}}dy\leq C N^{-\beta+k+\alpha},
\end{align*}
where $\mathcal{K}_\beta=\frac{4^{\frac{\beta}{2}}\Gamma(\frac{2+\beta}{2})}{\pi|\Gamma(-\frac{\beta}{2})|}$ and using that the linear term of the Taylor expansion is odd in $y_i$ and hence the integral vanishes. For $\alpha=0,2$ the proof of (\ref{normafraccionario}) is trivial.\\\\
Let us proof (\ref{normackvelocidad}) with $\gamma\in (-1,1), \alpha=0, k=0$. Let $K_\gamma$ be given by (\ref{kernel}). Let $x\in B_{2R_1}$, using in the second inequality the mean value theorem and in the last one the bounds (\ref{vradialmedio}) and (\ref{vangularmedio}),
\begin{align*}
    &|v^{\gamma}(g_N)(x, t)|
    \leq \left|
     C(\gamma) \int_{\{|x-y|<1/N\}} K_{\gamma}(x-y) (g_N (x,t)-g_N (y,t)) dy\right|  
      \\&
     + \left|
     C(\gamma) \int_{\{1/N<|x-y|<3R_1\}} K_{\gamma}(x-y) g_N (y,t) dy
    \right| 
    \\&
    \leq 
     C(\gamma)\|g_N\|_{C^1} N^{\gamma-1}   
     + \left|
     C(\gamma) \int_{\{1/N<|x-y|<3R_1\}} K_{\gamma}(x-y) g_N (y,t) dy
    \right|
    \\& \leq 
     C N^{-\beta+\gamma}+ \left|
     C(\gamma) \int_{\{1/N<|x-y|<3R_1\}} \hat{x}\cdot K_{\gamma}(x-y) g_N (y,t) dy
     \right|\\&\hspace{18mm}+ \left|
     C(\gamma) \int_{\{1/N<|x-y|<3R_1\}} \hat{x}^\perp \cdot K_{\gamma}(x-y) g_N (y,t) dy
     \right|\\&
     \leq CN^{-\beta+\gamma}\ln(e+N).
\end{align*}
For $x\in B_{2R_1}^c$ we use directly Lemma \ref{vradial} and (\ref{vangular}), this is
\begin{align*}
    |v^{\gamma}_r(g)(r,\alpha)|\leq \frac{C  \|   \tilde{g}_N  \|   _{L^\infty}}{N^2|r-R_1|^{2+\gamma}}\ln\left(
    e+\frac{r}{(r-R_1)^2}\right), \hspace{6mm}
    |v^{\gamma}_\alpha(g)(r,\alpha)|\leq \frac{C  \|   \tilde{g}_N\|   _{L^\infty}}{N^2|r-R_1|^{2+\gamma}}, \hspace{6mm} \forall r>2R_1,
\end{align*}
and taking supremum in $r>2R_1$ we obtain the result.\\ \\
For the case $k\geq 1$ we just have to take into account that, for functions regular enough, $\partial^{\bold{j}}v_i(f)=v_i(\partial^{\bold{j}}f)$ and repeat the reasoning. 
\begin{align*}
      \partial_{x_1}v_j^\gamma (g_N)= v_j^\gamma(\partial_{x_1}\tilde{g}_N(r) \sin (N\alpha+N\alpha_0(r,t)))+v_j^\gamma(\tilde{g}_N \partial_{x_1}\sin (N\alpha+N\alpha_0(r,t)))\\
      =v_j^\gamma(\partial_{r}\tilde{g}_N(r) \cos(\alpha) \sin (N\alpha+N\alpha_0(r,t)))
      +v_j^\gamma(\tilde{g}_N(r) \cos (\alpha) \partial_r \alpha_0(r,t) \cos (N\alpha+N\alpha_0(r,t)))\\
      +v_j^\gamma
      \left(\tilde{g}_N(r) \frac{\sin (\alpha)}{r}N\sin (N\alpha +N\alpha_0(r,t))
      \right)\\
      =\frac{1}{2}v_j^\gamma(\partial_{r}\tilde{g}_N(r)  \sin ((N-1)\alpha+N\alpha_0(r,t)))
      +\frac{1}{2}v_j^\gamma(\partial_{r}\tilde{g}_N(r)  \sin ((N+1)\alpha+N\alpha_0(r,t)))
      \\ +\frac{1}{2}v_j^\gamma(\tilde{g}_N(r)  \partial_r \alpha_0(r,t) \cos ((N-1)\alpha+N\alpha_0(r,t)))
      +\frac{1}{2}v_j^\gamma(\tilde{g}_N(r)  \partial_r \alpha_0(r,t) \cos ((N+1)\alpha+N\alpha_0(r,t)))
      \\ +\frac{1}{2}v_j^\gamma
      \left( \frac{\tilde{g}_N(r)}{r}N\sin ((N-1)\alpha +N\alpha_0(r,t))
      \right)
      -\frac{1}{2}v_j^\gamma
      \left( \frac{\tilde{g}_N(r)}{r}N\sin ((N+1)\alpha +N\alpha_0(r,t))
      \right).
\end{align*} 
Notice that we recover the type of functions that we had before. If we want to have the same factor $N\pm1$ also multiplying the phase $\alpha_0$ inside the sines or cosines, we just have to apply the sum and difference rules, for example, 
\begin{align*}
    \sin{((N+1)\alpha+N\alpha_0(r,t))}=\cos{(\alpha_0(r,t))}
    \sin{((N+1)\alpha+(N+1)\alpha_0(r,t))} \\
    -\sin{(\alpha_0(r,t))}\cos{((N+1)\alpha+(N+1)\alpha_0(r,t))}.
\end{align*}
Summarizing, we have a sum of functions of the form: radial function (which $C^{k_0+1}$ norm is bounded by $CN^{-\beta}$ with $C>0$ not depending on $N$) times a sine or cosine (which is a sine except for a phase that we can add to $\alpha_0$) with argument $(N\pm1)(\alpha+\text{phase}(r,t))$. Although we made this for the partial derivative with respect to the first coordinate, the same result holds for the second. Also, although we made just an example of one derivative, $k$ derivatives can be made iterating the process. Since for $N$ big enough, $(N\pm k_0)\sim N$, then
\begin{align*}
    \|v^\gamma (g_N)\|_{C^k}\leq CN^{-\beta +k},
\end{align*}
for $1\leq k\leq k_0$. \\\\
The proof can be generalized to the operator $\Lambda^\alpha v^{\gamma} (g_N)$ with $\alpha\in (0,2]$, since $g_N$ is $C^\infty_c$ and hence $\Lambda^\alpha v^{\gamma} (g_N)=\Lambda^{\alpha+\gamma-1}\nabla^\perp g_N$. \\\\
We first bound the case $\alpha+\gamma>1$. Let $\varphi=\partial_{x_i}g_N$. Then, the components of $\Lambda^{\alpha+\gamma-1}\nabla^\perp g_N$ are given by 
\begin{align*}
   |\Lambda^{\alpha+\gamma-1} \varphi (x,t)|\leq 
   C\left|
     P.V.\int_{\{|x-y|<1/N\}}\frac{\varphi (y,t)-\varphi (x,t)}{|x-y|^{1+\alpha+\gamma} } dy\right| +
      C\left|
     \int_{\{|x-y|>1/N\}}\frac{\varphi (y,t)-\varphi (x,t)}{|x-y|^{1+\alpha+\gamma} } dy\right|
     \\ 
     = C\left|
      P.V.\int_{\{|x-y|<1/N\}} \frac{\varphi (y,t)-\varphi (x,t)-\partial_{x_1}\varphi (x,t)(y_1-x_1)-\partial_{x_2}\varphi (x,t)(y_2-x_2)}{|x-y|^{1+\alpha+\gamma}} dy\right| \\
      + C\left|
     \int_{\{|x-y|>1/N\}}\frac{\varphi (y,t)-\varphi (x,t)}{|x-y|^{1+\alpha+\gamma} } dy\right|
     \\  
    \leq C\|\varphi\|_{C^2} \int_{\{|x-y|\leq 1/N\}}|x-y|^{1-\gamma-\alpha}dy
    +C\|\varphi\|_{L^\infty} \int_{\{|x-y|>1/N\}}|x-y|^{-1-\alpha-\gamma}dy
    \\
    \leq C\|\varphi\|_{C^2}N^{\alpha+\gamma-3}
    +C \|\varphi\|_{L^\infty}N^{\alpha+\gamma-1}
    \leq CN^{\alpha+\gamma}.
\end{align*}
We used that the linear term of the Taylor expansion is odd in $y_i$ and hence the integral vanishes. \\\\
The case $\alpha+\gamma=1$ is trivial since $\Lambda^\alpha v^\gamma (g_N)=\nabla^\perp g_N$. For the case $-1<\alpha+\gamma<1$, we must use the fundamental solution of the inverse fractional laplacian,
\begin{align*}
    |\Lambda^{\alpha+\gamma-1}\varphi (x,t)| \leq C\left|
     \int_{\{|x-y|<1/N\}}\frac{\varphi (y,t)}{|x-y|^{1+\alpha+\gamma} } dy\right| +
     C\left|
     \int_{\{|x-y|>1/N\}}\frac{\varphi (y,t)}{|x-y|^{1+\alpha+\gamma} } dy\right|
     \\ \leq  C\|\varphi\|_{L^\infty} \int_{\{|x-y|<1/N\}} |x-y|^{-1-\alpha-\gamma} dy
     +
     C\left|
     \int_{\{|x-y|>1/N\}}\frac{\varphi (y,t)}{|x-y|^{1+\alpha+\gamma} } dy\right|
      \\ \leq  CN^{\alpha+\gamma}
     +C\left|\int_{\partial B_{1/N}(x)} N^{1+\alpha+\gamma} g_N(y)dy_j
     -
     \int_{\{|x-y|>1/N\}}\frac{x_i-y_i}{|x-y|^{3+\alpha+\gamma}}g_N(y)dy
     \right|\\ 
     \leq  CN^{\alpha+\gamma}
     +C\left|
     \int_{\{|x-y|>1/N\}}\frac{x_i-y_i}{|x-y|^{3+\alpha+\gamma}}g_N(y)dy
     \right|,
\end{align*}
where we integrated by parts in the $i$-th coordinate. \\\\
Now we just have to bound this last integral. Notice that the proof is completely analogous to the one of the bound of the velocity (\ref{normal2velocidad}). This is, when $x\in B_{2R_1}(x)$,
\begin{align*}
    \left|
     \int_{\{|x-y|>1/N\}}\frac{(x-y)^\perp}{|x-y|^{3+\alpha+\gamma}}g_N(y)dy
     \right|\leq \left|
     \int_{\{1/N<|x-y|<3R_1\}}\hat{x} \cdot\frac{(x-y)^\perp}{|x-y|^{3+\alpha+\gamma}}g_N(y)dy
     \right|\\
     +\left|
     \int_{\{1/N<|x-y|<3R_1\}}\hat{x}^\perp \cdot \frac{(x-y)^\perp}{|x-y|^{3+\alpha+\gamma}}g_N(y)dy
     \right|\leq  CN^{\alpha+\gamma} \ln {N},
\end{align*}
using that $\alpha+\gamma \in (-1,1)$ and (\ref{vradialmedio}), (\ref{vangularmedio}). For $x\in B_{2R_1}(x)^c$, we use Lemma \ref{vradial} and (\ref{vangular}) to conclude 
\begin{align*}
    \left|
     \int_{\{|x-y|>1/N\}}\frac{(x-y)^\perp}{|x-y|^{3+\alpha+\gamma}}g_N(y)dy
     \right|\leq |v^{\alpha+\gamma}g_N(x)|\leq CN^{-2} \leq  CN^{\alpha+\gamma}.
\end{align*}
\end{proof}
\end{lma}
\subsection{Pseudo-solutions for gSQG}
\label{seccionpseudosol}
We say that a function $\overline{\theta}$ is a pseudo-solution of the $\gamma$-gSQG equation if it fulfils
  \begin{equation*}
\left\{\begin{array}{l}
\frac{\partial \overline{\theta}}{\partial t}+v^\gamma\cdot \nabla \overline{\theta}+F(x,t)=0 ,\\ \\
v^\gamma=\nabla^{\perp} \psi_\gamma=\left(\partial_{2} \psi_\gamma,-\partial_{1} \psi_\gamma\right), \quad \psi_\gamma=-\Lambda^{-1+\gamma} \overline{\theta} , \\ \\
\overline{\theta}(\cdot,0)=\theta_0(\cdot).
\end{array}\right.
\end{equation*}
Notice that, choosing the appropriate source term $F$, every regular enough function is a pseudo-solution. The key point is to try to make $F$ small in the needed norms.  We will consider pseudo-solutions of the form
\[
\overline{\theta}_{N,c,K}(r,\alpha,t)=f_{1,c,K}(r)+f_{2,c,K}(r)r_{c,K}^\beta \frac{\sin\left(N\alpha-Nt\frac{v^\gamma_\alpha(f_{1,c,K}(r))}{r}\right)}{N^\beta},
\]
with $f_{1,c,K}$ given by the Lemma \ref{lemaf1} for some values $c,K>0$, where $K>1$ will be big and $0<c<1$ small. The value of $r_{c,K}$ is given in the Lemma \ref{lemaf1} by the relation $r_{c,K}=2a_1$. Notice that, by continuity and Lemma \ref{lemaf1}, there exist an interval $[r_{c,K}-\epsilon_{c,K},r_{c,K}+\epsilon_{c,K}]$ such that, if $r\in [r_{c,K}-\epsilon_{c,K},r_{c,K}+\epsilon_{c,K}]$, then 
\begin{equation}
\label{crecimiento}
\left|
\frac{\partial\frac{v^\gamma_\alpha(f_{1,c,K})(r)}{r}}{\partial r}(r)
\right|
\geq \frac{K}{2r}.
\end{equation}
Regarding $f_{2,c,K}$, we choose a $C^\infty$ function with support in $[r_{c,K}-\epsilon_{c,K},r_{c,K}+\epsilon_{c,K}]\cap[\frac{2r_{c,K}}{3},\frac{3r_{c,K}}{2}]$, such that $ \|  f_{2,c,K}  \| _{L^2}=c$.\\\\
Notice how the coronas that contain the supports of $f_{1,c,K}$ and $f_{2,c,K}$ are disposed: they are disjoint, the support of $f_{1,c,K}$ is the inner one and the absolute value of the radial derivative of $v^\gamma_\alpha(f_{1,c,K})/r$ is big where $f_{2,c,K}$ is non-zero.  More specifically $supp(f_{1,c,K})(r)\subset [a_0,\frac{r_{c,K}}{2}]$ and $supp(f_{2,c,K})(r)\subset [\frac{2r_{c,K}}{3}, \frac{3r_{c,K}}{2}]$. Also notice that the supports are bounded uniformly on $K>1$, $0<c<1$ by construction (check Lemma \ref{lemaf1}).\\ \\
As a last remark, notice that for some high $k\in\mathbb{N}$, the $C^k$ norms of $f_{1,c,K}$ and $f_{2,c,K}$ can grow when $K>1$ grows or when $0<c<1$ goes to zero. Those norms will appear in the following lemmas, in which we try to bound $F_{N,c,k,\gamma}$ and the difference between the solution and the pseudosolution in terms of $N>0$. Hence, the bounds of the Lemmas \ref{FnormL2}, \ref{lemanormafuerzahk}, \ref{cotadiferencia} and \ref{cotaTheta} will not be uniform on $K>1$ and $0<c<1$. Nevertheless, that will not be a problem to prove the ill-posedness result (Theorem \ref{teorema}) as long as we choose the parameters in the proper order.
\\
\begin{rmk}
    Although we do not indicate it, the pseudosolutions depend also on $\gamma\in(-1,1)$. The dependence is hidden, fist of all, in the angular velocity of the radial function inside the sine and secondly in the $\beta$ playing the  role of the exponent of the $N$ in the  denominator, since $\beta$ is confined in an interval dependent on $\gamma$.
\end{rmk}
We define $F_{N,c,K,\gamma}$ as the source term associated to the pseudo-solution $\overline{\theta}_{N,c,K}(r,\alpha,t)$ with respect to the $\gamma$-gSQG. Notice that 
\begin{align*}
    F_{N,c,K,\gamma}= -\frac{v_{\alpha}^{\gamma}(\overline{\theta}_{N,c,K}-f_{1,c,K})}{r} \frac{\partial \overline{\theta}_{N,c,K}}{\partial \alpha}-v_r^\gamma(\overline{\theta}_{N,c,K})\frac{\partial \overline{\theta}_{N,c,K}}{\partial r}\\ 
    = -\frac{v_{\alpha}^{\gamma}(\overline{\theta}_{N,c,K}-f_{1,c,K})}{r} \frac{\partial \overline{\theta}_{N,c,K}}{\partial \alpha}-v_r^\gamma(\overline{\theta}_{N,c,K}-f_{1,c,K})\frac{\partial \overline{\theta}_{N,c,K}}{\partial r}.
\end{align*}
\\ 
\begin{lma}
\label{FnormL2}
For $\gamma \in [-1,1)$, $\beta \in [1,2+\gamma)\cap(\frac{3}{2}+\gamma, 2+\gamma)$, $t\in [0,T]$ and the pseudo-solution $\overline{\theta}_{N,c,K}$, the source term $F_{N,c,K,\gamma}$ satisfies
\[
  \|  F_{N,c,K,\gamma}  \|  _{L^2} \leq \frac{C\mathcal{X}(\gamma)}{N^{2\beta-1-\gamma}},
\]
with $C$ depending on $c,K,T, \gamma$ and $\mathcal{X}$ being 
\begin{align}
\label{funcionsalto}
    \mathcal{X}(\gamma)=
    \left\{ \begin{array}{lcc} 
    \ln(e+N) & if & \gamma <0 ,
    \\ \\ 1 & if & \gamma \geq0.
    \end{array} \right.
\end{align}
\begin{proof}
We divide the source term in different parts.\\ \\
For the first inequality, using that the pseudo-solution has compact support in a corona independent on N, we can bound the $1/r$ factor and by Hölder's inequality
\[
  \|  
\frac{v_{\alpha}^{\gamma}(\overline{\theta}_{N,c,K}-f_{1,c,K})}{r} \frac{\partial \overline{\theta}_{N,c,K}}{\partial \alpha}
 \|  _{L^2}\leq
C  \|  v_{\alpha}^{\gamma}(\overline{\theta}_{N,c,K}-f_{1,c,K})  \|  _{L^2}
  \|  \frac{\partial \overline{\theta}_{N,c,K}}{\partial \alpha}
  \|  _{L^\infty}
\leq CN^{-(2\beta-1-\gamma)}\mathcal{X}(\gamma),
\]
where we also used the bound (\ref{normal2velocidad}).
\\ \\
Analogously, using again the bound (\ref{normal2velocidad}),
\begin{align*}
  \|  
v_{r}^{\gamma}(\overline{\theta}_{N,c,K}-f_{1,c,K})\frac{\partial (\overline{\theta}_{N,c,K}-f_{1,c,K})}{\partial r}
 \|  _{L^2} \\
\leq
  \|  
v_{r}^{\gamma}(\overline{\theta}_{N,c,K}-f_{1,c,K})  \|  _{L^2}
 \|  
\frac{\partial (\overline{\theta}_{N,c,K}-f_{1,c,K})}{\partial r} \|  _{L^\infty} \leq CN^{-(2\beta-1-\gamma)}  \mathcal{X}(\gamma).
\end{align*}
Recall that $supp(f_{1,c,K})\subset [a_0,\frac{r_{c,K}}{2})$ and that $\partial_r f_{1,c,K}$ is bounded uniformly on $N>1$. Also $supp(\overline{\theta}_{N,c,K}-f_{1,c,K})\subset supp(f_{2,c,K})\subset[\frac{2r_{c,K}}{3},\frac{3r_{c,K}}{2}]$ and $  \|  \overline{\theta}_{N,c,K}-f_{1,c,K}  \|  _{L^\infty} \leq \frac{C}{N^\beta}$. Then, using the Lemma \ref{vradial} we have
\begin{align*}
 \|  
v_{r}^{\gamma}(\overline{\theta}_{N,c,K}-f_{1,c,K})\frac{\partial f_{1,c,K}}{\partial r} 
 \|  _{L^2}
\leq 
\left(
\int^{\frac{r_{c,K}}{2}}_{a_0} \frac{C}{N^{2+2\beta}(\frac{2r_{c,K}}{3}-r)^{4+2\gamma}}rdr 
\right)^\frac{1}{2}
\leq CN^{-(1+\beta)} \leq CN^{-(2\beta-1-\gamma)},
\end{align*}
using in this last inequality that $\beta<2+\gamma$.\\ \\
Combining the bounds we get the desired result. 
\end{proof}
\end{lma}

\begin{lma}
\label{lemanormafuerzahk}
For $\gamma \in [-1,1)$, $\beta \in [1,2+\gamma)\cap(\frac{3}{2}+\gamma, 2+\gamma)$, $t\in [0,T]$ and the pseudo-solution $\overline{\theta}_{N,c,K}$, the source term $F_{N,c,K,\gamma}$ satisfies, for $k\in \mathbb{N}$, $k\geq 2$,
\begin{align}
\label{normafuerzahk}
  \|  F_{N,c,K,\gamma}  \|  _{H^k} \leq \frac{C\mathcal{X}(\gamma)}{N^{2\beta-1-k-\gamma}},
\end{align}
with $C$ depending on $k,c,K,T$ and $\gamma$ and with $\mathcal{X}(\gamma)$ given by (\ref{funcionsalto}).
\begin{proof}
Just as we did in the last lemma, we split the computation in three parts.\\ \\
We use the derivative of the product and the velocity bound given by (\ref{normahsoperados}),
\[
  \|  
\frac{v_{\alpha}^{\gamma}(\overline{\theta}_{N,c,K}-f_{1,c,K})}{r} \frac{\partial \overline{\theta}_{N,c,K}}{\partial \alpha}
  \|  _{\dot{H}^k}\leq 
C\sum_{i=0}^{k} \|  \frac{\partial \overline{\theta}_{N,c,K}}{\partial \alpha}  \|  _{C^i}
  \|  
v_{\alpha}^{\gamma}(\overline{\theta}_{N,c,K}-f_{1,c,K}) 
 \|  _{\dot{H}^{k-i}} 
\]
\[
\leq
C\sum_{i=0}^{k}N^{-(\beta-1-i)}N^{-(\beta-k+i-\gamma)}\mathcal{X}(\gamma+k) \leq CN^{-(2\beta -1-k-\gamma)} \mathcal{X} (\gamma).
\]
The second bound is done analogously to the first using the bound (\ref{normahsoperados}),
\begin{align*}
  \|  \frac{\partial (\overline{\theta}_{N,c,K}-f_{1,c,K})}{\partial r} v_r^{\gamma}(\overline{\theta}_{N,c,K}-f_{1,c,K}) \|  _{\dot{H}^k}
\\
\leq
C\sum_{i=0}^{k} 
 \|  \frac{\partial (\overline{\theta}_{N,c,K}-f_{1,c,K})}{\partial r} \|  _{C^i}
  \|  v_r^{\gamma}(\overline{\theta}_{N,c,K}-f_{1,c,K})  \|  _{\dot{H}^{k-i}} \leq C N^{-(2\beta -1-k-\gamma)}\mathcal{X} (\gamma).
\end{align*}
Finally, we have to bound $  \|  \frac{\partial f_{1,c,K}}{\partial r} v_r^{\gamma}(\overline{\theta}_{N,c,K}-f_{1,c,K})  \|  _{\Dot{H}^k}$. Let $\bold{i},\bold{j}$ denote vectors of the form $\bold{j}=(j_1,j_2)\in \mathbb{N}_0^2$. Hence,
\[
  \|  \frac{\partial f_{1,c,K}}{\partial r} v_r^{\gamma}(\overline{\theta}_{N,c,K}-f_{1,c,K})  \|  _{\Dot{H}^k}\leq  
 \sum_{|\bold{j}|=k}   \|  \partial^{\bold{j}}(\frac{\partial f_{1,c,K}}{\partial r} v_r^{\gamma}(\overline{\theta}_{N,c,K}-f_{1,c,K})) \|  _{L^2}
\]
\[
\leq \sum_{|\bold{i}+\bold{j}|=k}
\|  \partial^\bold{j}(\frac{\partial f_{1,c,K}}{\partial r})
\partial^\bold{i} v_r^{\gamma}(\overline{\theta}_{N,c,K}-f_{1,c,K})\|  _{L^2}
\leq C \sum_{|\bold{i}|=k}
\| 
\partial^\bold{i} v_r^{\gamma}(\overline{\theta}_{N,c,K}-f_{1,c,K})\|  _{L^2(supp(f_{1,c,K}))}
\]
\[
\leq \left(
\int_{a_0}^{\frac{r_{c,K}}{2}}
\frac{CN^{2k}}{N^{2+2\beta}(\frac{2r_{c,K}}{3}-r)^{4+2\gamma}}rdr 
\right)^{\frac{1}{2}}\leq CN^{k-1-\beta}\leq  CN^{-(2\beta-1-k-\gamma)},
\]
using the Lemma \ref{vradial}, $\beta<2+\gamma$ and the disposal of the supports of the radial functions, more precisely $supp(f_{1,c,K})\subset [a_0,\frac{r_{c,K}}{2})$ and $supp(\overline{\theta}_{N,c,K}-f_{1,c,K})\subset supp(f_{2,c,K})\subset[\frac{2r_{c,K}}{3},\frac{3r_{c,K}}{2}]$. \\\\
To obtain the final non-homogeneous bound we just have to sum the $L^2$ bound obtained in the Lemma \ref{FnormL2} and the homogeneous one, noticing that the homogeneous is dominating.
\end{proof}
\end{lma}
Now, applying interpolation for Sobolev spaces with, for example, $L^2$ and $H^4$, we obtain the following corollary.
\vspace{3mm}
\begin{coro}
\label{FnormH}
For $\gamma \in [-1,1)$, $\beta \in [1,2+\gamma)\cap(\frac{3}{2}+\gamma, 2+\gamma)$, $t\in [0,T]$ and a pseudo-solution $\overline{\theta}_{N,c,K}$, the source term $F_{N,c,K,\gamma}$ satisfies
\[
  \|  F_{N,c,K,\gamma}  \|  _{H^{\beta+\frac{1}{2}}} \leq \frac{C\mathcal{X}(\gamma)}{N^{\beta-\frac{3}{2}-\gamma}},
\]
with $C$ depending on $c,K,T, \gamma$ and with $\mathcal{X}(\gamma)$ given by (\ref{funcionsalto}).
\end{coro}
To assure local existence of a solution to the $\gamma$-gSQG equation (\ref{gSQG}), we must consider the problem in $H^{2+\gamma+\epsilon}$, with $\epsilon>0$. Notice that whenever we consider a solution in the following steps we will do it in $H^{\beta+\frac{1}{2}}$, so we will be in the well-posedness regime. \\ \\
The condition $\beta<2+\gamma$ appears analytically in the proof of the Lemma \ref{lemaf1}, but it can be also justified by the fact that the strong ill-posedness we are proving would contradict local existence of solutions, result that holds when $\beta>2+\gamma$. We will not deal with the critical case $\beta=2+\gamma$.\\ \\
On the other hand, in the proof we will see that, to be able to close the estimates, they must be done in a Sobolev space $H^s$ where local existence holds. This is, $s> 2+\gamma$. On the other hand, the source term must converge to zero in that space, so (check (\ref{normafuerzahk}) and interpolate) $ 2\beta-1-s-\gamma>0$. Combining those inequalities we obtain that $\beta > \frac{3}{2}+\gamma$. This leads us to consider $\beta$ in the interval $(\frac{3}{2} +\gamma,2+\gamma)$.\\ \\
 We will find also the restriction $\beta\geq 1$ (check Lemma \ref{cotadiferencia}) and as a consequence the Euler case $\gamma=-1$ will not be included anymore. \\ \\
So, given $\gamma \in (-1,1)$, $\beta \in [1,2+\gamma)\cap (\frac{3}{2}+\gamma,2+\gamma)$, let us consider $\theta_{N,c,K,\gamma}$ the unique classical solution in $H^{\beta+\frac{1}{2}}$ to the $\gamma$-gSQG equation (\ref{gSQG}) with initial conditions $\theta_{N,c,K,\gamma}(x,0)=\overline{\theta}_{N,c,K}(x,0)\in H^{\beta+\frac{1}{2}}$.\\ \\
We define the difference between the solution and the  pseudosolution as 
\begin{equation}
\label{defdiferencia}
\Theta_{N,c,K,\gamma}:=\theta_{N,c,K,\gamma}-\overline{\theta}_{N,c,K}.
\end{equation}
In order to conclude that the growth over time of $\overline{\theta}_{N,c,K}$ and $\theta_{N,c,K,\gamma}$ in $H^{\beta}$ are similar, it is necessary to find bounds for the $H^\beta$ norm of $\Theta_{N,c,K,\gamma}$ in some period of time.\\ \\
Firstly, notice that $\Theta_{N,c,K,\gamma}$ solves the equation 
\begin{align*}
\frac{\partial \Theta_{N,c,K,\gamma}}{\partial t} + v_1^\gamma(\Theta_{N,c,K,\gamma}) \frac{\partial \Theta_{N,c,K,\gamma}}{\partial x_1}+ v_2^\gamma(\Theta_{N,c,K,\gamma}) \frac{\partial \Theta_{N,c,K,\gamma}}{\partial x_2}
\\
+v_1^\gamma(\Theta_{N,c,K,\gamma}) \frac{\partial \overline{\theta}_{N,c,K}}{\partial x_1}+ v_2^\gamma(\Theta_{N,c,K,\gamma}) \frac{\partial \overline{\theta}_{N,c,K}}{\partial x_2}
\\
+ v_1^\gamma(\overline{\theta}_{N,c,K}) \frac{\partial \Theta_{N,c,K,\gamma}}{\partial x_1}+ v_2^\gamma(\overline{\theta}_{N,c,K}) \frac{\partial \Theta_{N,c,K,\gamma}}{\partial x_2}- F_{N,c,K,\gamma}=0  .
\end{align*}
Now, before proving the Sobolev bounds for $\Theta_{N,c,K,\gamma}$, we state some technical lemmas. These first two can be found in \cite{kato}.
\vspace{3mm}
\begin{lma}
\label{leibniz}
    Let $s>0$. Then for any $s_1,s_2\geq 0$ such that $s=s_1+s_2$ and any $f,g \in \mathcal{S}(\mathbb{R}^2)$, the following holds
    \[
     \|  D^s(fg)-
    \sum_{|\bold{k}|\leq s_1}\frac{1}{\bold{k}!}\partial^{\bold{k}}fD^{s,\bold{k}}g-
    \sum_{|\bold{j}|\leq s_2}\frac{1}{\bold{j}!}\partial^{\bold{j}}gD^{s,\bold{j}}f
    \|  _{L^2}
    \leq C  \|  D^{s_1}f  \|  _{L^2}  \|  D^{s_2}g  \|  _{BMO},
    \]
    where $\bold{j}\in \mathbb{N}_0^2$ and $\bold{k}\in \mathbb{N}_0^2$ are multi-indexes, $\partial^{\bold{j}}=\frac{\partial}{\partial x_1^{j_1}\partial x_2^{j_2}}$, $\partial^{\bold{j}}_\xi = \frac{\partial}{\partial \xi_1^{j_1} \partial \xi_2^{j_2}}$ and $D^{s,\bold{j}}$ is defined on the Fourier side as 
    \begin{equation*}
                \widehat{D^{s,\bold{j}}f} ( \xi )= \widehat{D^{s,\bold{j}}}(\xi)\hat{f}(\xi),
    \end{equation*}
    where 
    \begin{equation*}
        \widehat{D^{s,\bold{j}}}(\xi)= i^{-|\bold{j}|} \partial_{\xi}^{\bold{j}} (|\xi|^s).
    \end{equation*}
\end{lma}
\vspace{3.5mm}
\begin{lma}
\label{leibnizv}
    Let $\gamma \in(0,1)$. Then for any two functions in the Schwartz class $f,g \in \mathcal{S}(\mathbb{R}^2)$ and $A^\gamma$ an homogeneous differential operator of order $\gamma>0$, the following holds
   \[ 
     \|  A^\gamma_i(fg)- A^\gamma_i(f)g-f A^\gamma_i(g)  \|  _{L^2} \leq C
     \|  f  \|  _{L^2}   \|  D^\gamma g  \|  _{BMO},
   \]
   where $D^\gamma$ is the operator defined on the Fourier domain as $\widehat{D^\gamma f}(\xi)=|\xi|^\gamma \Hat{f}(\xi)$.
\end{lma}
\begin{lma}
\label{gammanegativa}
    Let $f\in L^2$ and supported in $B_R(0)\subset \mathbb{R}^2$, $s\in (0,1)$. Then $\|v^{-s}(f)\|_{L^2}\leq C\|f\|_{L^2}$ where $C=C(R,s)>0$.
    \begin{proof}
        \begin{align*}
            \|v^{-s}(f)\|_{L^2}^2=\int_{\mathbb{R}^2}|v^{-s}(f)(x)|^2 dx= \int_{B_{2R}(0)}|v^{-s}(f)(x)|^2 dx+ \int_{\mathbb{R}^2 \setminus B_{2R}(0)}|v^{-s}(f)(x)|^2 dx.
        \end{align*}
        Firstly,
        \begin{align*}
            \int_{B_{2R}(0)}|v^{-s}(f)(x)|^2 dx =\int_{B_{2R}(0)}
            \left|C
            \int_{\mathbb{R}^2}\frac{(x-y)^{\perp}}{|x-y|^{3-s}}f(y)dy
            \right|^2
            dx
            \leq 
            \int_{B_{2R}(0)}
            \left(
            \int_{B_{R}(0)}
            \frac{|f(y)|}{|x-y|^{2-s}}dy
            \right)^2
            dx\\
            \leq
            \int_{B_{2R}(0)} 
            \left(
                \int_{B_{R}(0)} 
                    \frac{|f(y)|^2}{|x-y|^{2-s}} dy
                    \right)
                     \left(
                \int_{B_{R}(0)} 
                    \frac{1}{|x-y|^{2-s}} dy
                    \right)dx\leq
                    C
             \int_{B_{2R}(0)} 
             |f(y)|^2
            \left(
                \int_{B_{R}(0)} 
                    \frac{1}{|x-y|^{2-s}} dx
                    \right)dy        
            \\\leq C \|f\|_{L^2}^2,
        \end{align*}
        where $C>0$ depends on the $L^1$ norm of $K^{-s}$ given by (\ref{kernel}) in the ball $B_{4R}(0)$ .\\
        Notice that, fixed $x\in \mathbb{R}^2\setminus B_{2R}(0)$, the $y\in B_{R}(0)$ minimizing the distance between $x$ and $y$ is $y=R\frac{x}{|x|}$, so
        \begin{align*}
            \int_{\mathbb{R}^2 \setminus B_{2R}(0)}|v^{-s}(f)(x)|^2 dx=\int_{\mathbb{R}^2\setminus B_{2R}(0)}
            \left|C
            \int_{B_{R}(0)}\frac{(x-y)^{\perp}}{|x-y|^{3-s}}f(y)dy
            \right|^2
            dx\\
            \leq \int_{\mathbb{R}^2\setminus B_{2R}(0)}
            \left(C
            \int_{B_{R}(0)}\frac{|f(y)|}{|x-y|^{2-s}}dy
            \right)^2
            dx
            \leq 
            \int_{\mathbb{R}^2\setminus B_{2R}(0)}
            \left(C
            \int_{B_{R}(0)}\frac{|f(y)|}{|x-R\frac{x}{|x|}|^{2-s}}dy
            \right)^2
            dx
            \\
            \leq 
            C
            \int_{\mathbb{R}^2\setminus B_{2R}(0)}
            \left(
            \int_{B_{R}(0)}\frac{|f(y)|}{|x|^{2-s}(1-\frac{R}{|x|})^{2-s}}dy
            \right)^2
            dx=C
            (\int_{\mathbb{R}^2\setminus B_{2R}(0)}
            \frac{1}{|x|^{4-2s}}dx)
            \left|
            \int_{B_{R}(0)}|f(y)|dy
            \right|^2\\
            \leq C\|f\|_{L^1(B_R(0))}^2 \leq C \|f\|_{L^2}^2,
        \end{align*}
        noticing that $1-\frac{R}{|x|}\geq \frac{1}{2}$ and that the last inequality follows by Hölder. 
    \end{proof}
\end{lma}
\begin{lma}
\label{cotadiferencia}
If $\gamma \in (-1,1)$, $\beta \in [1,2+\gamma)\cap (\frac{3}{2}+\gamma,2+\gamma)$ and $\theta_{N,c,K,\gamma}\in H^{\beta+\frac{1}{2}}$ exist for $t \in [0,T]$, then 
\[
  \|  \Theta_{N,c,K,\gamma}  \|  _{L^2} \leq \frac{C\mathcal{X}(\gamma)t}{N^{2\beta-1-\gamma}}, \hspace{3mm} \forall t \in [0,T],
\]
where $C$ depends on $c$, $K$, $\gamma$ and $T$ and with $\mathcal{X}(\gamma)$ given by (\ref{funcionsalto}).
\begin{proof}
    The proof for $\gamma=0$ is done in \cite{paper}. Let $\gamma>0$. The evolution of $\Theta_{N,c,K,\gamma}$ is given by
    \begin{align*}
        \frac{\partial}{\partial t} \frac{  \|  \Theta_{N,c,K,\gamma}  \|  ^2_{L^2}}{2}\\
        = -\int_{\mathbb{R}^2} \Theta_{N,c,K,\gamma} 
        \left(
       (v_1^\gamma(\Theta_{N,c,K,\gamma})+
        v_1^\gamma(\overline{\theta}_{N,c,K})) \frac{\partial \Theta_{N,c,K,\gamma}}{\partial x_1} \right.
        \\
        \left.
         +(v_2^\gamma(\Theta_{N,c,K,\gamma})+
        v_2^\gamma(\overline{\theta}_{N,c,K})) \frac{\partial \Theta_{N,c,K,\gamma}}{\partial x_2}
        \right.
        \\
        \left.
        +
        v_1^\gamma(\Theta_{N,c,K,\gamma})\frac{\partial \overline{\theta}_{N,c,K}}{\partial x_1}+
        v_2^\gamma(\Theta_{N,c,K,\gamma})\frac{\partial \overline{\theta}_{N,c,K}}{\partial x_2}-F_{N,c,K,\gamma}
        \right)dx,
    \end{align*}
    but now we use the incompresibility property $\text{div}(v^\gamma)=0$ and obtain
    \begin{align*}
        \int_{\mathbb{R}^2} \Theta_{N,c,K,\gamma} 
        \left(
       (v_1^\gamma(\Theta_{N,c,K,\gamma})+
        v_1^\gamma(\overline{\theta}_{N,c,K})) \frac{\partial \Theta_{N,c,K,\gamma}}{\partial x_1} +
         (v_2^\gamma(\Theta_{N,c,K,\gamma})+
        v_2^\gamma(\overline{\theta}_{N,c,K})) \frac{\partial \Theta_{N,c,K,\gamma}}{\partial x_2}
        \right)dx\\
        = \int_{\mathbb{R}^2} 
        v^\gamma(\Theta_{N,c,K,\gamma}) \cdot \nabla( |\Theta_{N,c,K,\gamma}|^2) 
        dx+
        \int_{\mathbb{R}^2} 
        v^\gamma(\overline{\theta}_{N,c,K}) \cdot \nabla( |\Theta_{N,c,K,\gamma}|^2) 
        dx\\
        =-\int_{\mathbb{R}^2} \text{div}(v^\gamma(\Theta_{N,c,K,\gamma})) |\Theta_{N,c,K,\gamma}|^2 dx
        -\int_{\mathbb{R}^2} \text{div}(v^\gamma(\overline{\theta}_{N,c,K})) |\Theta_{N,c,K,\gamma}|^2 dx=0,
    \end{align*}
    Another important tool to bound $\frac{\partial}{\partial t}\frac{  \|  \Theta_{N,c,K,\gamma}  \|  ^2_{L^2}}{2}$ is that $v^\gamma$ is an odd operator \footnote{By an odd operator we mean that $(v^\gamma(f),g)_{L^2}=-(f,v^\gamma(g))_{L^2}$ where $(\cdot,\cdot)_{L^2}$ denotes the scalar product in $L^2$.}. This implies, for any $\varphi\in W^{1,\infty}$,
   \begin{align}
   \label{imparidad}
        \int_{\mathbb{R}^2} D^s \Theta_{N,c,K,\gamma} \frac{\partial \varphi}{\partial x_i} D^s(v_i^\gamma (\Theta_{N,c,K,\gamma}))dx
        =- \int_{\mathbb{R}^2} v_i^\gamma(D^s \Theta_{N,c,K,\gamma} \frac{\partial \varphi}{\partial x_i}) D^s \Theta_{N,c,K,\gamma}dx.
    \end{align}
    Hence, applying (\ref{imparidad}) with $\varphi=\overline{\theta}_{N,c,K}$, we can continue as follows:
    \begin{align*}
      \frac{\partial}{\partial t}\frac{  \|  \Theta_{N,c,K,\gamma}  \|  ^2_{L^2}}{2} \\
      = \left|\int_{\mathbb{R}^2} 
       \Theta_{N,c,K,\gamma} 
       \left(
          v_1^\gamma(\Theta_{N,c,K,\gamma})\frac{\partial \overline{\theta}_{N,c,K}}{\partial x_1}+
        v_2^\gamma(\Theta_{N,c,K,\gamma})\frac{\partial \overline{\theta}_{N,c,K}}{\partial x_2}-F_{N,c,K,\gamma}
       \right)
      dx \right| \\
      = \frac{1}{2}\left|\int_{\mathbb{R}^2} 
       \Theta_{N,c,K,\gamma} 
       \left(
          v_1^\gamma(\Theta_{N,c,K,\gamma})\frac{\partial \overline{\theta}_{N,c,K}}{\partial x_1}-
           v_1^\gamma(\Theta_{N,c,K,\gamma}\frac{\partial \overline{\theta}_{N,c,K}}{\partial x_1})+ \right. \right.\\
           \left. \left.
        v_2^\gamma(\Theta_{N,c,K,\gamma})\frac{\partial \overline{\theta}_{N,c,K}}{\partial x_2}-
        v_2^\gamma(\Theta_{N,c,K,\gamma}\frac{\partial \overline{\theta}_{N,c,K}}{\partial x_2})-2F_{N,c,K,\gamma}
       \right)
       dx \right|\\
      \leq  \|   \Theta_{N,c,K,\gamma}  \| _{L^2} 
      ( \|  v^\gamma ( \Theta_{N,c,K,\gamma} ) \cdot \nabla \overline{\theta}_{N,c,K}-v^\gamma ( \Theta_{N,c,K,\gamma} \cdot  \nabla \overline{\theta}_{N,c,K}) \| _{L^2}+ \|  F_{N,c,K,\gamma} \| _{L^2})\\
      \leq C  \|   \Theta_{N,c,K,\gamma}  \| _{L^2}( \|   \Theta_{N,c,K,\gamma}  \| _{L^2}  \|  D^\gamma \overline{\theta}_{N,c,K}  \| _{C^1}+ \|  F_{N,c,K,\gamma} \| _{L^2})\\
      \leq C  \|   \Theta_{N,c,K,\gamma}  \| _{L^2}( \|   \Theta_{N,c,K,\gamma}  \| _{L^2} + N^{-(2\beta -1-\gamma)}\mathcal{X}(\gamma)),
    \end{align*}
    where we used also the Lemma \ref{leibnizv}, the bound $ \|  D^\gamma \overline{\theta}_{N,c,K}  \| _{C^{1}} \leq \|  D^\gamma (\overline{\theta}_{N,c,K}-f_{1,c,K})  \| _{C^{1}}+\|  D^\gamma f_{1,c,K}\| _{C^{1}} \leq C$ proved in (\ref{normafraccionario}) and $  \|  F_{N,c,K,\gamma}  \|  _{L^2} \leq C N^{-(2\beta-1-\gamma)}\mathcal{X}(\gamma)$ (this last one proved in Lemma \ref{FnormL2}). 
    We have that $\eta(t)=  \|  \Theta_{N,c,K,\gamma}(\cdot,t)  \|  ^2_{L^2}$ fulfils
    \begin{align*}
        \frac{\partial \eta(t)^2}{\partial t} \leq C \eta(t)^2+ \frac{C\mathcal{X}(\gamma)}{N^{2\beta-1-\gamma}} \eta(t) \Longrightarrow \frac{\partial \eta(t)}{\partial t} \leq C \eta(t)+ \frac{C\mathcal{X}(\gamma)}{N^{2\beta-1-\gamma}}  \Rightarrow\\
          \frac{\partial(e^{-Ct}\eta(t))}{\partial t} \leq \frac{Ce^{-Ct}\mathcal{X}(\gamma)}{N^{2\beta-1-\gamma}} \Longrightarrow e^{-Ct}\eta(t) \leq \eta(0)+ \frac{C(1-e^{-Ct})\mathcal{X}(\gamma)}{N^{2\beta-1-\gamma}} =\frac{C(1-e^{-Ct})\mathcal{X}(\gamma)}{N^{2\beta-1-\gamma}},
    \end{align*}
    so finally for $t\in [0,T]$
    \[
        {  \|  \Theta_{N,c,K,\gamma}  \|  _{L^2}} \leq \frac{C\mathcal{X}(\gamma)(e^{Ct}-1)}{N^{2\beta-1-\gamma}} \leq  \frac{C\mathcal{X}(\gamma)t}{N^{2 \beta-1-\gamma}}.
    \]
    Now, let $\gamma \in (-1,0)$. Then using the Lemma \ref{gammanegativa} and $\beta\geq 1$,
    \begin{align*}
        \frac{\partial}{\partial t}\frac{  \|  \Theta_{N,c,K,\gamma}  \|  ^2_{L^2}}{2} \\
      = \left|\int_{\mathbb{R}^2} 
       \Theta_{N,c,K,\gamma} 
       \left(
          v_1^\gamma(\Theta_{N,c,K,\gamma})\frac{\partial \overline{\theta}_{N,c,K}}{\partial x_1}+
        v_2^\gamma(\Theta_{N,c,K,\gamma})\frac{\partial \overline{\theta}_{N,c,K}}{\partial x_2}-F_{N,c,K,\gamma}
       \right)
      dx \right|\\
      \leq 
      \| \Theta_{N,c,K,\gamma}\|_{L^2} 
      \left(
      \| \Theta_{N,c,K,\gamma}\|_{L^2}\|\frac{\partial \overline{\theta}_{N,c,K}}{\partial x_1}\|_{L^\infty} 
      +   \| \Theta_{N,c,K,\gamma}\|_{L^2}
      \|\frac{\partial \overline{\theta}_{N,c,K}}{\partial x_2}\|_{L^\infty}
      +\|F_{N,c,K,\gamma}\|_{L^2}
      \right)\\
      \leq C
      \| \Theta_{N,c,K,\gamma}\|_{L^2}(N^{-\beta +1} \| \Theta_{N,c,K,\gamma} \|_{L^2}+ N^{-2\beta+1+\gamma}\mathcal{X}(\gamma))
      \\
      \leq C
      \| \Theta_{N,c,K,\gamma}\|_{L^2}( \| \Theta_{N,c,K,\gamma} \|_{L^2}+ N^{-2\beta+1+\gamma}\mathcal{X}(\gamma)),
          \end{align*}
          so the same bound is valid.
\end{proof}
\end{lma}
\vspace{3.5mm}
Now, let us try to find the bounds for the Sobolev norms of $\Theta_{N,c,K,\gamma}$.
\vspace{3mm}
\begin{lma}
\label{cotaTheta}
    Let $T>0$, $\gamma \in (-1,1)$, $\beta \in [1,2+\gamma)\cap (\frac{3}{2}+\gamma,2+\gamma)$ and $\Theta_{N,c,K,\gamma}$ defined as in (\ref{defdiferencia}). Then we have that, for $N$ large, $\theta_{N,c,K,\gamma}$ exist for $t\in [0,T]$ and 
    \begin{equation}
    \label{cotaTheta2}
      \|  \Theta_{N,c,K,\gamma}(\cdot,t)  \|  _{H^{\beta+\frac{1}{2}}} \leq \frac{C\mathcal{X}(\gamma)t}{N^{\beta-\frac{3}{2}-\gamma}}, \hspace{3mm} \forall t \in [0,T]  ,  
    \end{equation}
    with $C$ depending on $c, K, \gamma$ and $T$ and with $\mathcal{X}(\gamma)$ given by (\ref{funcionsalto}).
    \begin{proof}
    We will separate cases depending on if $\gamma$ is positive or negative. The case $\gamma=0$ is already done in \cite{paper}.\\\\
    The proof of this lemma is highly technical and hard to follow. The main strategy is to do an energy estimate. The first problem we find is that when we apply the fractional Laplacian at both sides of the equality we find that, in the right hand side we have to deal with the fractional Laplacian applied to products. Since there is not a Leibniz rule for this operator, we have to use the Lemma \ref{leibniz}. Using the notation of the Lemma \ref{leibniz}, we will find products of the form $D^s(fg)$ so we have to sum and subtract 
    \begin{align*}
        \sum_{|\bold{k}|\leq s_1}\frac{1}{\bold{k}!}\partial^{\bold{k}}fD^{s,\bold{k}}g+
    \sum_{|\bold{j}|\leq s_2}\frac{1}{\bold{j}!}\partial^{\bold{j}}gD^{s,\bold{j}}f,
    \end{align*}
    and then the key point is to bound the quantities 
       \begin{align*}
        \left\|\sum_{|\bold{k}|\leq s_1}\frac{1}{\bold{k}!}\partial^{\bold{k}}fD^{s,\bold{k}}g\right\|_{L^2}, \hspace{3mm}
    \left\|\sum_{|\bold{j}|\leq s_2}\frac{1}{\bold{j}!}\partial^{\bold{j}}gD^{s,\bold{j}}f\right\|_{L^2}
    \hspace{3mm}
    \|D^{s_1}f\|_{L^2}, 
    \hspace{3mm}
    \|D^{s_2}g\|_{BMO},
    \end{align*}
    for the functions $f,g$ considered at each case.
    \\\\
    As $  \|  \Theta_{N,c,K,\gamma}(\cdot,0)  \|  _{H^{\beta+\frac{1}{2}}}=0$, and using the well-posedness of the equation in $H^{s}$, $s> 2+\gamma$, we know that 
    \[
    T^\ast_N:=sup\left\{T^\ast\in [0,T]:  \|  \Theta_{N,c,K,\gamma}(\cdot,t)  \|  _{H^{\beta+\frac{1}{2}}}\leq  \log(N) N^{-(\beta-\frac{3}{2}-\gamma)}\mathcal{X}(\gamma), \forall t<T^\ast \right\},
    \]
    will be greater than zero, this is, $T^\ast_N>0$. But if we prove the bound (\ref{cotaTheta2}) on the interval $[0,T^*_N]$, then for $N$ large enough, more precisely when $CT< \log(N)$, we will have $\|  \Theta_{N,c,K,\gamma}(\cdot,T^*)  \|  _{H^{\beta+\frac{1}{2}}}\leq CT^*N^{-\beta+\frac{3}{2}+\gamma}\mathcal{X}(\gamma)<  \log(N) N^{-(\beta-\frac{3}{2}-\gamma)}\mathcal{X}(\gamma)$ for all $T^\ast<T$ and hence by definition of supremum, $T^*_N=T$ for every $N$ big enough.  \\ \\
    Also, by Picard-Lindelöf's Theorem, as we already have local existence of the problem in the Sobolev spaces in $H^{s}$, $s> 2+\gamma$, this bound assures us existence on $[0,T]$.\\ \\
    To improve the notation, let us call $s=\beta+\frac{1}{2}$. Then, applying $D^s$ at both sides and multiplying at both sides by $D^s\Theta_{N,c,K,\gamma}  $ and integrating,
     \begin{align*}
        \frac{\partial}{\partial t} \frac{  \|  D^s\Theta_{N,c,K,\gamma}  \|  ^2_{L^2}}{2}
        = -\int_{\mathbb{R}^2} D^s\Theta_{N,c,K,\gamma} 
        D^s\left(
       (v_1^\gamma(\Theta_{N,c,K,\gamma})+
        v_1^\gamma(\overline{\theta}_{N,c,K})) \frac{\partial \Theta_{N,c,K,\gamma}}{\partial x_1} +\right.
        \\
        \left.
         (v_2^\gamma(\Theta_{N,c,K,\gamma})+
        v_2^\gamma(\overline{\theta}_{N,c,K})) \frac{\partial \Theta_{N,c,K,\gamma}}{\partial x_2} +
        v_1^\gamma(\Theta_{N,c,K,\gamma})\frac{\partial \overline{\theta}_{N,c,K}}{\partial x_1}+
        v_2^\gamma(\Theta_{N,c,K,\gamma})\frac{\partial \overline{\theta}_{N,c,K}}{\partial x_2}-F_{N,c,K,\gamma}
        \right)dx.
    \end{align*}
    \begin{enumerate}
        \item We start bounding, for $\gamma\in (-1,1)$, the term
    \[
    \int_{\mathbb{R}^2} D^s\Theta_{N,c,K,\gamma} 
        D^s\left(
        v_1^\gamma(\overline{\theta}_{N,c,K}) \frac{\partial \Theta_{N,c,K,\gamma}}{\partial x_1} +
        v_2^\gamma(\overline{\theta}_{N,c,K}) \frac{\partial \Theta_{N,c,K,\gamma}}{\partial x_2} 
        \right)dx,
    \]
    and we will do it using Lemma \ref{leibniz}. 
    \\ \\
    Let $s_1=1$ and $s_2=s-1$. Also, let $f=v^\gamma_i(\overline{\theta}_{N,c,K})$ and $g=\frac{\partial \Theta_{N,c,K,\gamma}}{\partial x_i}$ for $i=1,2$. Then 
    \begin{align*}
    \left(
    D^s \Theta_{N,c,K,\gamma}, D^s(fg)-
    \sum_{|\bold{k}|\leq s_1}\frac{1}{\bold{k}!}\partial^{\bold{k}}fD^{s,\bold{k}}g-
    \sum_{|\bold{j}|\leq s_2}\frac{1}{\bold{j}!}\partial^{\bold{j}}gD^{s,\bold{j}}f
    \right)_{L^2}\\
    \leq 
    C  \|  D^s \Theta_{N,c,K,\gamma}  \|  _{L^2}   \|  D^{s_1} f  \|  _{BMO}   \|  D^{s_2}g  \|  _{L^2} 
    \leq  C  \|  D^s \Theta_{N,c,K,\gamma}  \|  _{L^2}   \|  \Theta_{N,c,K,\gamma}  \|  _{H^s},
    \end{align*}
    where we used $   \|  D^{s_1} f  \|  _{BMO} \leq   \|  D^{s_1} f  \|  _{L^\infty} =  \|  D^1 v^\gamma_i(\overline{\theta}_{N,c,K})  \|  _{L^\infty}\leq  \| D^1 v^\gamma_i(\overline{\theta}_{N,c,K}-f_{1,c,K})  \|  _{L^\infty}+   \| D^1 v^\gamma_i(f_{1,c,K})  \|  _{L^\infty}$ $\leq C\frac{N^{1+\gamma}}{N^{\beta}}\ln(e+N)+C\leq C$ because of the bound \ref{normackvelocidad}.\\ \\
    Now, since we want to use the triangular inequality, we must bound 
    \begin{align*}
        \left(D^s \Theta_{N,c,K,\gamma},  \sum_{|\bold{k}|\leq s_1}\frac{1}{\bold{k}!}\partial^{\bold{k}}fD^{s,\bold{k}}g+
    \sum_{|\bold{j}|\leq s_2}\frac{1}{\bold{j}!}\partial^{\bold{j}}gD^{s,\bold{j}}f\right)_{L^2}.
    \end{align*}
    We first start with the term with $\bold{k}=0$.
    \begin{align*}
        \left(
    D^s \Theta_{N,c,K,\gamma}, D^s(\frac{\partial \Theta_{N,c,K,\gamma}}{\partial x_1})v^\gamma_1(\overline{\theta}_{N,c,K})+ D^s(\frac{\partial \Theta_{N,c,K,\gamma}}{\partial x_2})v^\gamma_2(\overline{\theta}_{N,c,K})
    \right)_{L^2} \\
    =\frac{1}{2} \int_{\mathbb{R}^2} (\frac{\partial}{\partial x_1} (D^s \Theta_{N,c,K,\gamma})^2 v^\gamma_1(\overline{\theta}_{N,c,K})+
    \frac{\partial}{\partial x_2} (D^s \Theta_{N,c,K,\gamma})^2 v^\gamma_2(\overline{\theta}_{N,c,K}))dx\\
    =\frac{1}{2} \int_{\mathbb{R}^2} \nabla (D^s \Theta_{N,c,K,\gamma})^2 \cdot v^\gamma (\overline{\theta}_{N,c,K}) dx =- \frac{1}{2} \int_{\mathbb{R}^2} (D^s \Theta_{N,c,K,\gamma})^2 div(v^\gamma(\overline{\theta}_{N,c,K}))dx=0,
    \end{align*}
    because of the incompresibility condition.\\ \\
    Now, let us bound the part of the sum in $\bold{k}$ with $|\bold{k}|=1$. Using that $D^{s,\bold{k}}$ is continuous from $H^a$ to $H^{a-s+|\bold{k}|}$ and the bound (\ref{normackvelocidad}),
    \begin{align*}
       |( D^s \Theta_{N,c,K,\gamma}, 
        \sum_{|\bold{k}|=1}\frac{1}{\bold{k}!}\partial^{\bold{k}}fD^{s,\bold{k}}g
       )_{L^2}|  \\
       \leq C   \|  D^s \Theta_{N,c,K,\gamma}   \|  _{L^2}   \|  v^\gamma_i(\overline{\theta}_{N,c,K})  \|  _{C^1}   \|  \Theta_{N,c,K,\gamma}  \|  _{H^s} \\
       \leq
       C   \|  D^s \Theta_{N,c,K,\gamma}   \|  _{L^2}   \|  \Theta_{N,c,K,\gamma}  \|  _{H^s}.
    \end{align*}
    Before continuing with the bounds for $j=0$, we need to estimate the $H^1$ norm of $\Theta_{N,c,K,\gamma}$ by interpolation (between the $L^2$ norm obtained in the Lemma \ref{cotadiferencia} and the bootstrap assumption), obtaining
    \begin{align*}
      \|   \Theta_{N,c,K,\gamma}  \|  _{H^1} \leq C   \|  \Theta_{N,c,K,\gamma}  \|  ^{1-\frac{1}{s}}_{L^2}   \|  \Theta_{N,c,K,\gamma}  \|  ^{\frac{1}{s}}_{H^s}\\
    \leq C N^{-(2 \beta -1-\gamma)(1-\frac{1}{s})} (\log(N))^{\frac{1}{s}} N^{-(\beta-\frac{3}{2}-\gamma)\frac{1}{s}}
    \leq C N^{-(2 \beta -1-\gamma)(1-\frac{1}{s})+\epsilon-(\beta-\frac{3}{2}-\gamma)\frac{1}{s}}  \\
    = C N^{-(2 \beta -1-\gamma)+\frac{1}{s}(\beta+\frac{1}{2})+\epsilon} = C N^{-(2 \beta -2-2\gamma)-\gamma+\varepsilon} = CN^{-\beta+\frac{1}{2}},
    \end{align*}
    choosing $\epsilon=\beta-\frac{3}{2}-\gamma>0$.\\ \\
    Let us obtain the bound for the sum in $\bold{j}=0$ for $\gamma\in (-1,1)$. Using the bound (\ref{normackvelocidad}) and proceeding in the same way as before,
    \begin{align*}
        |(
         D^s \Theta_{N,c,K,\gamma}, \frac{\partial \Theta_{N,c,K,\gamma}}{\partial x_i} D^s v^\gamma_i (\overline{\theta}_{N,c,K})
         )| \\
         \leq 
         C   \|   D^s \Theta_{N,c,K,\gamma}  \|  _{L^2}    \|   \Theta_{N,c,K,\gamma}  \|  _{H^1}   \|  D^{s}v^\gamma_i( \overline{\theta}_{N,c,K})  \|  _{L^\infty}\\
         \leq
         C   \|   D^s \Theta_{N,c,K,\gamma}  \|  _{L^2} 
        N^{-\beta+\frac{1}{2}}
         (\frac{N^{s+\gamma}}{N^\beta}\ln(e+N)+C)
         \leq 
          C   \|   D^s \Theta_{N,c,K,\gamma}  \|  _{L^2} N^{-\beta+\frac{3}{2}+\gamma}  .
    \end{align*}
    For the part of the sum in $\bold{j}$ with $|\bold{j}|=1$ and positive $\gamma>0$, we first need a $H^2$-bounding of $ \Theta_{N,c,K,\gamma}$. Taking into account the assumption of the bound of $  \|   \Theta_{N,c,K,\gamma}  \|  _{H^s}$ on $[0,T^\ast]$ and proceeding by interpolation, we have
    \begin{align*}
      \|   \Theta_{N,c,K,\gamma}  \|  _{H^2} \leq C   \|  \Theta_{N,c,K,\gamma}  \|  ^{1-\frac{2}{s}}_{L^2}   \|  \Theta_{N,c,K,\gamma}  \|  ^{\frac{2}{s}}_{H^s}\\
    \leq C N^{-(2 \beta -1-\gamma)(1-\frac{2}{s})} (\log(N))^{\frac{2}{s}} N^{-(\beta-\frac{3}{2}-\gamma)\frac{2}{s}}
    \leq C N^{-(2 \beta -1-\gamma)(1-\frac{2}{s})+\epsilon-(\beta-\frac{3}{2}-\gamma)\frac{2}{s}}  \\
    = C N^{-(2 \beta -1-\gamma)+\frac{2}{s}(\beta+\frac{1}{2})+\epsilon} = C N^{-(2 \beta -3-2\gamma)-\gamma+\varepsilon} = CN^{-\beta+\frac{3}{2}},
    \end{align*}
    choosing $\epsilon=\beta-\frac{3}{2}-\gamma>0$.\\ \\
    Then, if $\gamma>0$, using the bound (\ref{normackvelocidad}) 
    \begin{align*}
          |( 
          D^s \Theta_{N,c,K,\gamma}, \sum_{|\bold{j}|=1}\frac{1}{\bold{j}!}\partial^{\bold{j}}gD^{s,\bold{j}}f
          )_{L^2}| \\
          \leq 
          C   \|   D^s \Theta_{N,c,K,\gamma}  \|  _{L^2} 
            \|   \Theta_{N,c,K,\gamma}  \|  _{H^2}
           \sum_{|\bold{j}|=1}
            \|   D^{s,\bold{j}}  v^\gamma_i (\overline{\theta}_{N,c,K})  \|  _{L^\infty}\\
           \leq 
          C   \|   D^s \Theta_{N,c,K,\gamma}  \|  _{L^2} 
            \|   \Theta_{N,c,K,\gamma}  \|  _{H^2}
           \sum_{|\bold{j}|=1}
            \|   D^{s-2}  \partial^{\bold{j}}  v^\gamma_i (\overline{\theta}_{N,c,K})  \|  _{L^\infty}\\
           \leq C   \|   D^s \Theta_{N,c,K,\gamma}  \|  _{L^2} 
            \|   \Theta_{N,c,K,\gamma}  \|  _{H^2}
           \sum_{|\bold{j}|=1}
          (  \|   D^{s-2}  \partial^{\bold{j}}  v^\gamma_i (\overline{\theta}_{N,c,K}-f_{1,c,K})  \|  _{L^\infty}
          \\+  \|   D^{s-2}  \partial^{\bold{j}}  v^\gamma_i (f_{1,c,K})  \|  _{L^\infty})
          \end{align*}
          \begin{align}
          \label{aux}
          \leq 
          C   \|   D^s \Theta_{N,c,K,\gamma}  \|  _{L^2} 
            \|   \Theta_{N,c,K,\gamma}  \|  _{H^2} 
          (\frac{N^{s-1+\gamma}}{N^\beta}\ln(e+N)+C)             
          \end{align}
          \begin{align*}
          \leq 
           C   \|   D^s \Theta_{N,c,K,\gamma}  \|  _{L^2} N^{-\beta+\frac{3}{2}}(N^{\gamma-\frac{1}{2}}\ln(e+N)+C) \\ \leq 
           C   \|   D^s \Theta_{N,c,K,\gamma}  \|  _{L^2} N^{-\beta+\frac{3}{2}+\gamma}.
    \end{align*}
    If $\gamma <0$, then we must compute the sum in $|\bold{j}|=1$ whenever $s_2=s-1\geq 1$, that implies $s\geq 2$ and hence $\|\partial_{x_i} \partial_{x_j} \Theta_{N,c,K,\gamma}\|_{L^2}\leq \|\Theta_{N,c,K,\gamma}\|_{H^s}$. We repeat the previous computation until the bound (\ref{aux}) and complete it in a slightly different way. Indeed,
    \begin{align}
    \label{gammanegativo1}
    \begin{array}{l}
                 C   \|   D^s \Theta_{N,c,K,\gamma}  \|  _{L^2} 
            \|   \Theta_{N,c,K,\gamma}  \|  _{H^2} 
          (\frac{N^{s-1+\gamma}}{N^\beta}\ln(e+N)+C) \\ \\
          \leq  C   \|   D^s \Theta_{N,c,K,\gamma}  \|  _{L^2} 
            \|   \Theta_{N,c,K,\gamma}  \|  _{H^s} (N^{-\frac{1}{2}+\gamma}\ln(e+N)+C)\\ \\
            \leq
             C   \|   D^s \Theta_{N,c,K,\gamma}  \|  _{L^2} 
            \|   \Theta_{N,c,K,\gamma}  \|  _{H^s}  .
                \end{array}
    \end{align}    
   We consider $|\bold{j}|=2$ when $s_2=s-1\geq 2\Rightarrow s\geq 3$ so $  \|  \Theta_{N,c,K,\gamma}  \|  _{H^3} \leq   \|  \Theta_{N,c,K,\gamma}  \|  _{H^s}$. Using (\ref{normackvelocidad}) to bound the $C^k$ norms of the velocity, 
    \begin{align*}
         |( 
          D^s \Theta_{N,c,K,\gamma}, \sum_{|\bold{j}|=2}\frac{1}{\bold{j}!}\partial^{\bold{j}}gD^{s,\bold{j}}f
          )_{L^2}| \\
           \leq 
           C   \|   D^s \Theta_{N,c,K,\gamma}  \|  _{L^2} 
            \|   \Theta_{N,c,K,\gamma}  \|  _{H^3}
           \sum_{|\bold{j}|=2}
            \|   D^{s,\bold{j}}  v^\gamma_i (\overline{\theta}_{N,c,K})  \|  _{L^\infty}\\
          \leq 
           C   \|   D^s \Theta_{N,c,K,\gamma}  \|  _{L^2} 
            \|   \Theta_{N,c,K,\gamma}  \|  _{H^s}
           \sum_{|\bold{j}|=2}
            \|   D^{s,\bold{j}}  v^\gamma_i (\overline{\theta}_{N,c,K})  \|  _{L^\infty}\\
          \leq
           C   \|   D^s \Theta_{N,c,K,\gamma}  \|  _{L^2} 
            \|   \Theta_{N,c,K,\gamma}  \|  _{H^s} (\frac{N^{s-2+\gamma}}{N^\beta}\ln(e+N)+C)\\
           \leq
           C   \|   D^s \Theta_{N,c,K,\gamma}  \|  _{L^2} 
            \|   \Theta_{N,c,K,\gamma}  \|  _{H^s} (N^{-\frac{3}{2}+\gamma}\ln(e+N)+C)\\
          \leq  C   \|   D^s \Theta_{N,c,K,\gamma}  \|  _{L^2} 
            \|   \Theta_{N,c,K,\gamma}  \|  _{H^s}.
    \end{align*}
    \item \label{prueba.2} Now let us study the bound of the term 
       \[
    \int_{\mathbb{R}^2} D^s\Theta_{N,c,K,\gamma} 
        D^s\left(
        v_1^\gamma(\Theta_{N,c,K,\gamma} ) \frac{\partial \Theta_{N,c,K,\gamma}}{\partial x_1} +
        v_2^\gamma(\Theta_{N,c,K,\gamma} ) \frac{\partial \Theta_{N,c,K,\gamma}}{\partial x_2} 
        \right)dx.
    \]
    \vspace{3mm}
    First, let us study some useful bounds on $\Theta_{N,c,K,\gamma}$. By Sobolev embeddings, for $\gamma \in (0,1)$:
    \begin{itemize}
        \item As $\Theta_{N,c,K,\gamma} \in H^s$, then $D^1 v^\gamma(\Theta_{N,c,K,\gamma})\in  L^\infty$, as well as $\partial^{\bold{k}}v^\gamma(\Theta_{N,c,K,\gamma})\in L^\infty$ for $|\bold{k}|=1$. Indeed,
        \begin{align}
        \label{cota1}
        \begin{array}{l}
        \|D^1 v^\gamma(\Theta_{N,c,K,\gamma} )\|_{L^\infty}\leq \|D^1 v^\gamma(\Theta_{N,c,K,\gamma} )\|_{C^{0,\varepsilon}} 
        \leq C \|D^1 v^\gamma(\Theta_{N,c,K,\gamma} )\|_{W^{1+\varepsilon,2}}\\ \\
        \leq C \|\Theta_{N,c,K,\gamma} \|_{W^{2+\varepsilon+\gamma,2}}=C \|\Theta_{N,c,K,\gamma} \|_{H^s},
        \end{array}
        \end{align}
        \begin{align}
        \label{cota1.1}
        \begin{array}{l}
        \|\partial^{\bold{k}}
        v^\gamma(\Theta_{N,c,K,\gamma} )\|_{L^\infty}\leq \|v^\gamma(\Theta_{N,c,K,\gamma} )\|_{C^{1,\varepsilon}} 
        \leq C \| v^\gamma(\Theta_{N,c,K,\gamma} )\|_{W^{2+\varepsilon,2}}\\ \\
        \leq C \|\Theta_{N,c,K,\gamma} \|_{W^{2+\varepsilon+\gamma,2}}=C \|\Theta_{N,c,K,\gamma} \|_{H^s},
        \end{array}
        \end{align}
        choosing $\varepsilon=\beta-\frac{3}{2}-\gamma>0$.
        \item If $|\bold{j}|=1$, as $\Theta_{N,c,K,\gamma} \in H^s$ then $\partial^{\bold{j}}\Theta_{N,c,K,\gamma} \in L^{\frac{2}{1-\gamma}}$, since
        \begin{align}
        \label{cota2}
         \|  \partial^{\bold{j}}\Theta_{N,c,K,\gamma}  \|  _{L^{\frac{2}{1-\gamma}}} \leq C   \|  \partial^{\bold{j}}\Theta_{N,c,K,\gamma}  \|  _{ W^{\gamma,2}}\leq C   \|  \Theta_{N,c,K,\gamma}  \|  _{H^{\gamma+1}}\leq \|\Theta_{N,c,K,\gamma}\|_{H^{s}},
        \end{align}
        using $s=\beta+1/2>1+\gamma$.
        \item If $|\bold{j}|=1$, as $\Theta_{N,c,K,\gamma} \in H^s$, then $D^{s,\bold{j}}v^{\gamma}(\Theta_{N,c,K,\gamma}) \in L^{\frac{2}{\gamma}}$ and as before 
        \begin{align}
        \label{cota3}
        \|  D^{s,\bold{j}}v^{\gamma}(\Theta_{N,c,K,\gamma})  \|  _{ L^{\frac{2}{\gamma}}}\leq C \|  D^{s,\bold{j}}v^{\gamma}(\Theta_{N,c,K,\gamma})  \|  _{ W^{1-\gamma,2}}\leq C   \|  \Theta_{N,c,K,\gamma}  \|  _{H^s}.
        \end{align}
    \end{itemize}
    As before, $s_1=1$ and $s_2=s-1$. We define now $f$ and $g$ as $f=v^\gamma_i(\Theta_{N,c,K,\gamma})$ and $g=\frac{\partial \Theta_{N,c,K,\gamma}}{\partial x_i}$ for $i=1,2$. Then, applying Lemma \ref{leibniz},
    \begin{align*}
    \left(
    D^s \Theta_{N,c,K,\gamma}, D^s(fg)-
    \sum_{|\bold{k}|\leq s_1}\frac{1}{\bold{k}!}\partial^{\bold{k}}fD^{s,\bold{k}}g-
    \sum_{|\bold{j}|\leq s_2}\frac{1}{\bold{j}!}\partial^{\bold{j}}gD^{s,\bold{j}}f
    \right)_{L^2}\\
    \leq
    C  \|  D^s \Theta_{N,c,K,\gamma}  \|  _{L^2}   \|  D^{s_1} f  \|  _{BMO}   \|  D^{s_2}g  \|  _{L^2}
    \\
     \leq 
     C  \|  D^s \Theta_{N,c,K,\gamma}  \|  _{L^2}   \|D^1   v_i^\gamma(\Theta_{N,c,K,\gamma})  \|  _{L^\infty}   \|  D^{s-1}\partial_{i}\Theta_{N,c,K,\gamma}  \|  _{L^2}
     \\
     \leq
    C  \|  D^s \Theta_{N,c,K,\gamma}  \|  _{L^2}   \|  \Theta_{N,c,K,\gamma}  \|  ^2_{H^s},
    \end{align*}
    using the bound (\ref{cota1}) for the $L^\infty$ norm of the fractional laplacian of the velocity.\\ \\
   We can repeat the same process as before for $\bold{k}=0$.\\ \\
   For $|\bold{k}|=1$, using that $D^{s,\bold{k}}$ is continuous from $H^a$ to $H^{a-s+|\bold{k}|}$,
 \begin{align*}
       |( D^s \Theta_{N,c,K,\gamma}, 
        \sum_{|\bold{k}|=1}\frac{1}{\bold{k}!}\partial^{\bold{k}}fD^{s,\bold{k}}g
       )_{L^2}| \\
       \leq \sum_{|\bold{k}|=1} \|D^s \Theta_{N,c,K,\gamma}\|_{L^2} \|\partial^\bold{k} v_i^\gamma (\Theta_{N,c,K,\gamma})\|_{L^{\infty}} \|D^{s,\bold{k}}\partial_i \Theta_{N,c,K,\gamma}\|_{L^2
       } \\
       \leq C  \|  D^s \Theta_{N,c,K,\gamma}  \|  _{L^2}   \|  \Theta_{N,c,K,\gamma}  \|  ^2_{H^s}.
    \end{align*}
   Nevertheless, we must do some changes for $|\bold{j}|=0,1,2$.\\ \\
   To bound the term associated to $\bold{j}=0$, we need to use that $v^\gamma_i$ is an odd operator (property (\ref{imparidad}) with $\varphi=\Theta_{N,c,K,\gamma}$), so for $\gamma>0$, using the Lemma \ref{leibnizv},
     \begin{align*}
    |( 
          D^s \Theta_{N,c,K,\gamma}, gD^{s}f
          )_{L^2}|\\
          =\left| \int_{\mathbb{R}^2} D^s \Theta_{N,c,K,\gamma} \frac{\partial \Theta_{N,c,K,\gamma}}{\partial x_i} D^s(v_i^\gamma (\Theta_{N,c,K,\gamma}))dx \right| \\
          =\frac{1}{2}\left|\int_{\mathbb{R}^2} D^s \Theta_{N,c,K,\gamma} 
          \left(
          \frac{\partial \Theta_{N,c,K,\gamma}}{\partial x_i} D^s(v_i^\gamma (\Theta_{N,c,K,\gamma}))
          - v_i^\gamma(D^s \Theta_{N,c,K,\gamma} \frac{\partial \Theta_{N,c,K,\gamma}}{\partial x_i})
          \right)
          dx \right| \\
          \leq \frac{1}{2}
            \|  D^s\Theta_{N,c,K,\gamma}  \|  _{L^2} 
            \left \|   
          \frac{\partial \Theta_{N,c,K,\gamma}}{\partial x_i} D^s(v_i^\gamma (\Theta_{N,c,K,\gamma}))
          - v_i^\gamma(D^s \Theta_{N,c,K,\gamma} \frac{\partial \Theta_{N,c,K,\gamma}}{\partial x_i})
           \right\|  _{L^2}\\
                  \leq \frac{1}{2}
            \|  D^s\Theta_{N,c,K,\gamma}  \|  _{L^2} \\
            \left(
            \left \|   
          \frac{\partial \Theta_{N,c,K,\gamma}}{\partial x_i} D^s(v_i^\gamma (\Theta_{N,c,K,\gamma}))
          - v_i^\gamma(D^s \Theta_{N,c,K,\gamma} \frac{\partial \Theta_{N,c,K,\gamma}}{\partial x_i})+
          v_i^\gamma(\frac{\partial \Theta_{N,c,K,\gamma}}{\partial x_i})D^s \Theta_{N,c,K,\gamma}
           \right\|  _{L^2}
           \right.\\
           \left.
            +\|
             v_i^\gamma(\frac{\partial \Theta_{N,c,K,\gamma}}{\partial x_i})D^s \Theta_{N,c,K,\gamma}
            \|_{L^2}
            \right)
        \\
          \leq C
            \|  D^s\Theta_{N,c,K,\gamma}  \|  _{L^2}     \|  D^s\Theta_{N,c,K,\gamma}  \|  _{L^2}   \|  D^\gamma (\frac{\partial \Theta_{N,c,K,\gamma}}{\partial x_i})  \|  _{L^\infty}    \\
            +C
            \|  D^s\Theta_{N,c,K,\gamma}  \|  _{L^2}     \|  D^s\Theta_{N,c,K,\gamma}  \|  _{L^2}   \|  v_i^\gamma (\frac{\partial \Theta_{N,c,K,\gamma}}{\partial x_i})  \|  _{L^\infty}    \\
            \leq C
               \|  D^s\Theta_{N,c,K,\gamma}  \|  _{L^2}^2  \|  D^\gamma (\frac{\partial \Theta_{N,c,K,\gamma}}{\partial x_i})  \|  _{C^\varepsilon}    
            +C\|  D^s\Theta_{N,c,K,\gamma}  \|  _{L^2}^2   \|  v_i^\gamma (\frac{\partial \Theta_{N,c,K,\gamma}}{\partial x_i})  \|  _{C^\varepsilon} \\
            \leq C\|  D^s\Theta_{N,c,K,\gamma}  \|  _{L^2}^2
            \| \Theta_{N,c,K,\gamma}  \|  _{H^{2+\gamma+\varepsilon}}
          \leq 
          C   \|  D^s\Theta_{N,c,K,\gamma}  \|  _{L^2} 
   \|  \Theta_{N,c,K,\gamma}  \|  ^2_{H^s},
    \end{align*}
    choosing $\varepsilon=s-2-\gamma>0$.\\ \\
    For negative gammas we need the Hölder inequallity and the fractional Sobolev embeddings $H^{\varepsilon} \hookrightarrow L^q$. We obtain from the Theorem 6.7. of the article \cite{sobolevfraccional} that, in dimension two, the inequality $\|f\|_{L^q}\leq C\|f\|_{H^{\varepsilon}}$ holds for every $\varepsilon\in (0,1)$ and every $q\in [2,\frac{2}{1-\varepsilon}]$. In our case we apply this result once with $\varepsilon=1+\gamma, q=-2/\gamma$ and a second time with $\varepsilon=-\gamma, q=\frac{2}{1+\gamma}$, obtaining
    \begin{align}
    \label{gammanegativo2}
    \begin{array}{l}
        \left| \int_{\mathbb{R}^2} D^s \Theta_{N,c,K,\gamma} \frac{\partial \Theta_{N,c,K,\gamma}}{\partial x_i} D^s(v_i^\gamma (\Theta_{N,c,K,\gamma}))dx \right| \\ \\
        \leq 
        \| D^s \Theta_{N,c,K,\gamma}\|_{L^2} 
        \|\frac{\partial \Theta_{N,c,K,\gamma}}{\partial x_i} D^s(v_i^\gamma (\Theta_{N,c,K,\gamma}))\|_{L^2}\\ \\
        \leq
        C\| D^s \Theta_{N,c,K,\gamma}\|_{L^2} 
        \left\|\frac{\partial \Theta_{N,c,K,\gamma}}{\partial x_i} 
        \right\|_{L^{-\frac{2}{\gamma}}}
        \|D^s(v_i^\gamma (\Theta_{N,c,K,\gamma}))\|_{L^{\frac{2}{1+\gamma}}}
        \\ \\
        \leq
        C\| D^s \Theta_{N,c,K,\gamma}\|_{L^2} 
        \left\|\frac{\partial \Theta_{N,c,K,\gamma}}{\partial x_i} 
        \right\|_{H^{1+\gamma}}
        \|D^s(v_i^\gamma (\Theta_{N,c,K,\gamma}))\|_{H^{-\gamma}}
        \\ \\\leq
          C   \|  D^s\Theta_{N,c,K,\gamma}  \|  _{L^2} 
   \|  \Theta_{N,c,K,\gamma}  \|  ^2_{H^s}.
       \end{array}
    \end{align}
    In the case $|\bold{j}|=1$ and $\gamma>0$, using the bounds (\ref{cota2}) and (\ref{cota3}),
   \begin{align*}
        |( 
          D^s \Theta_{N,c,K,\gamma}, \sum_{|\bold{j}|=1}\frac{1}{\bold{j}!}\partial^{\bold{j}}gD^{s,\bold{j}}f
          )_{L^2}| \\
          \leq 
            \sum_{|\bold{j}|=1 }\parallel  D^s \Theta_{N,c,K,\gamma} \parallel _{L^2}  \parallel \partial^{\bold{j}} \partial_{x_i}\Theta_{N,c,K,\gamma}  D^{s, \bold{j}} v_i^\gamma(\Theta_{N,c,K,\gamma}) \parallel _{L^2}\\
          \leq  \sum_{|\bold{j}|=1 }
          C  \parallel  D^s \Theta_{N,c,K,\gamma} \parallel _{L^2}  \parallel \partial^{\bold{j}}\partial_{x_i}\Theta_{N,c,K,\gamma} \parallel _{L^\frac{2}{1-\gamma}}  \parallel D^{s, \bold{j}}v_i^\gamma(\Theta_{N,c,K,\gamma}) \parallel _{L^\frac{2}{\gamma}}
          \\
          \leq
          C  \parallel  D^s \Theta_{N,c,K,\gamma} \parallel _{L^2}  \parallel  \Theta_{N,c,K,\gamma} \parallel ^2_{H^s} .
    \end{align*}
    If $\gamma<0$, then the case $|\bold{j}|=1$ must be done just whenever $s_2=s-1\geq 1$, which implies $s\geq 2$ and hence $\|\partial_{x_i}\partial_{x_j} \Theta_{N,c,K,\gamma} \|_{L^2}\leq \|\Theta_{N,c,K,\gamma}\|_{H^s}$, so using that $D^{s,\bold{k}}$ is continuous from $H^a$ to $H^{a-s+|\bold{k}|}$, 
     \begin{align}
     \label{gammanegativo3}
     \begin{array}{l}
         \left|
            \int_{\mathbb{R}^2} D^s\Theta_{N,c,K,\gamma} \partial_{x_i}\partial_{x_j} \Theta_{N,c,K,\gamma} D^{s,\bold{j}} v^{\gamma} (\Theta_{N,c,K,\gamma})dx
         \right| \\\\
         \leq 
      C \| D^s\Theta_{N,c,K,\gamma}\|_{L^2} \|\partial_{x_i}\partial_{x_j} \Theta_{N,c,K,\gamma} D^{s,\bold{j}} v^{\gamma} (\Theta_{N,c,K,\gamma})\|_{L^2}\\ \\
      \leq C\| D^s\Theta_{N,c,K,\gamma}\|_{L^2} \|\partial_{x_i}\partial_{x_j} \Theta_{N,c,K,\gamma} \|_{L^2}
      \|D^{s,\bold{j}} v^{\gamma} (\Theta_{N,c,K,\gamma})\|_{L^\infty}\\ \\
         \leq C\| D^s\Theta_{N,c,K,\gamma}\|_{L^2} \|\partial_{x_i}\partial_{x_j} \Theta_{N,c,K,\gamma} \|_{L^2}
      \|D^{s,\bold{j}} v^{\gamma} (\Theta_{N,c,K,\gamma})\|_{C^\varepsilon}\\ \\
      \leq C\| D^s\Theta_{N,c,K,\gamma}\|_{L^2} \|\partial_{x_i}\partial_{x_j} \Theta_{N,c,K,\gamma} \|_{L^2}
      \|D^{s,\bold{j}} v^{\gamma} (\Theta_{N,c,K,\gamma})\|_{H^{1+\varepsilon}}
      \\ \\
      \leq C\| D^s\Theta_{N,c,K,\gamma}\|_{L^2} \|\partial_{x_i}\partial_{x_j} \Theta_{N,c,K,\gamma} \|_{L^2}
      \| \Theta_{N,c,K,\gamma}\|_{H^{s+\gamma+\varepsilon}} \\ \\\leq 
       C\| D^s\Theta_{N,c,K,\gamma}\|_{L^2}  \| \Theta_{N,c,K,\gamma}\|_{H^{s}}^2,
       \end{array}
     \end{align}   
     choosing $\varepsilon=-\gamma$.\\ \\ 
        We consider $|\bold{j}|=2$ (whenever $s_2=s-1\geq 2\Rightarrow s\geq 3$ so $  \|  \Theta_{N,c,K,\gamma}  \|  _{H^3} \leq   \|  \Theta_{N,c,K,\gamma}  \|  _{H^s}$). Using the embedding of $C^\varepsilon$ in Sobolev spaces and that $D^{s,\bold{k}}$ is continuous from $H^a$ to $H^{a-s+|\bold{k}|}$ we have
    \begin{align*}
         |( 
          D^s \Theta_{N,c,K,\gamma}, \sum_{|\bold{j}|=2}\frac{1}{\bold{j}!}\partial^{\bold{j}}gD^{s,\bold{j}}f
          )_{L^2}| \\
           \leq 
           C   \|   D^s \Theta_{N,c,K,\gamma}  \|  _{L^2} 
            \|   \Theta_{N,c,K,\gamma}  \|  _{H^3}
           \sum_{|\bold{j}|=2}
            \|   D^{s,\bold{j}}  v^\gamma_i (\Theta_{N,c,K,\gamma})  \|  _{L^\infty}\\
          \leq 
           C   \|   D^s \Theta_{N,c,K,\gamma}  \|  _{L^2} 
            \|   \Theta_{N,c,K,\gamma}  \|  _{H^s}
           \sum_{|\bold{j}|=2}
            \|   D^{s,\bold{j}}  v^\gamma_i (\Theta_{N,c,K,\gamma})  \|  _{L^\infty}\\
             \leq 
           C   \|   D^s \Theta_{N,c,K,\gamma}  \|  _{L^2} 
            \|   \Theta_{N,c,K,\gamma}  \|  _{H^s}
           \sum_{|\bold{j}|=2}
            \|   D^{s,\bold{j}}  v^\gamma_i (\Theta_{N,c,K,\gamma})  \|  _{C^\varepsilon}\\
            \leq C   \|   D^s \Theta_{N,c,K,\gamma}  \|  _{L^2} 
            \|   \Theta_{N,c,K,\gamma}  \|  _{H^s} \|   \Theta_{N,c,K,\gamma}  \|  _{H^{s-1+\varepsilon+\gamma}}
            \\
            =C   \|   D^s \Theta_{N,c,K,\gamma}  \|  _{L^2} 
            \|   \Theta_{N,c,K,\gamma}  \|  _{H^s}^2,
    \end{align*}
    choosing $\varepsilon=1-\gamma$.
    \item Let us estimate 
       \[
    \int_{\mathbb{R}^2} D^s\Theta_{N,c,K,\gamma} 
        D^s\left(
        v_1^\gamma(\Theta_{N,c,K,\gamma} ) \frac{\partial \overline{\theta}_{N,c,K}}{\partial x_1} +
        v_2^\gamma(\Theta_{N,c,K,\gamma} ) \frac{\partial \overline{\theta}_{N,c,K}}{\partial x_2} 
        \right)dx.
    \]
   For $\gamma>0$, we use the Lemma \ref{leibniz} with $f= v_i^\gamma (\Theta_{N,c,K,\gamma})$, $g=\partial_{x_i} \overline{\theta}_{N,c,K}$, $s_1=s-\gamma$ and $s_2=\gamma$. Hence,
\begin{align*}
     \left(
    D^s \Theta_{N,c,K,\gamma}, D^s(fg)-
    \sum_{|\bold{k}|\leq s-\gamma}\frac{1}{\bold{k}!}\partial^{\bold{k}}fD^{s,\bold{k}}g-
    \sum_{|\bold{j}|\leq \gamma}\frac{1}{\bold{j}!}\partial^{\bold{j}}gD^{s,\bold{j}}f
    \right)_{L^2}\\
     \leq C
     \| D^s \Theta_{N,c,K,\gamma}\|_{L^2} \| D^{s-\gamma} v_i^\gamma (\Theta_{N,c,K,\gamma})\|_{L^2} \|D^\gamma \partial_i \overline{\theta}_{N,c,K}\|_{BMO}\\
     \leq C \| D^s \Theta_{N,c,K,\gamma}\|_{L^2} \| \Theta_{N,c,K,\gamma}\|_{H^s},
\end{align*}
where we used $   \|  D^{s_2} g \|  _{BMO} \leq   \|  D^\gamma \partial_{x_i} \overline{\theta}_{N,c,K}  \|  _{L^\infty}\leq  \| D^\gamma \partial_{x_i} (\overline{\theta}_{N,c,K}-f_{1,c,K})  \|  _{L^\infty}+   \|D^\gamma \partial_{x_i}  f_{1,c,K} \|  _{L^\infty}\leq C\frac{N^{1+\gamma}}{N^{\beta}}+C\leq C$ because of the bound (\ref{normafraccionario}).\\\\
Now we have to bound all the crossed terms to be able to do triangular inequality.\\\\
The term associated to $\bold{j}=0$ is bounded as 
\begin{align*}
     |(
    D^s \Theta_{N,c,K,\gamma}, \partial_i \overline{\theta}_{N,c,K} D^s (v_i^\gamma (\Theta_{N,c,K,\gamma})) )_{L^2}|\\
    =\left| \int_{\mathbb{R}^2} D^s \Theta_{N,c,K,\gamma} \frac{\partial \overline{\theta}_{N,c,K}}{\partial x_i} D^s(v_i^\gamma (\Theta_{N,c,K,\gamma}))dx \right| \\
          =\frac{1}{2}\left|\int_{\mathbb{R}^2} D^s \Theta_{N,c,K,\gamma} 
          \left(
          \frac{\partial \overline{\theta}_{N,c,K}}{\partial x_i} D^s(v_i^\gamma (\Theta_{N,c,K,\gamma}))
          - v_i^\gamma(D^s \Theta_{N,c,K,\gamma} \frac{\partial\overline{\theta}_{N,c,K}}{\partial x_i})
          \right)
          dx \right| \\
          \leq \frac{1}{2}\left|\int_{\mathbb{R}^2} D^s \Theta_{N,c,K,\gamma} 
          \left(
          v_i^\gamma (\frac{\partial \overline{\theta}_{N,c,K}}{\partial x_i}) D^s(\Theta_{N,c,K,\gamma})
          \right. \right.\\
          +\left. \left.
          \frac{\partial \overline{\theta}_{N,c,K}}{\partial x_i} D^s(v_i^\gamma (\Theta_{N,c,K,\gamma}))
          - v_i^\gamma(D^s \Theta_{N,c,K,\gamma} \frac{\partial\overline{\theta}_{N,c,K}}{\partial x_i})
          \right)
          dx \right| \\
          +\frac{1}{2}\left|\int_{\mathbb{R}^2} D^s \Theta_{N,c,K,\gamma} 
          v_i^\gamma (\frac{\partial \overline{\theta}_{N,c,K}}{\partial x_i}) D^s(\Theta_{N,c,K,\gamma})
          dx \right| \\
          \leq C \| D^s \Theta_{N,c,K,\gamma}\|_{L^2}^2 \|D^\gamma(\partial_{x_i} \overline{\theta}_{N,c,K})\|_{L^\infty}+C
          \| D^s \Theta_{N,c,K,\gamma}\|_{L^2}^2 \|v_i^\gamma (\partial_{x_i} \overline{\theta}_{N,c,K})\|_{L^\infty}       \\
          \leq C \| D^s \Theta_{N,c,K,\gamma}\|_{L^2}^2,
\end{align*}
using the Lemma \ref{leibnizv} and again the bound (\ref{normackvelocidad}).\\\\
The terms associated to $2\leq |\bold{k}|\leq s-\gamma$ are bounded by 
\begin{align*}
    |( D^s \Theta_{N,c,K,\gamma}, 
        \sum_{2\leq |\bold{k}|\leq s-\gamma}\frac{1}{\bold{k}!}\partial^{\bold{k}}fD^{s,\bold{k}}g
       )_{L^2}|\\
       \leq C
       \sum_{2\leq |\bold{k}|\leq s-\gamma}
       \| D^s \Theta_{N,c,K,\gamma}\|_{L^2} 
       \| \partial^{\bold{k}} v_i^\gamma (\Theta_{N,c,K,\gamma})\|_{L^2} N^{s+1-|\bold{k}|-\beta}
       \leq C
        \| D^s \Theta_{N,c,K,\gamma}\|_{L^2}  \| \Theta_{N,c,K,\gamma}\|_{H^s}, 
\end{align*}
using (\ref{normafraccionario}) to bound the $C^1$ norm of $D^{s,\bold{k}}\overline{\theta}_{N,c,K}$.\\\\
Using the bootstrap assumption and the $L^2$ norm of $ \Theta_{N,c,K,\gamma}$ given by Lemma \ref{cotadiferencia} we obtain the interpolation bound $ \|  \Theta_{N,c,K,\gamma}  \|  _{H^{1+\gamma}}\leq C N^{-\beta+\frac{1}{2}+\gamma}$. The term with $|\bold{k}|=1$ is bounded as
\begin{equation}
\label{casok1}
    \begin{aligned}
     |( D^s \Theta_{N,c,K,\gamma}, 
        \sum_{|\bold{k}| =1}\frac{1}{\bold{k}!}\partial^{\bold{k}}fD^{s,\bold{k}}g
       )_{L^2}|\\
       \leq C\| D^s \Theta_{N,c,K,\gamma}\|_{L^2}N^{-\beta+\frac{1}{2}+\gamma+s-\beta}\leq C\| D^s \Theta_{N,c,K,\gamma}\|_{L^2}N^{-\beta+\frac{3}{2}+\gamma} ,       
    \end{aligned}      
\end{equation}
using (\ref{normafraccionario}) to bound the $C^1$ norm of $D^{s,\bold{k}}\overline{\theta}_{N,c,K}$.\\\\
Also by interpolation, $ \|  \Theta_{N,c,K,\gamma}  \|  _{H^{\gamma}}\leq C N^{-\beta-\frac{1}{2}+\gamma}$ and hence the case $\bold{k}=0$ can be done as follows
\begin{align*}
     |( D^s \Theta_{N,c,K,\gamma},  v_i^\gamma (\Theta_{N,c,K,\gamma}) D^s \partial_{x_i} \overline{\theta}_{N,c,K})|\leq C\| D^s \Theta_{N,c,K,\gamma}\|_{L^2}N^{-\beta-\frac{1}{2}+\gamma+s+1-\beta}\\ \leq C\| D^s \Theta_{N,c,K,\gamma}\|_{L^2}N^{-\beta+\frac{3}{2}+\gamma}, 
\end{align*}
using (\ref{normafraccionario}) to bound the $C^1$ norm of $D^{s}\overline{\theta}_{N,c,K}$.\\\\
This last bound finishes the case of singular velocities. For $\gamma\leq 0$ we also use the Lemma \ref{leibniz} with $f= v_i^\gamma (\Theta_{N,c,K,\gamma})$, $g=\partial_{x_i} \overline{\theta}_{N,c,K}$, $s_1=s$ and $s_2=0$
\begin{align*}
     \left(
    D^s \Theta_{N,c,K,\gamma}, D^s(fg)-
    \sum_{|\bold{k}|\leq s}\frac{1}{\bold{k}!}\partial^{\bold{k}}fD^{s,\bold{k}}g-
    \sum_{|\bold{j}|=0}\frac{1}{\bold{j}!}\partial^{\bold{j}}gD^{s,\bold{j}}f
    \right)_{L^2}\\
     \leq C
     \| D^s \Theta_{N,c,K,\gamma}\|_{L^2} \| D^{s} v_i^\gamma (\Theta_{N,c,K,\gamma})\|_{L^2} \| \partial_i \overline{\theta}_{N,c,K}\|_{BMO}\\
     \leq C \| D^s \Theta_{N,c,K,\gamma}\|_{L^2} \| \Theta_{N,c,K,\gamma}\|_{H^s}.
\end{align*}
The term associated to $\bold{j}=0$ is bounded as 
\begin{align*}
     |(
    D^s \Theta_{N,c,K,\gamma}, \partial_i \overline{\theta}_{N,c,K} D^s (v_i^\gamma (\Theta_{N,c,K,\gamma})) )_{L^2}|
    \leq C \| D^s \Theta_{N,c,K,\gamma}\|_{L^2} \| \Theta_{N,c,K,\gamma}\|_{H^s}.
\end{align*}
The cases $2\leq|\bold{k}|\leq s$ are solved like  
\begin{align*}
     |( D^s \Theta_{N,c,K,\gamma}, 
        \sum_{2\leq|\bold{k}|\leq s}\frac{1}{\bold{k}!}\partial^{\bold{k}}fD^{s,\bold{k}}g
       )_{L^2}|\\
       \leq 
        \sum_{2\leq|\bold{k}|\leq s} \|D^s \Theta_{N,c,K,\gamma}\|_{L^2} \| \partial^{\bold{k}} v_i^\gamma (\Theta_{N,c,K,\gamma})\|_{L^2} N^{s-|\bold{k}|+1-\beta}
        \\ \leq
         \sum_{2\leq|\bold{k}|\leq s}
        \| D^s \Theta_{N,c,K,\gamma}\|_{L^2} \| \Theta_{N,c,K,\gamma}\|_{H^s} N^{3/2-|\bold{k}|},
\end{align*}
using (\ref{normafraccionario}) to bound the $C^1$ norm of $D^{s,\bold{k}}\overline{\theta}_{N,c,K}$.\\\\
The case $|\bold{k}|=1$ is done just as in (\ref{casok1}) and the case $|\bold{k}|=0$ is done as 
\begin{align*}
     |( D^s \Theta_{N,c,K,\gamma},  v_i^\gamma (\Theta_{N,c,K,\gamma}) D^s \partial_{x_i} \overline{\theta}_{N,c,K})|\leq\| D^s \Theta_{N,c,K,\gamma}\|_{L^2}\|  \Theta_{N,c,K,\gamma}\|_{L^2} N^{s+1-\beta} \\
     \leq\| D^s \Theta_{N,c,K,\gamma}\|_{L^2} N^{-2\beta +1+\gamma+3/2}\leq \| D^s \Theta_{N,c,K,\gamma}\|_{L^2} N^{-\beta+3/2+\gamma},
\end{align*}
using (\ref{normafraccionario}) to bound the $C^1$ norm of $D^{s}\overline{\theta}_{N,c,K}$ and $\beta\geq 1$.
    \item Finally, for the term involving $F_{N,c,K,\gamma}$we use the Lemma \ref{FnormH} to obtain
    \begin{align*}
    \left| 
    \int_{\mathbb{R}^2} D^s\Theta_{N,c,K,\gamma} D^s F_{N,c,K,\gamma} dx
    \right|
    \leq
      \|  D^s\Theta_{N,c,K,\gamma}  \|  _{L^2}   \|  D^s F_{N,c,K,\gamma}  \|  _{L^2}
    \leq \frac{C\mathcal{X}(\gamma)}{N^{\beta-\frac{3}{2}-\gamma}}
       \|  D^s\Theta_{N,c,K,\gamma}  \|  _{L^2}.
    \end{align*}
    \end{enumerate}
    Putting all this together, we obtain 
    \begin{align*}
        \frac{\partial   \|  D^s\Theta_{N,c,K,\gamma}  \|  _{L^2}^2}{\partial t} \leq C   \|  D^s\Theta_{N,c,K,\gamma}  \|  _{L^2} (N^{-(\beta-\frac{3}{2}-\gamma)}\mathcal{X}(\gamma)+  \|  \Theta_{N,c,K,\gamma}  \|  _{H^s}+  \|  \Theta_{N,c,K,\gamma}  \|  _{H^s}^2),
    \end{align*}
    and using 
    \begin{align*}
       \|  \Theta_{N,c,K,\gamma}  \|  _{H^s}=  \|  \Theta_{N,c,K,\gamma}  \|  _{L^2}+   \|  D^s\Theta_{N,c,K,\gamma}  \|  _{L^2} \leq C(N^{-(2\beta-1-\gamma)}\mathcal{X}(\gamma)+  \|  D^s\Theta_{N,c,K,\gamma}  \|  _{L^2}),
    \end{align*}
    we have that 
    \begin{align*}
        \frac{\partial   \|  D^s\Theta_{N,c,K,\gamma}  \|  _{L^2}}{\partial t} \leq C (N^{-(\beta-\frac{3}{2}-\gamma)}\mathcal{X}(\gamma)+  \|  D^s\Theta_{N,c,K,\gamma}  \|  _{L^2}+  \|  D^s\Theta_{N,c,K,\gamma}  \|  _{L^2}^2),
    \end{align*}
    and recalling $  \|  D^s\Theta_{N,c,K,\gamma}  \|  _{L^2}\leq C \log(N) N^{-(\beta-\frac{3}{2}-\gamma)}\mathcal{X}(\gamma)\leq C$ in $[0,T^\ast]$,
    \begin{align*}
       \frac{\partial   \|  D^s\Theta_{N,c,K,\gamma}  \|  _{L^2}}{\partial t} \leq C (N^{-(\beta-\frac{3}{2}-\gamma)}\mathcal{X}(\gamma)+  \|  D^s\Theta_{N,c,K,\gamma}  \|  _{L^2}),
    \end{align*}
    so 
    \begin{align*}
          \|  D^s\Theta_{N,c,K,\gamma}  \|  _{L^2} \leq \frac{C(e^{Ct}-1)\mathcal{X}(\gamma)}{N^{\beta-\frac{3}{2}-\gamma}},
    \end{align*}
    as we wanted to prove.
    \end{proof}
\end{lma}
\begin{rmk}
    We may see in this last part of the proof that, to close the estimates, we need to make them in a space $H^s$ in which the force term vanishes when $N$ goes to infinity, this is, 
    \begin{align*}
        \|F_{N,c,K,\gamma}\|_{H^s}\rightarrow 0, \text{ when } N\rightarrow\infty.
    \end{align*}    
    Moreover, in the bound (\ref{normafuerzahk}) one can see (via interpolation) that the $H^s$ norm of $F_{N,c,K,\gamma}$ goes with a negative power of $N$ if and only if $2\beta-1-s-\gamma<0$. Combining this with the fact that we must be in the well posedness regime ($s>2+\gamma$) to apply the bootstrap argument and close the estimates, this lead us to $\beta>\frac{3}{2}+\gamma$. 
\end{rmk}
\subsection{Proof of the strong ill-posedness}
\label{prueba}
Finally we are ready to prove one of the main theorems of the essay (Theorem \ref{teorema}). We restate it here before the proof.
\\\\
\textbf{Theorem} (Strong ill-posedness in $H^\beta$):
For any $T>0$, $0<c_0<1$, $M>1$, $\gamma \in (-1,1)$, $\beta \in [1,2+\gamma)\cap (\frac{3}{2}+\gamma,2+\gamma)$ and $t_\ast>0$ as small as we want, we can find a $H^{\beta+\frac{1}{2}}$ function $\theta_0$ with $  \|  \theta_0  \|  _{H^{\beta}}\leq c_0$ such that the unique solution $\theta (x,t)$ in $H^{\beta+\frac{1}{2}}$ for $t\in[0,T]$ to the $\gamma$-gSQG equation (\ref{gSQG}) with initial conditions $\theta_0$ is such that $  \|  \theta(x,t_\ast)  \|  _{H^\beta}\geq Mc_0$. Moreover, along its time of existence the solution is $C^\infty_c$ and $supp(\theta)\subset supp(\theta_0)+B_{2c_0T}(0)$.
\begin{proof} 
Throughout this proof we will specify the bounds that are not uniform on $0<c<1$, $K>1$ with subscripts in the constants $C_{c,K}>0$ of the bounds. Hence, whenever we write a constant without subscripts $C>0$ in a bound, we are saying that that bound is uniform on $0<c<1$, $K>1$. As a last remark notice that all constants, with or without subscripts are uniform on $N>1$ big enough.\\\\
As the reader may guess, we choose $\theta_0=\overline{\theta}_{N,c,K}(\cdot,0)$ as initial condition. Recall that $  \|  f_{2,c,K}  \|  _{L^2}=c$, so
\begin{align*}
      \|  \overline{\theta}_{N,c,K}(\cdot,0)-f_{1,c,K}  \|  _{L^2}=  \|  f_{2,c,K}\frac{r^\beta_{c,K} \sin(N\alpha)}{N^\beta}  \|  _{L^2}  \\
    \leq   \|  f_{2,c,K}  \|  _{L^2}  \|  \frac{r^\beta_{c,K} \sin(N\alpha)}{N^\beta}  \|  _{L^\infty} \leq \frac{c r^\beta_{c,K} }{N^\beta}.
\end{align*}
Now let us bound the homogeneous norm $\Dot{H}^3$ as follows
\begin{align*}
    \|  \overline{\theta}_{N,c,K}(\cdot,0)-f_{1,c,K}  \|  _{\Dot{H}^{3}}=
 \|  f_{2,c,K}\frac{r^\beta_{c,K} \sin(N\alpha)}{N^\beta}  \|  _{\Dot{H}^{3}}
 \\ \leq 
\sum_{|\bold{j}+\bold{k}|=3} \|\partial^{\bold{k}} f_{2,c,K} \|_{L^2}
\|
\partial^{\bold{j}}\frac{r^{\beta}_{c,K} \sin(N\alpha)}{N^\beta} 
\|_{L^\infty}\\
\leq 
\sum_{|\bold{j}+\bold{k}|=3} \|\partial^{\bold{k}} f_{2,c,K} \|_{L^2} r_{c,K}^{\beta-|\bold{j}|} N^{|\bold{j}|-\beta}\\
\leq C c r_{c,K}^{\beta-3} N^{3-\beta}+ C\|f_{2,c,K} \|_{C^3}r_{c,K}^{\beta-2} N^{2-\beta}.
\end{align*}
Summing it to the $L^2$-norm bound, using interpolation and that $(x+y)^s\leq x^s+y^s$ for $x,y\geq 0$ and $s\in (0,1)$, we obtain 
\begin{align*}
      \|  \overline{\theta}_{N,c,K}(\cdot,0)-f_{1,c,K}  \|  _{H^\beta}\\
    \leq 
       \|  \overline{\theta}_{N,c,K}(\cdot,0)-f_{1,c,K}  \|  _{L^2}^{1-\frac{\beta}{3}}
        \|  \overline{\theta}_{N,c,K}(\cdot,0)-f_{1,c,K}  \|  _{H^{3}}^{\frac{\beta}{3}}\\
      \leq \frac{Cc}{N^{\beta(1-\frac{\beta}{3})+(\beta-3)(\frac{\beta}{3})}}+ \frac{C c^{1-\frac{\beta}{3}} \|f_{2,c,K}\|_{C^3}^{\frac{\beta}{3}}}{(N/r_{c,K})^{\beta(1-\frac{\beta}{3})+(\beta-2)(\frac{\beta}{3})}}
       \leq Cc +C_{c,K}r_{c,K}^{\frac{\beta}{3}}N^{-\frac{\beta}{3}}.
\end{align*}
So, as $  \|  f_{1,c,K}  \|  _{H^\beta}\leq c$, we obtain
\begin{align}
\label{cismall}
      \|  \theta_0  \|  _{H^\beta}\leq Cc +C_{c,K}N^{-\frac{\beta}{3}}\leq \frac{c_0}{2}+C_{c,K}N^{-\frac{\beta}{3}},
\end{align}
choosing $c>0$ small enough.\\ \\
Now, let us make the $H^\beta$ norm of $\overline{\theta}_{N,c,K}(\cdot,t_\ast)$ as big as we want. 
\begin{enumerate}
    \item On the one hand,
    \begin{align*}
           \|  \overline{\theta}_{N,c,K}(\cdot,t)-f_{1,c,K}  \|  _{L^2}=  \|  f_{2,c,K}\frac{r^\beta_{c,K} \sin(N\alpha-Nt\frac{v^\gamma_\alpha(f_{1,c,K}(r))}{r})}{N^\beta}  \|  _{L^2} \leq \frac{cr^\beta_{c,K}}{N^\beta}.
    \end{align*}
     \item On the other hand, we can obtain a lower bound of the $H^1$ norm. Let $t_\ast>0$, then
    \begin{align*}
        \frac{\partial (\overline{\theta}_{N,c,K}-f_{1,c,K})}{\partial x_1}\\
        =\cos(\alpha) \frac{\partial (\overline{\theta}_{N,c,K}-f_{1,c,K})}{\partial r}-\frac{\sin(\alpha)}{ r}\frac{\partial (\overline{\theta}_{N,c,K}-f_{1,c,K})}{\partial \alpha}\\
        = \cos(\alpha)\frac{\partial f_{2,c,K}}{\partial r}\frac{r^\beta_{c,K} \sin(N\alpha-Nt_\ast\frac{v^\gamma_\alpha(f_{1,c,K}(r))}{r})}{N^\beta}\\
        -f_{2,c,K}\frac{r^{\beta}_{c,K}}{N^{\beta-1}}\cos(N\alpha-Nt_\ast\frac{v^\gamma_\alpha(f_{1,c,K}(r))}{r})
        \left(
        t_\ast \cos(\alpha) \frac{\partial \frac{v_\alpha^\gamma(f_1)}{ r}}{\partial r}+ \frac{\sin(\alpha)}{r}
        \right).
    \end{align*}
    Then, using that (\ref{crecimiento}) holds on the support of $f_{2,c,K}$ and that integrating in polar coordinates one can see that the $L^2$ norm of $f_{2,c,K}\cos(N\alpha-Nt_\ast\frac{v^\gamma_\alpha(f_{1,c,K}(r))}{r})  \cos(\alpha)$ is of the form $Cc$ (with $C>0$ not depending  on $N$),
    \begin{align*}
         \| 
        \frac{\partial (\overline{\theta}_{N,c,K}-f_{1,c,K})}{\partial x_1}
         \| _{L^2} \\
        \geq
       \|f_{2,c,K}\frac{r^{\beta}_{c,K}}{N^{\beta-1}}\cos(N\alpha-Nt_\ast\frac{v^\gamma_\alpha(f_{1,c,K}(r))}{r})        t_\ast \cos(\alpha) \frac{\partial \frac{v_\alpha^\gamma(f_1)}{ r}}{\partial r}
       \|_{L^2}
        \\
        -\|f_{2,c,K}\frac{r^{\beta}_{c,K}}{N^{\beta-1}}\cos(N\alpha-Nt_\ast\frac{v^\gamma_\alpha(f_{1,c,K}(r))}{r})\frac{\sin(\alpha)}{r} 
        \|_{L^2}\\
        -\|\cos(\alpha)\frac{\partial f_{2,c,K}}{\partial r}\frac{r^\beta_{c,K} \sin(N\alpha-Nt_\ast\frac{v^\gamma_\alpha(f_{1,c,K}(r))}{r})}{N^\beta}\|_{L^2}
        \\
        \geq
        Cct_{\ast} r^{\beta-1}_{c,K} N^{-\beta+1} K-Cc r^{\beta-1}_{c,K} N^{-\beta+1}-C\|f_{2,c,K}\|_{C^1} r^{\beta}_{c,K} N^{-\beta}
        \\ \geq Ccr^{\beta-1}_{c,K} N^{-\beta+1}(Kt_\ast-1-\|f_{2,c,K}\|_{C^1}N^{-1}),
    \end{align*}
 and since $supp((\overline{\theta}_{N,c,K}-f_{1,c,K}))\cap supp(f_{1,c,K})=\emptyset$, then    
 \[
       \|  
          \overline{\theta}_{N,c,K}
           \|  _{H^1}\geq
     \|  
          (\overline{\theta}_{N,c,K}-f_{1,c,K})
          \|  _{H^1}  \geq
         \frac{Ccr_{c,K}^{\beta-1}(Kt_\ast-1-r_{c,K}\|f_{2,c,K}\|_{C^1}c^{-1}N^{-1})}{N^{\beta-1}}.
    \] 
    \item Now we are ready to use interpolation inequality, which gives us
    \[
      \|  \overline{\theta}_{N,c,K}  \|  _{H^1}\leq 
       \|  \overline{\theta}_{N,c,K}  \|  ^{\frac{\beta-1}{\beta}}_{L^2}
        \|  \overline{\theta}_{N,c,K}  \|  ^{\frac{1}{\beta}}_{H^\beta},
    \]
    and using the bounds for $  \|  \overline{\theta}_{N,c,K}  \|  _{L^2}$ and $  \|  \overline{\theta}_{N,c,K}  \|  _{H^1}$ we obtain
    \begin{align}
    \label{pseudosolbig}
      \|  \overline{\theta}_{N,c,K}  \|  _{H^\beta} \geq Cc (Kt_\ast-1-r_{c,K}\|f_{2,c,K}\|_{C^{1}}c^{-1}N^{-1})^\beta,
    \end{align}
    so as $c>0$ and $t_\ast$ are already chosen, we just have to first fix a $K>t_*^{-1}\left(2+\left(\frac{2Mc_0}{Cc}\right)^\frac{1}{\beta}\right)$ with the $C>0$ in (\ref{pseudosolbig}) and then take $N>0$ big enough such that, for the fixed $K>0$ we obtain $r_{c,K}\|f_{2,c,K}\|_{C^{1}}c^{-1}N^{-1}<1$. Hence,
     \begin{align}
    \label{pseudosolbig2}
       \|  \overline{\theta}_{N,c,K}(\cdot,t_\ast)  \|  _{H^\beta} \geq 2Mc_0.
    \end{align}
    Finally, now that $c,K>0$ are fixed, we again make $N>0$ big to turn (\ref{cismall}) into 
    \[
    \|\theta_0\|_{H^\beta}\leq c_0,
    \]
    just as we wanted.
\end{enumerate} 
Let $\theta_{N,c,K,\gamma}$ be the local-in-time solution  of (\ref{gSQG}) with initial conditions $\theta_0=\overline{\theta}_{N,c,K}(\cdot,0)$. By Lemma \ref{cotaTheta}, we have that, choosing $N>0$ big enough, the corresponding solution will exist until the desired fixed time $T>0$. Furthermore,  
\begin{align}
\label{difsmall}   
   \|  \theta_{N,c,K,\gamma}(\cdot,t_\ast)-\overline{\theta}_{N,c,K}(\cdot,t_\ast)  \|  _{H^\beta} \leq
   \|  \theta_{N,c,K,\gamma}(\cdot,t_\ast)-\overline{\theta}_{N,c,K}(\cdot,t_\ast)  \|  _{H^{\beta+\frac{1}{2}}}\leq \frac{C_{c,K}t_\ast}{N^{\beta-\frac{3}{2}-\gamma}} \mathcal{X}(\gamma),
\end{align}
and since $c,K>0$ are fixed, taking again $N$ large enough
\begin{align}
\label{solbig}
   \|  \theta_{N,c,K,\gamma}(\cdot,t_\ast)  \|  _{H^\beta} \geq
  \|   \overline{\theta}_{N,c,K}(\cdot,t_\ast)  \|  _{H^\beta}-
   \|  \theta_{N,c,K,\gamma}(\cdot,t_\ast)-\overline{\theta}_{N,c,K}(\cdot,t_\ast)  \|  _{H^\beta}
 \geq Mc_0.
\end{align}
Also by construction, since the initial conditions are $C^\infty_c$ and we fixed a time interval $[0,T]$ where the $H^{\beta+\frac{1}{2}}$-norm does not explode, then the solution will be $C^\infty_c$ in the hole interval of time.\\ \\
To prove the confinement of the support, notice that  $\theta_{N,c,K,\gamma}$ are solutions of a transport equation with uniformly bounded velocity on $0<c<1<N,K$. Indeed, 
     \begin{align*}
         \|v^\gamma(\theta_{N,c,K,\gamma})(\cdot,t)\|_{L^\infty}
         \leq \|v^\gamma(\theta_{N,c,K,\gamma}-\overline{\theta}_{N,c,K})(\cdot,t)\|_{L^\infty}+\|v^\gamma(\overline{\theta}_{N,c,K})(\cdot,t)\|_{L^\infty}\\
         \leq \|v^\gamma(\theta_{N,c,K,\gamma}-\overline{\theta}_{N,c,K})(\cdot,t)\|_{C^{0,\varepsilon}}+\|v^\gamma(\overline{\theta}_{N,c,K})(\cdot,t)\|_{C^{0,\varepsilon}}\\
         \leq \|v^\gamma(\theta_{N,c,K,\gamma}-\overline{\theta}_{N,c,K})(\cdot,t)\|_{H^{1+\varepsilon}}+\|v^\gamma(\overline{\theta}_{N,c,K})(\cdot,t)\|_{H^{1+\varepsilon}}\\
         \leq \|(\theta_{N,c,K,\gamma}-\overline{\theta}_{N,c,K})(\cdot,t)\|_{H^{1+\varepsilon+\gamma}}+\|\overline{\theta}_{N,c,K}(\cdot,t)\|_{H^{1+\varepsilon+\gamma}}\\
          \leq \|(\theta_{N,c,K,\gamma}-\overline{\theta}_{N,c,K})(\cdot,t)\|_{H^{\beta}}+\|\overline{\theta}_{N,c,K}(\cdot,t)\|_{H^{1+\varepsilon+\gamma}}\\
         \leq Ct\mathcal{X}(\gamma)N_j^{-\beta-\frac{3}{2}-\gamma}+ \|f_{1,c,K}\|_{H^{1+\varepsilon+\gamma}}+\|\overline{\theta}_{N,c,K}-f_{1,c,K}\|_{H^{1+\varepsilon+\gamma}}\\
         \leq Ct\mathcal{X}(\gamma)N^{-\beta-\frac{3}{2}-\gamma}+ c+Ct\mathcal{X}(\gamma)N^{1+\varepsilon+\gamma-\beta}
          \leq  2c,
     \end{align*} 
     for $N$ big enough and $t\in[0,T]$ and hence the confinement of the solution follows from the bound of the velocity and the regularity of the solutions.\\\\
\end{proof}
\begin{rmk}
    Notice how important it has been, in the previous proof, the order in which we have fixed the parameters $c,K,N>0$. 
    \begin{enumerate}
        \item First of all, we obtain the  condition (\ref{cismall}) and we fix just $0<c<1$ small enough to get $Cc \leq \frac{c_0}{2}$ uniformly on $K,N>1$.
        \item Then we obtain (\ref{pseudosolbig}). With this, we can fix $K>1$ big enough and then make $N>1$ also big (bigger than a certain $N_0>1$ but not fixed) such that (\ref{pseudosolbig2}) holds.
        \item Now that $c,K>0$ are fixed, then we can make $N>1$ even bigger (this is, we can establish a bigger lower bound $N_0>1$) such that, (\ref{cismall}) implies $\|\theta_0\|_{H^\beta}\leq c_0$.
        \item Finally, using again that $c,K>0$ are fixed, we make again $N>1$ big enough such that (\ref{difsmall}) and (\ref{pseudosolbig2}) imply (\ref{solbig}).
    \end{enumerate}
\end{rmk}
\section{Non-existence of solutions for gSQG}
\label{seccion3}
In this section we prove the non-existence result for gSQG in Sobolev spaces announced in Theorem \ref{nonexistence}. Let us remember the precise result.
\begin{thm*}[Non-existence in $H^\beta$ in the supercritical case]
 For any $t_0>0$, $ 0<c_0<1$, $\gamma \in (-1,1)$ and $\beta \in [1,2+\gamma)\cap (\frac{3}{2}+\gamma,2+\gamma)$, we can find initial conditions $\theta_0$ with $\|\theta_{0}\|_{H^\beta}\leq c_0$ such that there exist a solution $\theta$ to $\gamma$-gSQG (\ref{gSQG}) with $\theta(x,0)=\theta_0(x)$ satisfying $\|\theta(x,t)\|_{H^\beta}=\infty$ for all $t\in (0,t_0]$.  Furthermore, it is the only solution with initial conditions $\theta_0$ that satisfies $\theta \in L^\infty([0,t_0],H^{\frac{3}{2}+\gamma})\cap C([0,t_0],C^2_x(K))$ for any compact set $K\subset\mathbb{R}^2$.
 \end{thm*}
 First we prove some lemmas. Notice that the confinement of the support in the previous strong ill-posedness result is crucial in the following proofs. The key point here is that for making the gluing argument we use that the interaction between separated patches of compact support will be small in a finite interval of time, and hence we can consider that the behavior of the different patches will be almost independent. 
 \begin{lma}
 \label{lemanonexistence}
     Let $\theta_1^\gamma$, $\theta_2^\gamma\in C_tC_{x}^\infty$ compactly supported solutions to $\gamma$-gSQG (\ref{gSQG}) with $supp(\theta_i^\gamma)\subset B_D(0)$ for $t\in[0,T]$ and $i\in\{1,2\}$. Then, for any $\varepsilon>0$, there exist $R_D>0$ such that for any $R>R_D$, 
     \begin{align}
         \|\theta^\gamma(\cdot,t)-\theta_1^\gamma(\cdot,t)-T_R\theta_2^\gamma(\cdot,t)\|_{H^4}<\varepsilon, \hspace{3mm} \forall t\in [0,T],
     \end{align}
     where $T_Rf(x_1,x_2)=f(x_1+R,x_2)$ and $\theta^\gamma$ is the solution to $\gamma$-gSQG (\ref{gSQG}) with initial conditions $\theta_1^\gamma(x,0)+T_R\theta_2^\gamma(x,0)$.
     \begin{proof}
        Let us define the function $\overline{\theta}^\gamma_R:=\theta_1^\gamma+T_R \theta_2^\gamma$. It solves the PDE
        \begin{align}
            \partial_t \overline{\theta}^\gamma_R +v^\gamma(\overline{\theta}^\gamma_R)\cdot \nabla \overline{\theta}^\gamma_R=v^\gamma(\theta_1^\gamma)\cdot \nabla T_R\theta_2^\gamma+v^\gamma(T_R\theta_2^\gamma)\cdot \nabla \theta_1^\gamma=:F_R.
        \end{align}
        The source term $F_R$ can be bounded in $H^4$ in terms of $R$ in the following way: let us consider $\sigma=(\sigma_1,\sigma_2)\in \mathbb{N}^2$. Then, taking $R>4D$ we obtain $dist(supp(\theta_1^\gamma), supp(T_R \theta_2^\gamma))\geq R-2D\geq R/2$ for all time, so
        \begin{align*}
             \sum_{|\sigma|=0}^4\int_{\mathbb{R}^2}(\partial^\sigma F_R)^2 dx=
            \sum_{|\sigma|=0}^4\int_{\mathbb{R}^2} \left[ \partial^\sigma (v^\gamma(\theta_1^\gamma)\cdot \nabla T_R\theta_2^\gamma+v^\gamma(T_R\theta_2^\gamma)\cdot \nabla \theta_1^\gamma)\right]^2dx\\
            \leq \|\theta_2^\gamma\|_{C^5}^2 \sum_{|\sigma|=0}^4 \int_{supp(T_R\theta_2^\gamma)} |\partial^\sigma v^\gamma(\theta_1^\gamma)|^2 dx+\|\theta_1^\gamma\|_{C^5}^2 \sum_{|\sigma|=0}^4 \int_{supp(\theta_1^\gamma)} |\partial^\sigma v^\gamma(T_R \theta_2^\gamma)|^2 dx\\
            \leq C \sum_{|\sigma|=0}^4
            \left(R^{-2-\gamma-|\sigma|}\right)^2
            \left(|supp(T_R\theta_2^\gamma)| \|\theta_2^\gamma\|_{C^5}^2 \|\theta_1^\gamma\|_{L^1}^2+|supp(\theta_1^\gamma)| \|\theta_1^\gamma\|_{C^5}^2 \|\theta_2^\gamma\|_{L^1}^2
            \right)
            \\ \leq C 
            \left(R^{-2-\gamma} D \|\theta_1^\gamma\|_{C^5} \|\theta_2^\gamma\|_{C^5}\right)^2,
        \end{align*}
        so 
        \begin{align}
        \label{cotaFH4}
            \|F_R\|_{H^4}\leq C D \|\theta_1^\gamma\|_{C^5} \|\theta_2^\gamma\|_{C^5} R^{-2-\gamma}.
        \end{align}
        Defining $\Theta^\gamma_R:=\theta^\gamma-\theta_1^\gamma-T_R\theta_2^\gamma=\theta^\gamma-\overline{\theta}_R^\gamma$, it solves the PDE 
        \begin{align}
        \label{diferencia}
            \partial_t\Theta^\gamma_R+ v^\gamma(\Theta^\gamma_R) \cdot \nabla \Theta^\gamma_R+v^\gamma(\Theta^\gamma_R)\cdot \nabla \overline{\theta}^\gamma_R +v^\gamma(\overline{\theta}^\gamma_R)\cdot \nabla \Theta^\gamma_R-F_R=0.
        \end{align}
        By (\ref{diferencia}), using the incompresibility condition and the Lemma \ref{leibniz}, following the strategy of the proof of the Lemma \ref{cotadiferencia},
        \begin{align*}
            \partial_t \|\Theta^\gamma_R\|_{L^2}\leq C\|\Theta^\gamma_R\|_{L^2}\|\nabla\overline{\theta}^\gamma_R\|_{L^\infty}+\|F_R\|_{L^2},  \text{ if } \gamma\leq 0,\\
            \partial_t \|\Theta^\gamma_R\|_{L^2}\leq C\|\Theta^\gamma_R\|_{L^2}\|v^\gamma(\nabla\overline{\theta}^\gamma_R)\|_{L^\infty}+\|F_R\|_{L^2},  \text{ if } \gamma>0,
        \end{align*}
        and hence, for $\gamma \in (-1,1)$,
        \begin{align}
        \label{edoL2}
           \partial_t \|\Theta^\gamma_R\|_{L^2}\leq C\|\Theta^\gamma_R\|_{L^2}\|\overline{\theta}^\gamma_R\|_{C^3}+\|F_R\|_{L^2} .
        \end{align}
        Now we look for a bound of the $\Dot{H}^4$ norm. To do this we follow the same energy-estimate technique of considering a generic 4-th order derivative $\partial^\sigma$ ($\sigma=(\sigma_1,\sigma_2)$, $\sigma_1+\sigma_2=4$) we apply it at both sides of the equation (\ref{diferencia}) and we multiply at both sides times $\partial^\sigma \Theta^\gamma_R$ and integrate. We bound every term apart from the time derivative: for the first part of the sum we obtain
        \begin{align*}
            \left|-\int_{\mathbb{R}^2}\partial^\sigma \Theta^\gamma_R \partial^\sigma (v^\gamma(\Theta^\gamma_R) \cdot \nabla \Theta^\gamma_R)dx \right| \leq C\|\Theta^\gamma_R\|_{H^4}^3,
        \end{align*}
        and it can be seen repeating the proof of the well-posedness of $\gamma$-gSQG done in \cite{existence} (proof of the Theorem 1.1).\\\\
        For the third term we obtain the bound 
        \begin{align*}
            \left|-\int_{\mathbb{R}^2}\partial^\sigma \Theta^\gamma_R \partial^\sigma (v^\gamma(\overline{\theta}^\gamma_R) \cdot \nabla \Theta^\gamma_R)dx \right|\leq C \|\overline{\theta}^\gamma_R\|_{H^{6}}\|\Theta^\gamma_R\|_{H^4}^2,
        \end{align*}
        just applying the Leibniz rule and the Hölder's inequality and noticing that, when the four derivatives go to $\nabla \Theta^\gamma_R$, the integral vanishes due to the incompresibility condition.\\ \\
        Finally we bound the second term, this is, 
        \begin{align*}
           \left| -\int_{\mathbb{R}^2} \partial^\sigma \Theta^\gamma_R \partial^\sigma (v^\gamma(\Theta^\gamma_R)\cdot \nabla \overline{\theta}^\gamma_R)dx\right|\leq C \|\overline{\theta}^\gamma_R\|_{C^5}\|\Theta^\gamma_R\|_{H^4}^2.
        \end{align*}
        The idea is the same as before, this is, applying Leibniz rule and the Hölder's inequality with the exception of the high order terms in $\Theta_R$ when $\gamma \in (0,1)$, in which we must apply the fact that $v_i^\gamma$ are odd operators and the Lemma \ref{leibniz}, obtaining
        \begin{align*}
           \left| -\int_{\mathbb{R}^2} \partial^\sigma \Theta^\gamma_R  v^\gamma_i(\partial^\sigma \Theta^\gamma_R) \partial_{x_i} \overline{\theta}^\gamma_R dx\right| =
            \left| \frac{1}{2}\int_{\mathbb{R}^2} \partial^\sigma \Theta^\gamma_R \left(v^\gamma_i(\partial^\sigma \Theta^\gamma_R \partial_{x_i} \overline{\theta}^\gamma_R)  -v^\gamma_i(\partial^\sigma \Theta^\gamma_R) \partial_{x_i} \overline{\theta}^\gamma_R\right) dx\right|\\
            \leq 
            \frac{1}{2}\|\partial^\sigma \Theta^\gamma_R\|_{L^2} \|v^\gamma_i(\partial^\sigma \Theta^\gamma_R \partial_{x_i} \overline{\theta}^\gamma_R)  -v^\gamma_i(\partial^\sigma \Theta^\gamma_R) \partial_{x_i} \overline{\theta}^\gamma_R\|_{L^2}
            \leq C \|D^\gamma (\partial_{x_i}\overline{\theta}^\gamma_R)\|_{L^\infty}\|\Theta^\gamma_R\|^2_{H^4}\leq C \|\overline{\theta}^\gamma_R\|_{C^3}\|\Theta^\gamma_R\|_{H^4}^2.
        \end{align*}
          For lower order terms, this is, when three or less derivatives are affecting $\Theta_R$, we just use $\|\partial^\tau v^\gamma(\Theta_R)\|_{L^2}\leq C\|\Theta_R\|_{H^{|\tau|+\gamma}}$ for $\tau \in \mathbb{N}^2$.\\\\
        Putting these three bounds together we obtain 
        \begin{align*}
            \partial_t \|\Theta^\gamma_R\|_{\Dot{H}^4} \leq C( \|\Theta^\gamma_R\|_{H^4}^2+ \|\overline{\theta}^\gamma_R\|_{C^{6}}\|\Theta^\gamma_R\|_{H^4})+\|F_R\|_{\Dot{H}^4},
        \end{align*}
        and together with (\ref{edoL2}) we get
        \begin{align}
            \partial_t (\|\Theta^\gamma_R\|_{L^2}+\|\Theta^\gamma_R\|_{\Dot{H}^4})\leq C( \|\Theta^\gamma_R\|_{H^4}^2+ \|\overline{\theta}^\gamma_R\|_{C^{6}}\|\Theta^\gamma_R\|_{H^4})+\|F\|_{H^4}.
        \end{align}
        Since $\Theta^\gamma_R \in C_tC_x^\infty$ and $\Theta^\gamma_R(x,0)=0$, we continue via a bootstrap argument, defining $0<T_R\leq T$ as the supremum time such that  $\|\Theta^\gamma_R\|_{H^4}\leq R^{-2-\gamma}\ln(R)$. Then, for $R>0$ big enough, using the bootstrap assumption to absorb the quadratic term in the linear one and (\ref{cotaFH4}),
        \begin{align*}
            \partial_t (\|\Theta^\gamma_R\|_{L^2}+\|\Theta^\gamma_R\|_{\Dot{H}^4}) \leq C\|\overline{\theta}^\gamma_R\|_{C^{6}}\|\Theta^\gamma_R\|_{H^4}+\|F_R\|_{H^4}\\ \leq C(\|\theta_1^\gamma\|_{C^6}+\|\theta_2^\gamma\|_{C^6})\|\Theta_R^\gamma\|_{H^4}+C D \|\theta_1^\gamma\|_{C^5} \|\theta_2^\gamma\|_{C^5} R^{-2-\gamma},
        \end{align*}
        for $t\in [0,T_R]$, so 
        \begin{align}
        \label{cotaR}
            \|\Theta^\gamma_R\|_{H^4}\leq CR^{-2-\gamma} t, \hspace{3mm} \forall t \in[0,T_R],
        \end{align}
        where $C>0$ is uniform on $R>4D$. Finally, for $R>max\{R_D,e^{CT}\}>0$ for the $C>0$ in (\ref{cotaR}) and by definition of supremum, $T_R=T$ and the result holds.
     \end{proof}
 \end{lma}
      \begin{lma}
          \label{fractionalsobolev}
          Let $f,g\in H^\beta\cap C^{\lfloor\beta\rfloor}$ with $\beta > 0$ such that
          \begin{itemize}
              \item $f$ and $g$ have disjoint supports with $d=\text{dist}(\text{supp}(f),\text{supp}(g))>2\pi\mathcal{K}_\beta^2+3$, where $\mathcal{K}_\beta=\frac{4^{\frac{\beta}{2}}\Gamma(\frac{2+\beta}{2})}{\pi|\Gamma(-\frac{\beta}{2})|}$,
              \item $\text{diam}(\text{supp}(g))\leq 1$ and
              \item $\|f\|_{H^{\lfloor\beta\rfloor}},\|g\|_{H^{\lfloor\beta\rfloor}}\leq 1$.
          \end{itemize} 
          Then
          \begin{align*}
             C\|f+g\|_{H^{\beta}}\geq \|f\|_{H^{\beta}} +\|g\|_{H^{\beta}}-C(\beta-\lfloor \beta \rfloor )d^{-\beta+\lfloor\beta\rfloor},
          \end{align*}
        where we use $\lfloor\cdot\rfloor$ for the floor function and $C>0$ just depend on $\beta$.
          \begin{proof}
          If $\beta\in \mathbb{N}$ the result is trivial.\\\\
          Assume now that $\beta\in(0,1)$. Let $A=\{x\in \mathbb{R}^2: dist(x,\text{supp}(f))\leq d/3\}$ and $B=\{x\in \mathbb{R}^2: dist(x,\text{supp}(g))\leq d/3\}$. Notice that $A\cap B=\emptyset$. Then, using Hölder's inequality and $(a+b)^2\leq 2(a^2+b^2)$,
          \begin{align*}
              \|\Lambda^\beta f\|_{L^2}^2= \mathcal{K}_\beta^2\int_{\mathbb{R}^2} \left(P.V.\int_{\mathbb{R}^2}\frac{f(x+h)-f(x)}{|h|^{2+\beta}}dh\right) ^2dx\\
              =\mathcal{K}_\beta^2\int_{A} \left(P.V.\int_{\mathbb{R}^2}\frac{f(x+h)-f(x)}{|h|^{2+\beta}}dh\right) ^2dx
              +\mathcal{K}_\beta^2\int_{B} \left(P.V.\int_{\mathbb{R}^2}\frac{f(x+h)}{|h|^{2+\beta}}dh\right) ^2dx\\
                +\mathcal{K}_\beta^2\int_{(A\cup B)^c} \left(P.V.\int_{\mathbb{R}^2}\frac{f(x+h)}{|h|^{2+\beta}}dh\right) ^2dx
                \end{align*}
                \begin{align*}
                \leq 
                2\mathcal{K}_\beta^2\int_{A} \left(P.V.\int_{\mathbb{R}^2}\frac{(f+g)(x+h)-f(x)}{|h|^{2+\beta}}dh\right) ^2dx
                +2\mathcal{K}_\beta^2\int_{A} \left(\int_{\mathbb{R}^2} 
                \frac{g(x+h)}{|h|^{2+\beta}}dh\right)^2dx\\
                +\mathcal{K}_\beta^2\int_{B} \left(\int_{2d/3\leq |h|\leq \infty}\frac{f(x+h)}{|h|^{2+\beta}}dh\right) ^2dx\\
                +2\mathcal{K}_\beta^2\int_{(A\cup B)^c} \left(P.V.\int_{\mathbb{R}^2}\frac{(f+g)(x+h)}{|h|^{2+\beta}}dh\right) ^2dx
                +2\mathcal{K}_\beta^2\int_{(A\cup B)^c} \left(\int_{\mathbb{R}^2}\frac{g(x+h)}{|h|^{2+\beta}}dh\right) ^2dx 
                 \end{align*}
                \begin{align*}
                \leq 
                 2\mathcal{K}_\beta^2\int_{A} \left(P.V.\int_{\mathbb{R}^2}\frac{(f+g)(x+h)-f(x)}{|h|^{2+\beta}}dh\right) ^2dx
                 +2\mathcal{K}_\beta^2|\text{supp}(g)|\int_{A} \int_{2d/3\leq |h|\leq \infty}\frac{g(x+h)^2}{|h|^{4+2\beta}}dh dx
                 \\
                 +\mathcal{K}_\beta^2 \int_B  \left(\int_{2d/3\leq|h|\leq \infty} \frac{1}{|h|^{4+2\beta}}dh
                 \right)
                 \left( \int_{\mathbb{R}^2} f(x+h)^2dh\right)dx\\
                 +2\mathcal{K}_\beta^2\int_{(A\cup B)^c} \left(P.V.\int_{\mathbb{R}^2}\frac{(f+g)(x+h)}{|h|^{2+\beta}}dh\right) ^2dx
                 +2\mathcal{K}_\beta^2|\text{supp}(g)|\int_{(A\cup B)^c} \int_{d/3\leq |h|\leq \infty} \frac{g(x+h)^2}{|h|^{4+2\beta}}dh dx
                  \end{align*}
                \begin{align*}
                 \leq 
                  2\mathcal{K}_\beta^2\int_{A} \left(P.V.\int_{\mathbb{R}^2}\frac{(f+g)(x+h)-f(x)}{|h|^{2+\beta}}dh\right) ^2dx
                  +2\mathcal{K}_\beta^2\int_{(A\cup B)^c} \left(P.V.\int_{\mathbb{R}^2}\frac{(f+g)(x+h)}{|h|^{2+\beta}}dh\right) ^2dx\\
                 +2\mathcal{K}_\beta^2 \frac{(2d/3)^{-2-2\beta}}{2+2\beta} \|g\|_{L^2}^2        
                 +\mathcal{K}_\beta^2 \pi (1+d/3)^2 (2d/3)^{-2-2\beta} \|f\|_{L^2}^2
                 +2\mathcal{K}_\beta^2 \frac{(d/3)^{-2-2\beta}}{2+2\beta} \|g\|_{L^2}^2
          \end{align*}
          \begin{align*}
              \leq 2\mathcal{K}_\beta^2 \int_{A} \left(P.V.\int_{\mathbb{R}^2}\frac{(f+g)(x+h)-f(x)}{|h|^{2+\beta}}dh\right) ^2dx
              +2\mathcal{K}_\beta^2 \int_{(A\cup B)^c} \left(P.V.\int_{\mathbb{R}^2}\frac{(f+g)(x+h)}{|h|^{2+\beta}}dh\right) ^2dx
              \\+C d^{-2\beta}.
               \end{align*}
               Then, since $ \|\Lambda^{\beta} (f+g)\|_{L^2}^2$ can be decomposed as
               \begin{align*}
                   \|\Lambda^{\beta} (f+g)\|_{L^2}^2=\mathcal{K}_\beta^2 \int_{A} \left(P.V.\int_{\mathbb{R}^2}\frac{(f+g)(x+h)-f(x)}{|h|^{2+\beta}}dh\right) ^2dx\\
               +\mathcal{K}_\beta^2\int_{B} \left(P.V.\int_{\mathbb{R}^2}\frac{(f+g)(x+h)-g(x)}{|h|^{2+\beta}}dh\right) ^2dx 
              +\mathcal{K}_\beta^2 \int_{(A\cup B)^c} \left(P.V.\int_{\mathbb{R}^2}\frac{(f+g)(x+h)}{|h|^{2+\beta}}dh\right) ^2dx,
              \end{align*}
               we have
          \begin{align*}
              \|\Lambda^{\beta} f\|_{L^2}^2\leq 2\|\Lambda^{\beta}(f+g)\|_{L^2}^2+Cd^{-2\beta}.
          \end{align*}
          Notice that along this last computation we have already bounded the quantities 
          \begin{align*}
              \int_B\left(\int_{\mathbb{R}^2} \frac{f(x+h)}{|h|^{2+\beta}}dh\right)^2dx, \hspace{3mm}
              \int_A\left(\int_{\mathbb{R}^2} \frac{g(x+h)}{|h|^{2+\beta}}dh\right)^2dx
              , \hspace{3mm}
              \int_{(A\cup B)^c}\left(\int_{\mathbb{R}^2} \frac{g(x+h)}{|h|^{2+\beta}}dh\right)^2dx,
          \end{align*}
         by $Cd^{-2\beta}$. Using this, 
         \begin{align*}
              \|\Lambda^\beta g\|_{L^2}^2= \mathcal{K}_\beta^2\int_{\mathbb{R}^2} \left(P.V.\int_{\mathbb{R}^2}\frac{g(x+h)-g(x)}{|h|^{2+\beta}}dh\right) ^2dx\\
              =\mathcal{K}_\beta^2\int_{B} \left(P.V.\int_{\mathbb{R}^2}\frac{g(x+h)-g(x)}{|h|^{2+\beta}}dh\right) ^2dx
              +\mathcal{K}_\beta^2\int_{A} \left(P.V.\int_{\mathbb{R}^2}\frac{g(x+h)}{|h|^{2+\beta}}dh\right) ^2dx\\
                +\mathcal{K}_\beta^2\int_{(A\cup B)^c} \left(P.V.\int_{\mathbb{R}^2}\frac{g(x+h)}{|h|^{2+\beta}}dh\right) ^2dx
                \end{align*}
                \begin{align*}
                \leq
                2\mathcal{K}_\beta^2\int_{B} \left(P.V.\int_{\mathbb{R}^2}\frac{(f+g)(x+h)-g(x)}{|h|^{2+\beta}}dh\right) ^2dx 
                +2\mathcal{K}_\beta^2\int_{B} \left(\int_{\mathbb{R}^2}\frac{f(x+h)}{|h|^{2+\beta}}dh\right) ^2dx
                +Cd^{-2\beta}
                 \\ \leq 2\|\Lambda^{\beta}(f+g)\|_{L^2}^2+Cd^{-2\beta},
                \end{align*}
                so we conclude 
                \begin{align*}
                    \|\Lambda^{\beta}f\|_{L^2}^2+
                    \|\Lambda^{\beta}g\|_{L^2}^2\leq 4\|\Lambda^{\beta}(f+g)\|_{L^2}^2+Cd^{-2\beta},
                \end{align*}
                and hence since the sum of the $L^2$ norms is equal to the $L^2$ norm of the sum of functions with disjoint support, there exist $C>0$ depending only on $\beta \in (0,1)$ such that
                \begin{align*}
                    C\|f+g\|_{H^{\beta}}\geq \|f\|_{H^{\beta}} +\|g\|_{H^{\beta}}-Cd^{-\beta}.
                \end{align*}
                To prove it for general $\beta \geq 1$, noticing that the classical derivative does not create any problem since it is a local  operator. It is enough to repeat the proof replacing $f$ and $g$ for $\partial^{\mathbf{k}}f$ and $\partial^{\mathbf{k}}g$ with every $|\mathbf{k}|=\lfloor\beta\rfloor$ and  $\beta$ for $\beta-\lfloor\beta\rfloor$. 
          \end{proof}
        \end{lma}
        \begin{lma}
            Let $f\in L^2(\mathbb{R}^2)\cap C^2(K)$ for any compact set $K\subset \mathbb{R}^2$. Then, for $\beta \in (0,2)$,
            \begin{align}
                \label{normascrecientes}
                \begin{split}
                \|f\|_{H^{\beta}(K)}\leq C(\|f\|_{L^{2}(\mathbb{
                R}^2)}+ \|f\|_{H^{2}(K+B_{1}(0))}),\\          \|f\|_{H^{\beta+1}(K)}\leq C(\|f\|_{H^{1}(\mathbb{
                R}^2)}+ \|f\|_{H^{3}(K+B_{1}(0))}).\\      
                \end{split}
            \end{align}
            Here $C>0$ depends just on $\beta$ and $K+B_{1}(0):=\{x+y\in \mathbb{R}^2:x\in K, y\in B_1(0)\}$.
            \begin{proof}
                Let $\beta \in (0,2)$. For $\bold{k}\in \mathbb{N}_0^2$, let $R_\bold{k} f$ be the $\bold{k}$-th reminder of the Taylor expansion of $f$. Then, 
                \begin{align*}
                    \|\Lambda^\beta f\|_{L^2(K)}^2=\mathcal{K}_\beta^2 \int_K \left( \int_{\mathbb{R}^2} \frac{f(x+h)-f(x)}{|h|^{2+\beta}}dh \right)^2dx
                    \\ \leq
                    2\mathcal{K}_\beta^2 \int_K \left( \int_{\{0\leq |h|\leq 1\}} \frac{f(x+h)-f(x)}{|h|^{2+\beta}}dh \right)^2dx
                    +2\mathcal{K}_\beta^2\int_K \left( \int_{\{1\leq |h|\leq \infty\}} \frac{f(x+h)-f(x)}{|h|^{2+\beta}}dh \right)^2dx
                    \\=
                     2\mathcal{K}_\beta^2\int_K \left( \int_{\{0\leq |h|\leq 1\}} \frac{f(x+h)-f(x)-\partial_{x_1}f(x)h_1-\partial_{x_2}f(x)h_2}{|h|^{2+\beta}}dh \right)^2dx
                    \\+2\mathcal{K}_\beta^2\int_K \left( \int_{\{1\leq |h|\leq \infty\}} \frac{f(x+h)-f(x)}{|h|^{2+\beta}}dh \right)^2dx
                    \\ \leq 
                    2\mathcal{K}_\beta^2\int_K \left( \int_{\{0\leq |h|\leq 1\}} \frac{|\sum_{|\mathbf{k}|=2 }R_\mathbf{k} f(x,h)| |h|^2}{|h|^{2+\beta}}dh \right)^2dx\\
                    +4\mathcal{K}_\beta^2\int_K \left( \int_{\{1\leq |h|\leq \infty\}} \frac{|f(x+h)|}{|h|^{2+\beta}}dh \right)^2dx  
                    +4\mathcal{K}_\beta^2\int_K \left( \int_{\{1\leq |h|\leq \infty\}} \frac{|f(x)|}{|h|^{2+\beta}}dh \right)^2dx  
                      \\ \leq 
                   2\mathcal{K}_\beta^2 \int_K \left( \int_{\{0\leq |h|\leq 1\}} \frac{|\sum_{|\mathbf{k}|=2 }R_\mathbf{k} f(x,h)| }{|h|^{\beta}}dh \right)^2dx\\
                    +4\mathcal{K}_\beta^2\left( \int_{\{1\leq |h|\}}|h|^{-2-\beta}dh\right)\int_K  \int_{\{1\leq |h|\}} \frac{|f(x+h)|^2}{|h|^{2+\beta}}dhdx  
                    +4\mathcal{K}_\beta^2
                    \left(\int_K |f(x)|^2 dx \right)
                    \left( \int_{\{1\leq |h|\}} \frac{1}{|h|^{2+\beta}}dh \right)^2 \\
                    \leq
                    2\mathcal{K}_\beta^2\left(\int_{\{0\leq |h|\leq 1\}}|h|^{-\beta} dh\right)\int_K \left( \int_{\{0\leq |h|\leq 1\}} |h|^{-\beta}\left(\sum_{|\mathbf{k}|=2 }|R_\mathbf{k} f(x,h)| \right)^2 dh \right)dx
                    +
                    C\mathcal{K}_\beta^2\|f\|^2_{L^2(\mathbb{R}^2)}
                    \end{align*}
                    \begin{align*}
                    \leq
                    C\mathcal{K}_\beta^2   \int_{\{0\leq |h|\leq 1\}}|h|^{-\beta} \sum_{|\mathbf{k}|=2 }\left(\int_K|R_\mathbf{k} f(x,h)|^2dx\right)  dh
                    +
                    C\mathcal{K}_\beta^2\|f\|^2_{L^2(\mathbb{R}^2)}\\
                    \leq
                    C\mathcal{K}_\beta^2  \int_{\{0\leq |h|\leq 1\}}|h|^{-\beta} \sum_{|\mathbf{k}|=2 }
                    \int_K \left(\left|
                    \int_{0}^1 (1-t) \partial^{\mathbf{k}} f (x+(1-t)h) dt
                    \right|^2dx \right) dh  
                    +
                    C\mathcal{K}_\beta^2\|f\|^2_{L^2(\mathbb{R}^2)}
                   \\
                    \leq 
                    C\mathcal{K}_\beta^2\|f\|_{\Dot{H}^2(K+B_1(0))}^2 \left(\int_{\{0\leq |h|\leq 1\}} 
                    |h|^{-\beta}
                     dh  \right)
                    +
                    C\mathcal{K}_\beta^2\|f\|^2_{L^2(\mathbb{R}^2)}\leq C(\|f\|_{L^{2}(\mathbb{
                R}^2)}+ \|f\|_{H^{2}(K+B_{1}(0))}) ,
                \end{align*}
                using Young's and Hölder's inequalities, the fact that the first order terms of the Taylor's expansion vanish when integrating in $h$ in symmetric balls and the integral expression of the reminder of the Taylor's expansion 
                \begin{align*}
                    R_{\bold{k}}f=\frac{|\bold{k}|}{\bold{k}!} \int_{0}^1 (1-t)^{|\bold{k}|-1} \partial^{\mathbf{k}} f (x+(1-t)h) dt.
                \end{align*}
                The second inequality follows from repeating the proof with $\partial_{x_i}f$ instead of $f$.
            \end{proof}
        \end{lma}
 \begin{proof}[Proof of Theorem \ref{nonexistence}]
      Now we are ready to proof non-existence of solutions. First of all notice that, by the strong ill-posedness Theorem \ref{teorema} and the  invariance of the gSQG equation under traslations, we know that for every $j\in \mathbb{N}$ there exist $\theta_{j,\gamma}$ a solution of $\gamma$-gSQG with initial condition $\theta_{j,\gamma}(\cdot,0)=T_{R_j}\overline{\theta}_{N_j,c_j,K_j}(\cdot,0)$ such that for $N_j,K_j>0$ big enough and $c_j>0$ small enough,
     \begin{enumerate}
         \item \label{punto1} $\|\theta_{j,\gamma}(\cdot,0)\|_{H^{\beta}}\leq 2^{-j} $,
         \item \label{punto2} $ supp (\theta_{j,\gamma}(\cdot,t))\subset B_{2^{-j}}(-R_j,0)$ for $t\in [0,t_0]$, where $R_j>0$ will be fixed later and
         \item \label{punto3} $\|\theta_{j,\gamma}(\cdot,t)\|_{H^{\beta}}\geq 2^{j}$ for $t\in [ 2^{-j},t_0]$.
         \item \label{punto4} Let $\beta_0=\frac{\max\{\lfloor \beta \rfloor, \frac{3}{2}+\gamma \}+\beta}{2}$ if $\beta\neq 1,2$ and $\beta_0=\frac{\frac{3}{2}+\gamma+\beta}{2}$ if $\beta=1,2$. For any $s \in [0,\beta_0]$, we can obtain $\|\theta_{j,\gamma}(\cdot,t)\|_{H^s}\leq 2^{-j}$ for $t\in[0,t_0]$. 
                  \item \label{punto5} $\theta_{j,\gamma}(\cdot,t)\in C^\infty_c$ for all $t\in [0,t_0]$.
     \end{enumerate}
         Points \ref{punto1}, \ref{punto3} and \ref{punto5} are true because of Theorem \ref{teorema}.\\\\
         The point \ref{punto4} holds by the construction of the solutions done in the proof of Theorem \ref{teorema}. Indeed,
         \begin{align*}
             \|\theta_{j,\gamma}(\cdot,t)\|_{H^s}\leq \|T_{R_j}\overline{\theta}_{N_j,c_j,K_j}(\cdot,t)\|_{H^s}+\|\theta_{j,\gamma}(\cdot,t)-T_{R_j}\overline{\theta}_{N_j,c_j,K_j}(\cdot,t)\|_{H^s} \\\leq 
             \|T_{R_j}\overline{\theta}_{N_j,c_j,K_j}(\cdot,t)\|_{H^s}+\|\theta_{j,\gamma}(\cdot,t)-T_{R_j}\overline{\theta}_{N_j,c_j,K_j}(\cdot,t)\|_{H^{\beta+\frac{1}{2}}}\\
\leq \|f_{1,c_j,K_j}\|_{H^{\beta}}+\|\overline{\theta}_{N_j,c_j,K_j}-f_{1,c_j,K_j}\|_{H^{s}}
+Ct\mathcal{X}(\gamma)N_j^{-\beta+\frac{3}{2}+\gamma}\\ \leq c_j+Ct\mathcal{X}(\gamma)N_j^{s-\beta}+Ct\mathcal{X}(\gamma)N_j^{-\beta+\frac{3}{2}+\gamma}
\\ \leq c_j+Ct\mathcal{X}(\gamma)N_j^{\beta_0-\beta}+Ct\mathcal{X}(\gamma)N_j^{-\beta+\frac{3}{2}+\gamma}\leq 2^{-j},
         \end{align*}
        using the norms of the pseudosolutions and (\ref{cotaTheta2}) for $N_j,K_j>0$ big enough and $c_j>0$ small enough.\\\\
   The point \ref{punto2} is also true by construction. Notice that Lemma \ref{lemaf1} joint with the particular choice of the supports of the radial functions involved in the definitions of the initial conditions ($supp(f_{1,c,K})(r)\subset [a_0,\frac{r_{c,K}}{2}]$ and $supp(f_{2,c,K})(r)\subset [\frac{2r_{c,K}}{3}, \frac{3r_{c,K}}{2}]$) imply that we can take the support of the initial values in a ball of radio as small as we want. Then by Theorem \ref{teorema} and choosing $(c_j)$ small enough, the point \ref{punto2} follows.\\\\
     We will consider the truncated initial conditions 
     \begin{align}
     \label{truncatedic}
         \sum_{j=0}^J \theta_{j,\gamma}(\cdot,0) =\sum_{j=0}^J T_{R_j}\overline{\theta}_{N_j,c_j,K_j}(\cdot,0),
     \end{align}
      and we will call $\theta_{tr,J,\gamma}(x,t)$ the corresponding local in time solution to $\gamma$-gSQG (\ref{gSQG}) with the initial conditions given by (\ref{truncatedic}).\\\\ Our aim is to prove that there exist a $\theta_{\infty,\gamma}$ (obtained by taking the limit $J\mapsto \infty$ in $\theta_{tr,J,\gamma}$) that solves $\gamma$-gSQG (\ref{gSQG}) with initial conditions $ \sum_{j=0}^\infty \theta_{j,\gamma}(\cdot,0)\in H^\beta $ and that leaves $H^\beta$ instantaneously. To do this, we define $R_0=0$, $R_{j+1}=D_j+R_j$ and $(D_j)_{j=0}^J\subset (1,\infty)$ a growing sequence such that $D_j>2R_j$ to be fixed. We also define the $J+1$-th approximation by 
      \begin{align}
      \label{solaprox}
\overline{\theta}_{tr,J+1,\gamma}:=\theta_{tr,J,\gamma}+\theta_{J+1,\gamma},
      \end{align}
      this is, (\ref{solaprox}) is the sum of the smooth compact solution to $\gamma$-gSQG (\ref{gSQG}) with initial conditions $\sum_{j=0}^J \theta_{j,\gamma}(\cdot,0) $ and the smooth compact solution to $\gamma$-gSQG (\ref{gSQG}) with initial conditions $\theta_{J+1,\gamma}(\cdot,0)=T_{R_{J+1}}\overline{\theta}_{N_{J+1},c_{J+1},K_{J+1}}(\cdot,0)$, this is, the translated solution of the one with initial conditions $\overline{\theta}_{N_{J+1},c_{J+1},K_{J+1}}(\cdot,0)$. Hence we can apply Lemma \ref{lemanonexistence}, obtaining
      \begin{align}
      \label{cauchy}
          \|\theta_{tr,J+1,\gamma}-\overline{\theta}_{tr,J+1,\gamma}\|_{H^4}\leq 2^{-J}, \text{ for } t\in[0,t_0],
      \end{align}
      by an inductive process, doing $D_j>1$ as big as needed and taking into account that the common period of existence \footnote{$\theta_{J,\gamma}\in H^4$ for $t\in [0,t_0]$ clearly because of point \ref{punto5}. Then, $\theta_{tr,J, \gamma}\in H^{4}$ also for $t\in [0,t_0]$ because of an inductive process. Indeed, $\theta_{tr,J,\gamma}(\cdot,0)=\overline{\theta}_{tr,J,\gamma}(\cdot,0)$ and using (\ref{cauchy}) and Picard–Lindelöf theorem we can extend the existence of the truncated solutions in the hole time interval.} of all the functions is $t\in [0,t_0]$. \\\\Let us fix a compact set $K\times [0,t_0]\subset \mathbb{R}^2\times [0,t_0]$. Then, for $J\geq J_0 \in \mathbb{N}$ big enough we have that $\theta_{J+1,\gamma}(x,t)=0$ for every $(x,t)\in K\times [0,t_0]$ since $D_j>1$, so (\ref{cauchy}) implies 
      \begin{align}
      \label{aproxK}
            \|\theta_{tr,J+1,\gamma}-\theta_{tr,J,\gamma}\|_{H^4(K)}\leq 2^{-J}, \text{ for } t\in[0,t_0], \hspace{2mm} J\geq J_0,
      \end{align}
      which implies that $\{\theta_{tr,J,\gamma}\}_{J=0}^\infty$ is a Cauchy sequence in $C([0,t_0],H^4(K))$, that is a Banach space, so there exist $\theta_{\infty,\gamma}$ a limit. Since the compact set $K$ was arbitrary and the limit must be unique for every $K$, the function $\theta_\infty(x,t)$ can be defined for every $(x,t) \in \mathbb{R}^2\times [0,t_0]$ and it is $C_tH_x^4$ in every compact set. Now, we would like to prove some convergence result also in the velocities. More specifically, we will proof convergence for the velocities in the space $C_tH^2(K)$ using (\ref{normascrecientes}). Thus,
      \begin{align*}
         \|v^\gamma(\theta_{\infty,\gamma})-v^\gamma(\theta_{tr,J,\gamma})\|_{C_tH^2(K)}\leq 
          \|\theta_{\infty,\gamma}-\theta_{tr,J,\gamma}\|_{C_tH^{2+\gamma}(K)}
       \\
       \leq 
       C\sum_{j\geq J} \left\|
       \theta_{tr,j+1,\gamma}-\theta_{tr,j,\gamma}
        \right\|_{C_tH^{4}(K+B_1(0))}+
        C\sum_{j\geq J} \left\|
       \theta_{tr,j+1,\gamma}-\theta_{tr,j,\gamma}
        \right\|_{C_tH^{1}}\\
        \leq C\sum_{j\geq J} 2^{-j}\rightarrow 0, \hspace{3mm} \text{when } J\longrightarrow \infty,
      \end{align*}
      using also the point \ref{punto4}. So we conclude that for any $\gamma\in (-1,1)$,
      \begin{align*}
      v^\gamma(\theta_{tr,J,\gamma})  \longrightarrow v^\gamma(\theta_{\infty,\gamma})\hspace{2mm} \text{when } J\rightarrow \infty, \text{ in } C_tH^2(K).
      \end{align*}
      Since $\theta_{tr,J,\gamma} \in C([0,t_0],H^4(K))$ is a solution of (\ref{gSQG}), we can integrate in time (\ref{gSQG}), obtaining for every $0\leq t_1< t_2\leq t_0$,
       \begin{align*}
           \theta_{tr,J,\gamma}(x,t_2)-\theta_{tr,J,\gamma}(x,t_1)=-\int_{t_1}^{t_2}v^\gamma(\theta_{tr,J,\gamma})(x,t)\cdot \nabla \theta_{tr,J,\gamma}(x,t) dt,
       \end{align*}
       and passing to the limit in $J\mapsto \infty$,
        \begin{align}
        \label{eclimite}
           \theta_{\infty,\gamma}(x,t_2)-\theta_{\infty,\gamma}(x,t_1)=-\int_{t_1}^{t_2}v^\gamma(\theta_{\infty,\gamma})(x,t)\cdot \nabla \theta_{\infty,\gamma}(x,t) dt.
       \end{align}
         Notice that, since the convergence of $(\theta_{tr,J,\gamma})_{J>0}$ occurs in $C([0,t_0],H^4(K))$ and the one of $(v^\gamma(\theta_{tr,J,\gamma}))_{J>0}$ in $C([0,t_0],H^{2}(K))$, then $(v^\gamma(\theta_{tr,J,\gamma})\cdot \nabla \theta_{tr,J,\gamma})_{J>0}$ converges, in particular, in $C([0,t_0],C^{1/2}(K))$
         by Sobolev embedding and hence the limit (\ref{eclimite}) make sense classically for any $(x,t_i)\in K$, $i=1,2$. Moreover, since $K$ is a compact set and the integrand of the RHS is continuous, we can bound it by some positive number $M_K$ and hence 
         \begin{align*}
             |\theta_{\infty,\gamma}(x,t_2)-\theta_{\infty,\gamma}(x,t_1)|\leq M_K |t_2-t_1|, 
         \end{align*}
         obtaining that the limit $\theta_{\infty,\gamma}$ is Lipschitz in time in $K$. Hence, now we can divide in (\ref{eclimite}) by $t_2-t_1$ at both sides and consider the limit $t_2\mapsto t_1$. In the LHS, the time derivative in $t_1$  appears and in the RHD, using the Lebesgue differentiation theorem, we obtain the integrand evaluated in $t_1$. Since the compact set $K$ was arbitrary, we conclude that $\theta_{\infty,\gamma}$ solves (\ref{gSQG}) pointwise in the classical sense.  
        \begin{rmk}
        \label{outsupport}
        Let $f\in H^\beta$ with  $supp(f)\subset B_R(0)=B_R$ and let $K=\overline{B_{2 R}(0)}$. Then:
        \begin{itemize}
        \item If $\beta \in \mathbb{N}$,  
        \begin{align*}            \|f\|_{H^\beta(\mathbb{R}^2\setminus K)}=0.
        \end{align*}
            \item Let $\beta=1+s$ with $s\in (0,1)$. Notice that, fixed $x\in \mathbb{R}^2\setminus K$, the $y\in \overline{B_R}$ that minimizes the distance between $x$ and $y$ is $y=R\frac{x}{|x|}$ and hence $|x-y|\geq |x-R\frac{x}{|x|}|=(1-\frac{R}{|x|})|x|\geq \frac{|x|}{2}$, so
        \begin{align*}    \|f\|^2_{\Dot{H}^\beta(\mathbb{R}^2\setminus K)}=\|\Lambda^\beta f\|^2_{L^2(\mathbb{R}^2\setminus K)}=C\int_{\mathbb{R}^2\setminus K} 
        \left(
        \int_{B_R} \frac{f(x)-f(y)}{|x-y|^{3+s}} dy
        \right)^2 dx \\
        =C\int_{\mathbb{R}^2\setminus K} 
        \left(
        \int_{B_R} \frac{-f(y)}{|x-y|^{3+s}} dy
        \right)^2 dx
        \leq 
        C\int_{\mathbb{R}^2\setminus K} 
        \left(
        \int_{B_R} \frac{|f(y)|}{|x|^{3+s}} dy
        \right)^2 dx\\
         \leq C \|f\|_{L^1}^2\int_{\mathbb{R}^\setminus K}|x|^{-6-2s}dx \leq C\|f\|_{L^2}^2 R^{2}R^{-4-2s} \leq C\|f\|_{L^2}^2 R^{-2-2s}.
        \end{align*}
        \item Let $\beta =2+s$ with $s\in (0,1)$. Again, repeating the previous steps,
        \begin{align*}   \|f\|^2_{\Dot{H}^\beta(\mathbb{R}^2\setminus K)}=\|\Delta \Lambda^s f\|_{L^2(\mathbb{R}^2\setminus K)}^2
        =C\int_{\mathbb{R}^2\setminus K} 
        \left(
        \int_{B_R} \frac{\Delta f(x)-\Delta f(y)}{|x-y|^{2+s}} dy
        \right)^2 dx
        \\=
        C\int_{\mathbb{R}^2\setminus K} 
        \left(
        \int_{B_R} \frac{-\Delta f(y)}{|x-y|^{2+s}} dy
        \right)^2 dx
        \leq 
         C\int_{\mathbb{R}^2\setminus K} 
        \left(
        \int_{B_R} \frac{|\Delta f(y)|}{|x|^{2+s}} dy
        \right)^2 dx\leq C\|\Delta f\|_{L^2}^2R^{-2s},
        \end{align*}
        \end{itemize}
        so taking into account that the $L^2$ norm is zero out of the support, we conclude $\|f\|_{H^\beta(\mathbb{R}^2\setminus K)}\leq C\|f\|_{H^{\lfloor \beta \rfloor}}R^{-(\beta-\lfloor \beta \rfloor)}$, where $\lfloor \cdot \rfloor$ denotes the floor function and $\beta \in [1,2+\gamma)\cap (\frac{3}{2}+\gamma,2+\gamma)$. As a consequence, 
        \begin{align*}
            \|f\|_{H^\beta(K)}\geq \|f\|_{H^\beta(\mathbb{R}^2)}- C\|f\|_{H^{\lfloor \beta \rfloor}}R^{-(\beta-\lfloor \beta \rfloor)}.
        \end{align*}
        \end{rmk}
      Now let us check that $\theta_{\infty,\gamma}$ loses regularity instantaneously. Let $t\in (0,t_0]$ and let $J_0\in \mathbb{N}$ be such that $0<2^{-J_0}<t\leq t_0$. Notice that $supp(\theta_{j})\subset B_{R_j+1}(0)$ for $t\in[0,t_0]$ and define $K_{j}=\overline{B_{ 2R_j+2}}(0)$. If $\beta=2$, then it is easy to see, using the locality of the $H^2$ norm and the point \ref{punto3},
      \begin{align*}
          \|\theta_{J_1+1,\gamma}(\cdot,t)\|_{H^\beta(K_{J_1+1})}=\|\theta_{J_1+1,\gamma}(\cdot,t)\|_{H^\beta(\mathbb{R}^2)}\rightarrow \infty, \hspace{3mm} \text{when }  J_1\rightarrow \infty.
      \end{align*}   
      When $\beta \neq 2$, using the Remark \ref{outsupport} and the points \ref{punto3} and \ref{punto4} written at the beginning of the proof, for any $J_1>J_0,$
      \begin{align*}
      \|\theta_{J_1+1,\gamma}(\cdot,t)\|_{H^\beta(K_{J_1+1})}\geq 
          \|\theta_{J_1+1,\gamma}\|_{H^\beta} -C\|\theta_{J_1+1,\gamma}\|_{H^{\lfloor \beta \rfloor}}(2R_j+2)^{-\beta+\lfloor \beta \rfloor}\\ \geq 
          2^{J_1+1}-C2^{-J_1-1}(2R_j+2)^{-\beta+\lfloor \beta \rfloor} \rightarrow \infty , \hspace{6mm} J_1 \rightarrow \infty,
          \end{align*}
      where $\lfloor\cdot\rfloor$ denote the floor function and where we have used in the last limit that $R_j$ goes to infinity.\\\\
      Now, using (\ref{cauchy}) in the second inequality, Lemma \ref{fractionalsobolev} in the third and the previous computation when taking limits,
      \begin{align*}
    \|\theta_{tr,J_1+1,\gamma}(\cdot,t)\|_{H^\beta(K_{J_1+1})}\geq \|\overline{\theta}_{tr,J_1+1,\gamma}\|_{H^\beta(K_{J_1+1})}-  \|\theta_{tr,J_1+1,\gamma}-\overline{\theta}_{tr,J_1+1,\gamma}\|_{H^\beta}
         \\ \geq          \|\overline{\theta}_{tr,J_1+1,\gamma}\|_{H^\beta(K_{J_1+1})}-  2^{-J_1} \\\geq 
\|\theta_{J_1+1,\gamma}\|_{H^\beta(K_{J_1+1})}+\|\theta_{tr,J_1,\gamma}\|_{H^\beta(K_{J_1+1})}-C(D_{J_1}-2)^{-(\beta-\lfloor \beta \rfloor)}-2^{-J_1}
\\ \geq 
\|\theta_{J_1+1,\gamma}\|_{H^\beta(K_{J_1+1})}-C(D_{J_1}-2)^{-(\beta-\lfloor \beta \rfloor)}-2^{-J_1}  \longrightarrow \infty,\hspace{6mm} J_1 \mapsto \infty,
      \end{align*}
      making $(D_j)$ tend to infinity.\\\\
      Now, let us proof that $\|\theta_{\infty,\gamma}(\cdot,t)\|_{H^\beta}=\infty$ for any $t>0$. Since for any compact set $K$, $\theta_{tr,j,\gamma}\rightarrow \theta_{\infty,\gamma}$ in $C([0,t_0], H^4(K))$, using (\ref{normascrecientes}) and the point \ref{punto3}, for any $J_0\in \mathbb{N}$ there exist a big enough $J_1>J_0$ such that $$\|\theta_{\infty,\gamma}(\cdot,t)-\theta_{tr,J_1,\gamma}(\cdot,t)\|_{H^\beta(K_{J_0})}\leq \|\theta_{\infty,\gamma}(\cdot,t)\|_{H^1}+\|\theta_{tr,J_1,\gamma}(\cdot,t)\|_{H^1}+\|\theta_{\infty,\gamma}(\cdot,t)-\theta_{tr,J_1,\gamma}(\cdot,t)\|_{H^4(K_{J_0}+B_ 1)}<3.$$
      Using this together with the inequalities  (\ref{normascrecientes}), (\ref{aproxK}) and the point \ref{punto3}, for any $J_0\in \mathbb{N}$ there exist a big enough $J_1>J_0+2$ such that
      \begin{align*}
          \|\theta_{\infty,\gamma}(\cdot,t)\|_{H^\beta}\geq \|\theta_{\infty,\gamma}(\cdot,t)\|_{H^\beta(K_{J_0})}\\
          \geq \|\theta_{tr,J_0,\gamma}(\cdot,t)\|_{H^\beta(K_{J_0})}  -\|\theta_{\infty,\gamma}(\cdot,t)-\theta_{tr,J_1,\gamma}(\cdot,t)\|_{H^\beta(K_{J_0})}
          -\sum_{j=J_0}^{J_1-1}\|\theta_{tr,j+1,\gamma}(\cdot,t)-\theta_{tr,j,\gamma}(\cdot,t)\|_{H^\beta(K_{J_0})} \\
          \geq \|\theta_{tr,J_0,\gamma}(\cdot,t)\|_{H^\beta(K_{J_0})} -3-\sum_{j=J_0}^{J_1-1}\|\theta_{tr,j+1,\gamma}(\cdot,t)-\theta_{tr,j,\gamma}(\cdot,t)\|_{H^4(K_{J_0}+B_1(0))}
          -2\sum_{j=J_0}^{J_1}\|\theta_{tr,j,\gamma}(\cdot,t)\|_{H^1}\\
          \geq \|\theta_{tr,J_0,\gamma}(\cdot,t)\|_{H^\beta(K_{J_0})} -C,
      \end{align*}
       and doing $J_0\rightarrow \infty$, by the previous computation we obtain $ \|\theta_{\infty,\gamma}(\cdot,t)\|_{H^\beta}=\infty$.
      \begin{rmk}
          Notice that (\ref{aproxK}) is written for some compact $K$ and for $J$ big enough, but indeed the only relation we need between $K$ and $J$ is that $\theta_{J+1}=0$ in $K$, so we can use (\ref{aproxK}) in the third inequality of the previous computation doing $D_j>3$.
      \end{rmk}
      Now, let us proof uniqueness. We assume that there exist two solutions $\theta_1,\theta_2\in L^\infty_tH^{\frac{3}{2}+\gamma}\cap C([0,t_0],C^2( K))$ for any compact set $K\subset \mathbb{R}^2$ and $\theta_1$ is the one that we constructed before. Notice that $\theta_1$ is in the claimed spaces. It is in $L^\infty_tH^{\frac{3}{2}+\gamma}$ basically because of the point \ref{punto3} and in $C([0,t_0],C^2( K))$ because we  obtain the convergence of the truncated solutions in $C([0,t_0],H^ 4( K))$, so by Sobolev embeddings we obtain that the limit has in particular spatial regularity $C^2$ in compact sets.\\ \\
Let us study uniqueness in the interval $[0,t_0]$. Let $\|v^\gamma_j(\theta_i)\|_{L^\infty}\leq v_{max}$ for $i,j=1,2$. We have that $supp(\theta_i)\subset \cup_{j\in \mathbb{N}}B_{t_0v_{max}+2^{-j}}(-R_j,0)$ for $i=1,2$, since $\theta_1$ and $\theta_2$, the considered solutions to the transport equation (\ref{gSQG}), are of class $C^1$ in spatial coordinates. We define 
\begin{align*}
    \theta^{j}_1(x,t)=\mathbbm{1}_{B_{t_0v_{max}+2^{-j}}(-R_j,0)}\theta_1(x,t),\\ 
     \theta^{j}_2(x,t)=\mathbbm{1}_{B_{t_0v_{max}+2^{-j}}(-R_j,0)}\theta_2(x,t),
\end{align*}
and we define $\Theta^j:=\theta^j_2-\theta^j_1$ and $\Theta:=\theta_2-\theta_1$.\\ \\
Notice that 
\begin{align*}
    \frac{\partial \Theta^j}{\partial t}= \frac{\partial \theta_2^j}{\partial t}-\frac{\partial \theta_1^j}{\partial t}\\ \\
    =v^\gamma_1(\theta_1)\frac{\partial \theta_1^j}{\partial x_1}+
    v^\gamma_2(\theta_1)\frac{\partial \theta_1^j}{\partial x_2}
    -v^\gamma_1(\theta_2)\frac{\partial \theta_2^j}{\partial x_1}
    -v^\gamma_2(\theta_2)\frac{\partial \theta_2^j}{\partial x_2}\\ \\
    = -v^\gamma_1(\Theta) \frac{\partial \Theta^j}{\partial x_1}
    -v^\gamma_1(\Theta) \frac{\partial \theta_1^j}{\partial x_1}
     -v^\gamma_1(\theta_1) \frac{\partial \Theta_1^j}{\partial x_1}\\
     -v^\gamma_2(\Theta) \frac{\partial \Theta^j}{\partial x_2}
    -v^\gamma_2(\Theta) \frac{\partial \theta_1^j}{\partial x_2}
     -v^\gamma_2(\theta_1) \frac{\partial \Theta_1^j}{\partial x_2}
          \end{align*}
     \begin{align*}
     =-v^\gamma_1(\Theta^j)\frac{\partial \theta_1^j}{\partial x_1}
      -v^\gamma_1(\Theta^j)\frac{\partial \Theta^j}{\partial x_1}
      -v^\gamma_2(\Theta^j)\frac{\partial \theta_1^j}{\partial x_2}
      -v^\gamma_2(\Theta^j)\frac{\partial \Theta^j}{\partial x_2}\\
       -v^\gamma_1(\theta_1^j)\frac{\partial \Theta^j}{\partial x_1}
       -v^\gamma_2(\theta_1^j)\frac{\partial \Theta^j}{\partial x_2}
       -v^\gamma_1(\Theta-\Theta^j)\frac{\partial \theta_1^j}{\partial x_1} 
       -v^\gamma_1(\Theta-\Theta^j)\frac{\partial \Theta^j}{\partial x_1}\\
       -v^\gamma_2(\Theta-\Theta^j)\frac{\partial \theta_1^j}{\partial x_2} 
       -v^\gamma_2(\Theta-\Theta^j)\frac{\partial \Theta^j}{\partial x_2}\\
       -v^\gamma_1(\theta_1-\theta_1^j)\frac{\partial \Theta^j}{\partial x_1}
       -v^\gamma_2(\theta_1-\theta_1^j)\frac{\partial \Theta^j}{\partial x_2}.
\end{align*}
By an standard energy method,
\begin{align*}
    \frac{1}{2}\frac{\partial \|\Theta^j\|^2_{L^2}}{\partial t}=\int \Theta^j \left(
    -v^\gamma_1(\Theta^j)\frac{\partial \theta_1^j}{\partial x_1}
      -v^\gamma_2(\Theta^j)\frac{\partial \theta_1^j}{\partial x_2}  -v^\gamma_1(\Theta-\Theta^j)\frac{\partial \theta_1^j}{\partial x_1} -v^\gamma_2(\Theta-\Theta^j)\frac{\partial \theta_1^j}{\partial x_2} 
       \right)dx,
\end{align*}
so 
\begin{align*}
   \frac{\partial \|\Theta^j\|_{L^2}}{\partial t} 
      \leq C\|\theta_1\|_{C^2(\overline{B}_{t_0 v_{max}+2^{-j}})}\|\Theta^j\|_{L^2}
     +C\|\theta_1\|_{C^1(\overline{B}_{t_0 v_{max}+2^{-j}})} \frac{\|\Theta\|_{L^2}}{(D_j-C)^{1+\gamma}}.
\end{align*}
For positive $\gamma$, we used that $v^\gamma$ is odd and the Lemma \ref{leibnizv} as we did other times to bound the first part of the sum. For the second sum, we take into account firstly that $v^\gamma_i(\Theta-\Theta^j)$ is multiplied by a function with support in $B_{t_0v_{max}+2^{-j}}(-R_j,0)$, so we just have to bound 
\begin{align*}
    \left(
    \int_{supp(\Theta_j)}
    \left|
    [(\Theta-\Theta^j) \ast K^\gamma](x)
    \right|^2
    dx
    \right)^{\frac{1}{2}}
    \\ \leq
    \|\Theta-\Theta^j\|_{L^2} 
    \left(
       \int_{supp(\Theta_j)}
   (\int
   \frac{\mathbbm{1}_{supp(\Theta-\Theta^j)}(y)}{|x-y|^{4+2\gamma}}dy)
    dx
    \right)^{\frac{1}{2}}
    \\ \leq 
    C \frac{\pi(1+2t_0v_{max}) \|\Theta\|_{L^2}}{(D_j-2t_0v_{max}-2)^{1+\gamma}} \leq  C \frac{\|\Theta\|_{L^2}}{ (D_j-C)^{1+\gamma}},
\end{align*}
using that $(D_j)$ is an increasing sequence big enough.\\\\
Hence, for all time $t\in [0,t_0]$ and taking a sequence $(D_j)$ big enough we get 
\begin{align*}
    \|\Theta^j\|_{L^2}\leq \frac{\epsilon \|\Theta\|_{L^2}}{(D_j-C)^{\frac{1+\gamma}{2}}},
\end{align*}
compensating the possible growth of $\|\theta_1\|_{C^2(\overline{B}_{t_0 v_{max}+2^{-j}})}$ with the factor $(D_j-C)^{-\frac{1+\gamma}{2}}$. Adding over all $j$ and taking $\epsilon$ small we have, taking again $(D_j)$ big, 
\begin{align*}
    \|\Theta\|_{L^2}\leq \frac{  \|\Theta\|_{L^2}}{2},
\end{align*}
and thus $  \|\Theta\|_{L^2}=0$ in $[0,t_0]$. 
 \end{proof} 
 \section*{Acknowledgments}
 This work is supported in part by the Spanish Ministry of Science and Innovation, through the “SeveroOchoa Programme for Centres of Excellence in R$\&$D (CEX2019-000904-S $\&$ CEX2023-001347-S)” and PID2023-152878NB-I00. We were also partially supported by the ERC Advanced Grant 788250, and by the SNF grant FLUTURA: Fluids, Turbulence, Advection No. 212573.


\end{document}